\documentclass{amsart}

\usepackage{url}
\usepackage[T1]{fontenc}
\usepackage[latin1]{inputenc}
\usepackage{amsmath,amsthm,amstext,amsfonts,amssymb,amscd,amsxtra}
\usepackage{bm}
\usepackage{enumerate}
\usepackage{eucal}
\usepackage{mathrsfs,mathtools}
\usepackage{stmaryrd}
\usepackage[all]{xy}
\usepackage{color}

\allowdisplaybreaks

\theoremstyle{plain}
\newtheorem{theorem}{Theorem}[section]

\newtheorem*{theoremA}{Theorem~A}
\newtheorem*{theoremB}{Theorem~B}
\newtheorem*{theoremC}{Theorem~C}
\newtheorem{lemma}[theorem]{Lemma}

\newtheorem{proposition}[theorem]{Proposition}
\newtheorem{proposition-definition}[theorem]{Proposition-Definition}
\newtheorem{corollary}[theorem]{Corollary}
\newtheorem{claim}[theorem]{Claim}

\theoremstyle{definition}
\newtheorem{definition}{Definition}[section]
\newtheorem{example}{Example}[section]

\theoremstyle{remark}
\newtheorem{remark}[theorem]{Remark}

\newcommand{\NN}{\mathbb{N}}
\newcommand{\ZZ}{\mathbb{Z}}
\newcommand{\QQ}{\mathbb{Q}}
\newcommand{\RR}{\mathbb{R}}
\newcommand{\CC}{\mathbb{C}}
\newcommand{\FF}{\mathbb{F}}

\newcommand{\PP}{\mathbb{P}}

\makeatletter
 
 \@addtoreset{equation}{section}
\makeatother

\makeatletter
\newcommand{\esssup}{\mathop{\operator@font ess.sup}\displaylimits}
\newcommand{\essinf}{\mathop{\operator@font ess.inf}\displaylimits}
\makeatother

\makeatletter 
\let\@@pmod\pmod 
\DeclareRobustCommand{\pmod}{\@ifstar\@pmods\@@pmod} 
\def\@pmods#1{\mkern4mu({\operator@font mod}\mkern 6mu#1)} 
\makeatother

\makeatletter
\DeclareRobustCommand\widecheck[1]{{\mathpalette\@widecheck{#1}}}
\def\@widecheck#1#2{%
    \setbox\z@\hbox{\m@th$#1#2$}%
    \setbox\tw@\hbox{\m@th$#1%
       \widehat{%
          \vrule\@width\z@\@height\ht\z@
          \vrule\@height\z@\@width\wd\z@}$}%
    \dp\tw@-\ht\z@
    \@tempdima\ht\z@ \advance\@tempdima2\ht\tw@ \divide\@tempdima\thr@@
    \setbox\tw@\hbox{%
       \raise\@tempdima\hbox{\scalebox{1}[-1]{\lower\@tempdima\box
\tw@}}}%
    {\ooalign{\box\tw@ \cr \box\z@}}}
\makeatother

\renewcommand{\div}{\mathop{\mathrm{div}}\nolimits}
\newcommand{\Image}{\mathop{\mathrm{Image}}\nolimits}
\renewcommand{\leq}{\leqslant}
\renewcommand{\geq}{\geqslant}

\DeclareMathOperator{\trdeg}{tr.\! deg}
\DeclareMathOperator{\depth}{depth}
\DeclareMathOperator{\height}{ht}
\DeclareMathOperator{\Frac}{Frac}
\DeclareMathOperator{\Ker}{Ker}
\DeclareMathOperator{\Coker}{Coker}
\DeclareMathOperator{\aHzf}{\widehat{\Gamma}^{\rm f}}

\DeclareMathOperator{\aHzst}{\widehat{\Gamma}^{\rm ss}}
\DeclareMathOperator{\aHzsm}{\widehat{\Gamma}^{\rm s}}

\DeclareMathOperator{\Sym}{Sym}
\DeclareMathOperator{\Spec}{Spec}

\DeclareMathOperator{\Supp}{Supp}

\DeclareMathOperator{\Ass}{Ass}
\DeclareMathOperator{\Aff}{Aff}
\DeclareMathOperator{\Rat}{Rat}
\DeclareMathOperator{\Exc}{Ex}

\DeclareMathOperator{\SHom}{\mathcal{H}om}

\DeclareMathOperator{\Bs}{Bs}
\DeclareMathOperator{\aBs}{\widehat{Bs}^{\rm ss}}

\DeclareMathOperator{\aPic}{\widehat{Pic}}

\DeclareMathOperator{\aInt}{\widehat{Int}}

\DeclareMathOperator{\Conv}{Conv}

\DeclareMathOperator{\adeg}{\widehat{deg}}

\DeclareMathOperator{\ord}{ord}

\DeclareMathOperator{\vol}{vol}
\DeclareMathOperator{\avol}{\widehat{vol}}

\DeclareMathOperator{\rk}{rk}

\newcommand{\Hz}{\mathit{H}^0}

\newcommand{\id}{{\rm id}}
\newcommand{\pr}{{\rm pr}}

\newcommand{\aN}{\widehat{\mathbf{N}}}
\newcommand{\aDelta}{\widehat{\Delta}}
\newcommand{\aTheta}{\widehat{\Theta}}

\newcommand{\Red}{{\rm Red}}

\newcommand{\CL}{{\it CL}}

\newcommand{\sbullet}{{\scriptscriptstyle\bullet}}

\newcommand{\aPhi}{\widehat{\Phi}}
\newcommand{\vAff}{\overrightarrow{\Aff}}

\newcommand{\aS}{\widehat{S}}

\newcommand{\akappa}{\widehat{\kappa}}
\newcommand{\SBs}{\mathbf{B}}
\newcommand{\aSBs}{\widehat{\mathbf{B}}^{\rm ss}}
\newcommand{\aSBssm}{\widehat{\mathbf{B}}^{\rm s}}
\newcommand{\aSBsn}{\widehat{\mathbf{B}}}
\newcommand{\aBsp}{\widehat{\mathbf{B}}_+}
\newcommand{\Bsp}{\mathbf{B}_+}

\def\aSpan#1#2{\langle#2\rangle_{#1}}
\def\Hzq#1{\mathit{H}^0_{#1}}

\def\aHzstq#1{\widehat{\Gamma}_{#1}^{\rm ss}}
\def\aHzsmq#1{\widehat{\Gamma}_{#1}^{\rm s}}
\def\aHzsq#1#2{\widehat{\Gamma}_{#2}^{#1}}
\def\aSBss#1{\widehat{\mathbf{B}}^{#1}}
\def\aBss#1{\widehat{\Bs}^{#1}}
\def\aHzs#1{\widehat{\Gamma}^{#1}}
\def\CLq#1{\widehat{\it CL}_{#1}}
\def\akappaq#1{\widehat{\kappa}_{#1}}
\def\aRq#1{\widehat{R}_{#1}}
\def\volq#1{\vol_{#1}}
\def\avolq#1{\widehat{\vol}_{#1}}
\def\aeq#1{\widehat{e}_{#1}}
\def\avolsmq#1{\widehat{\vol}_{#1}'}

\title[Restricted volumes and base loci]{Remarks on the arithmetic restricted volumes and the arithmetic base loci}
\author{Hideaki Ikoma}
\thanks{Communicated by S.\ Mochizuki. Revised December 12, 2015, June 11, 2016.}
\thanks{This research is supported by JSPS KAKENHI 25$\cdot$1895.}
\address{Department of Mathematics, Faculty of Science, Kyoto University, Kitashirakawa Oiwake-cho, Sakyo-ku, Kyoto, 606-8502, Japan}
\email{ikoma@math.kyoto-u.ac.jp}
\begin{document}

\begin{abstract}
In this paper, we collect some fundamental properties of the arithmetic restricted volumes (or the arithmetic multiplicities) of the adelically metrized line bundles.
The arithmetic restricted volume has the concavity property and characterizes the arithmetic augmented base locus as the null locus.
We also show a generalized Fujita approximation for the arithmetic restricted volumes.
\end{abstract}

\maketitle
\tableofcontents

\section{Introduction}

Let $K$ be a number field.
We denote the set of all the finite places of $K$ by $M_K^{\rm f}$, and set $M_K:=M_K^{\rm f}\cup\{\infty\}$.
Let $X$ be a projective variety over $K$.
According to \cite{ZhangAdelic,Moriwaki13}, we consider an \emph{adelically metrized line bundle} $\overline{L}=\left(L,(|\cdot|_v^{\overline{L}})_{v\in M_K}\right)$ on $X$, which is defined as a pair of a line bundle $L$ on $X$ and an adelic metric $(|\cdot|_v^{\overline{L}})_{v\in M_K}$ on $L$ (see \S\ref{sec:arithaugbs} for detail).
For each place $v\in M_K$, we denote by $\|\cdot\|_{v,\sup}^{\overline{L}}$ the supremum norm defined by $|\cdot|_v^{\overline{L}}$.
Let
\[
 \aHzf(\overline{L}):=\left\{s\in\Hz(L)\,:\,\text{$\|s\|_{v,\sup}^{\overline{L}}\leq 1$, $\forall v\in M_K^{\rm f}$}\right\}
\]
and
\[
 \aHzst(\overline{L}):=\left\{s\in\aHzf(\overline{L})\,:\,\|s\|_{\infty,\sup}^{\overline{L}}<1\right\}.
\]
We call an element in $\aHzst(\overline{L})$ a \emph{strictly small section of $\overline{L}$}, and set
\[
 \aSBs(\overline{L}):=\left\{x\in X\,:\,\text{$s(x)=0$, $\forall s\in\aHzst(m\overline{L})$, $\forall m\geq 1$}\right\}.
\]
An adelically metrized $\QQ$-line bundle $\overline{L}$ is said to be \emph{weakly ample}, or \emph{w-ample} for short, if $L$ is ample and $\aSpan{K}{\aHzst(m\overline{L})}=\Hz(mL)$ for every sufficiently large $m$ such that $mL$ is a line bundle (see the beginning of \S\ref{sec:adelicallymetrizedlbd} for the notation $\aSpan{K}{\,\cdot\,}$).
Let $Y$ be a closed subvariety of $X$.
We say that $\overline{L}$ is \emph{$Y$-big} if there exist a positive integer $m$, a w-ample adelically metrized line bundle $\overline{A}$, and an $s\in\aHzst(m\overline{L}-\overline{A})$ such that $s|_Y$ is non-zero.

The purpose of this paper is twofold.
First, we give an elementary proof to Theorem~A below, which was first proved by Moriwaki \cite[Theorem~A and Corollary~B]{MoriwakiFree} by using Zhang's technique \cite{ZhangVar}.
(More precisely, Moriwaki \cite{MoriwakiFree} treated only continuous Hermitian line bundles but it is easy to generalize it to the following form).
Following Moriwaki's suggestion, we call this result a ``Zhang--Moriwaki theorem''.
The proof given in this paper is very simple and based on some elementary properties of the base loci.
We do not need neither induction on dimension nor estimates of the last minima.

\begin{theoremA}[Theorem~\ref{thm:reprove}]
Let $X$ be a projective variety over $K$, and let $\overline{L}$ be an adelically metrized line bundle on $X$ such that $L$ is semiample.
Then the following are equivalent.
\begin{enumerate}
\item[\textup{(a)}] $\aSBs(\overline{L})=\emptyset$.
\item[\textup{(b)}] $\Hz(mL)=\aSpan{K}{\aHzst(m\overline{L})}$ for every $m\gg 1$.
\end{enumerate}
\end{theoremA}

The celebrated arithmetic Nakai-Moishezon criterion for the arithmetic ampleness was first proved by Zhang \cite{ZhangVar, ZhangAdelic}, and later was slightly generalized by Moriwaki \cite{MoriwakiFree} by using the same technique.
Theorem~A itself is not so powerful as recovering Zhang's criterion (\cite[Corollary~(4.8)]{ZhangVar}, \cite[Theorem~(1.8)]{ZhangAdelic}, \cite[Theorem~4.2]{MoriwakiFree}).

Next, we establish some fundamental properties of arithmetic restricted volumes (or arithmetic multiplicities) of adelically metrized line bundles along a closed subvariety.
A \emph{CL-subset} of a free $\ZZ$-module of finite rank is a subset written as an intersection of a $\ZZ$-submodule and a convex subset (see \cite[\S 1.2]{MoriwakiEst} for detail, where such a subset is referred to as a \emph{convex lattice}).
Given an adelically metrized line bundle $\overline{L}$ and a closed subvariety $Y$ of $X$, we define $\CLq{X|Y}(\overline{L})$ as the smallest CL-subset of $\aHzf(\overline{L}|_Y)$ containing the image of $\aHzst(\overline{L})$.
Denote by $\NN$ the semigroup of all the positive integers and set $\aN_{X|Y}(\overline{L}):=\left\{m\in\NN\,:\,\CLq{X|Y}(m\overline{L})\neq\{0\}\right\}$.
We define the \emph{arithmetic restricted volume of $\overline{L}$ along $Y$} as
\[
 \avolq{X|Y}(\overline{L}):=\limsup_{m\to+\infty}\frac{\log\sharp\CLq{X|Y}(m\overline{L})}{m^{\dim Y+1}/(\dim Y+1)!},
\]
and the \emph{arithmetic augmented base locus} of $\overline{L}$ as
\[
 \aBsp(\overline{L}):=\bigcap_{\text{$\overline{A}$: w-ample}}\aSBs(\overline{L}-\overline{A}),
\]
where $\overline{A}$ runs over all the w-ample adelically metrized $\QQ$-line bundles on $X$.
In the literature, we can find several other definitions of $\aBsp(\overline{L})$.
For example, in \cite[\S 3]{MoriwakiEst}, Moriwaki defined it by using the ``small sections'' and, in \cite[\S 4]{ChenSesh}, Chen treated the sections having normalized Arakelov degrees not less than zero.
In important cases, these three definitions all coincide (Remark~\ref{rem:aaugbscompare}).
The definitions given above have some desirable properties.
For example, if $X$ is normal, then
\[
 \aBsp(\overline{L})=\bigcup_{\substack{Z\subset X, \\ \avolq{X|Z}(\overline{L})=0}}Z
\]
(Corollary~\ref{cor:propertiesarestvol}).
By the theory of Okounkov bodies \cite{Kaveh_Khovanskii12,BoucksomBourbaki}, if $Y\not\subset\SBs(L)$, we have the following limit called the \emph{multiplicity of $L$ along $Y$},
\[
 e_{X|Y}(L):=\lim_{\substack{m\in\aN_{X|Y}(\overline{L}) \\ m\to+\infty}}\frac{\dim_{\Hz(\mathcal{O}_Y)}\aSpan{\Hz(\mathcal{O}_Y)}{\Image\left(\Hz(mL)\to\Hz(mL|_Y)\right)}}{m^{\kappa_{X|Y}(L)}/\kappa_{X|Y}(L)!},
\]
where $\kappa_{X|Y}(L)$ is defined as
\[
 \kappa_{X|Y}(L):=\trdeg_{\Hz(\mathcal{O}_Y)}\left(\bigoplus_{m\geq 0}\aSpan{\Hz(\mathcal{O}_Y)}{\Image\left(\Hz(mL)\to\Hz(mL|_Y)\right)}\right)-1.
\]
The following is an arithmetic analogue of this limit, which we call the \emph{arithmetic multiplicity of $\overline{L}$ along $Y$} and denote by $\aeq{X|Y}(\overline{L})$.

\begin{theoremB}[Theorem~\ref{thm:convergence} and Lemma~\ref{lem:kappaqequal}]
Let $X$ be a projective variety over a number field, let $Y$ be a closed subvariety of $X$, and let $\overline{L}$ be an adelically metrized line bundle on $X$.
If $Y\not\subset\aSBs(\overline{L})$, then the sequence
\[
 \left(\frac{\log\sharp\CLq{X|Y}(m\overline{L})}{m^{\kappa_{X|Y}(L)+1}/(\kappa_{X|Y}(L)+1)!}\right)_{m\in\aN_{X|Y}(\overline{L})}
\]
converges to a positive real number.
\end{theoremB}

As an application, we show a generalized Fujita approximation for the arithmetic restricted volumes, which can be viewed as an arithmetic analogue of \cite[Theorem~2.13]{Ein_Laz_Mus_Nak_Pop06}.
Let $X$ be a normal projective variety over a number field, let $\overline{L}$ be an adelically metrized $\QQ$-line bundle on $X$, and let $Z$ be a closed subvariety of $X$.
A \emph{$Z$-compatible approximation for $\overline{L}$} is a pair $(\mu:X'\to X,\overline{M})$ consisting of a projective birational morphism $\mu:X'\to X$ and a nef adelically metrized $\QQ$-line bundle $\overline{M}$ on $X'$ having the following properties.
\begin{enumerate}
\item[\textup{(a)}] $X'$ is smooth and $\mu$ is isomorphic around the generic point of $Z$.
\item[\textup{(b)}] Denote the strict transform of $Z$ via $\mu$ by $\mu_*^{-1}(Z)$.
Then $\overline{M}$ is $\mu_*^{-1}(Z)$-big and $\mu^*\overline{L}-\overline{M}$ is a $\mu_*^{-1}(Z)$-pseudoeffective adelically metrized $\QQ$-line bundle.
\end{enumerate}
We denote the set of all the $Z$-compatible approximations for $\overline{L}$ by $\aTheta_Z(\overline{L})$.

\begin{theoremC}[Theorem~\ref{thm:genFujita}]
Let $X$ be a normal projective variety over a number field, let $Z$ be a closed subvariety of $X$, and let $\overline{L}$ be an adelically metrized $\QQ$-line bundle on $X$.
If $\overline{L}$ is $Z$-big, then for every closed subvariety $Y$ containing $Z$
\[
 \avolq{X|Y}(\overline{L})=\sup_{(\mu,\overline{M})\in\aTheta_Z(\overline{L})}\adeg\left((\overline{M}|_{\mu_*^{-1}(Y)})^{\cdot(\dim Y+1)}\right).
\]
\end{theoremC}

This paper is organized as follows: we give a definition and properties of the adelically metrized line bundles in section \ref{sec:adelicallymetrizedlbd} (Definition~\ref{defn:adelicallymetrizedlinebdl}) and of the augmented base loci of general graded linear series in section \ref{sec:prelim} (Definition~\ref{defn:Chens}).
As an interlude, we give the proof of Theorem~A (Theorem~\ref{thm:reprove}) in section \ref{sec:reprove}, which is independent of the previous sections except Lemma~\ref{lem:Kodairatype}.
In section \ref{sec:arithaugbs}, we give a definition of the arithmetic augmented base locus of an adelically metrized line bundle (Definition~\ref{defn:aaugbs}) and characterize it as the exceptional loci of the Kodaira maps (Theorem~\ref{thm:charaaugbs}).
Sections \ref{sec:YuansEst} and \ref{sec:arithrestvol} are devoted to giving a definition of the arithmetic restricted volume of an adelically metrized line bundle along a closed subvariety, and a proof of Theorem~B (Theorem~\ref{thm:convergence}).
The arguments here are based on Yuan's idea \cite{Yuan09} and the general theory of Okounkov bodies \cite{Kaveh_Khovanskii12,BoucksomBourbaki}.
In section \ref{sec:genFujita}, we give a proof of the generalized Fujita approximation for arithmetic restricted volumes (Theorem~\ref{thm:genFujita}).

\section{Preliminaries}\label{sec:adelicallymetrizedlbd}

Let $M$ be a module over a ring $R$ and let $\Gamma$ be a subset of $M$.
We denote by $\langle\Gamma\rangle_R$ the $R$-submodule of $M$ generated by $\Gamma$.
Let $K$ be a number field and let $O_K$ be the ring of integers.
Denote by $M_K^{\rm f}$ the set of all finite places of $K$ and set $M_K:=M_K^{\rm f}\cup\{\infty\}$.
For each $P\in M_K^{\rm f}$, we denote the $P$-adic completion of $K$ by $K_P$, a uniformizer of $K_P$ by $\varpi_P$, and set
\begin{equation}\label{eqn:KPnorm}
 |a|_P:=\begin{cases} \sharp(O_K/P)^{-\ord_P(a)} & \text{if $a\in K_P^{\times}$,} \\ 0 & \text{if $a=0$} \end{cases}
\end{equation}
for $a\in K_P$.
We denote the valuation ring of $K_P$ by $O_{K_P}$, and the residue field of $K_P$ by $\widetilde{K}_P$.
For $v=\infty$, we set $K_{\infty}:=\CC$ and denote by $|\cdot|_{\infty}$ the absolute value on $\CC$.

\paragraph{Adelically normed linear series}

Let $V$ be a finite-dimensional $K$-vector space.
We set for $v\in M_K^{\rm f}$ $V_{K_v}:=V\otimes_KK_v$, and for $v=\infty$ $V_{K_v}:=V\otimes_{\QQ}\CC$.
An \emph{adelic norm} on $V$ is a collection of norms $(\|\cdot\|_v)_{v\in M_K}$ such that,
\begin{enumerate}
\item[(a)] if $v\in M_K^{\rm f}$, $\|\cdot\|_v$ is a non-Archimedean $(K_v,|\cdot|_v)$-norm on $V_{K_v}$,
\item[(b)] if $v=\infty$, $\|\cdot\|_{\infty}$ is a $(K_{\infty},|\cdot|_{\infty})$-norm on $V_{K_{\infty}}$, and
\item[(c)] for each $s\in V$, $\|s\|_v\leq 1$ for all but finitely many $v\in M_K^{\rm f}$.
\end{enumerate}
An \emph{adelically normed $K$-vector space} is a pair $\overline{V}=\left(V,(\|\cdot\|_v^{\overline{V}})_{v\in M_K}\right)$ of a finite-dimensional $K$-vector space $V$ and an adelic norm $(\|\cdot\|_v^{\overline{V}})_{v\in M_K}$ on $V$ such that the following equivalent conditions are satisfied (see \cite[Proposition~2.4]{Boucksom_Chen}).
\begin{enumerate}
\item[(a)] $\aHzf(\overline{V}):=\left\{s\in V\,:\,\text{$\|s\|_v^{\overline{V}}\leq 1$, $\forall v\in M_K^{\rm f}$}\right\}$ is a finitely generated $O_K$-module that spans $V$ over $K$.
\item[(b)] $\aHzsm(\overline{V}):=\left\{s\in\aHzf(\overline{V})\,:\,\|s\|_{\infty}^{\overline{V}}\leq 1\right\}$ is a finite set.
\item[(c)] $\aHzst(\overline{V}):=\left\{s\in\aHzf(\overline{V})\,:\,\|s\|_{\infty}^{\overline{V}}<1\right\}$ is a finite set.
\end{enumerate}
Note that the above definition of an adelically normed $K$-vector space is strictly weaker than the classic one as given in \cite[(1.6)]{ZhangAdelic}, \cite[Definition~3.1]{Gaudron07}.
We do not require the existence of an integral model of $V$ that induces all but finitely many $\|\cdot\|_v$, which corresponds to the condition (b) of \cite[(1.6)]{ZhangAdelic} and to the condition 1) of \cite[Definition~3.1]{Gaudron07}.

Let $X$ be a projective variety over $K$, that is, a projective, reduced, and irreducible scheme over $K$.
Let $L$ be a line bundle on $X$.
A \emph{$K$-linear series belonging to $L$} is a $K$-subspace $V$ of $\Hz(L)$.
The base locus of $V$ is defined as
\[
 \Bs V:=\bigcap_{s\in V}\left\{x\in X\,:\,s(x)=0\right\}.
\]
We consider endowing $V$ with an adelic norm $(\|\cdot\|_v)_{v\in M_K}$ on $V$ such that $\overline{V}:=\left(V,(\|\cdot\|_v)_{v\in M_K}\right)$ is an adelically normed $K$-vector space in the above sense.
An \emph{adelically normed graded $K$-linear series $\overline{V}_{\sbullet}$ belonging to $L$} is a graded $K$-linear series $V_{\sbullet}=\bigoplus_{m\geq 0}V_m$ endowed for each $m\geq 0$ with an adelic norm $(\|\cdot\|_v^{\overline{V}_m})_{v\in M_K}$ on $V_m$ in such a way that
\begin{equation}
 \|s\otimes t\|_v^{\overline{V}_{m+n}}\leq\|s\|_v^{\overline{V}_m}\cdot\|t\|_v^{\overline{V}_n}
\end{equation}
for every $v\in M_K$, $s\in V_m$, $t\in V_n$, and $m,n\geq 0$.
Given an adelically normed graded $K$-linear series $\overline{V}_{\sbullet}$, we have a graded $K$-linear series $\bigoplus_{m\geq 0}\aSpan{K}{\aHzst(\overline{V}_m)}$ belonging to $L$, and set
\begin{equation}
 \aBss{?}(\overline{V}_m):=\Bs\aSpan{K}{\aHzs{?}(\overline{V}_m)}\quad\text{and}\quad\aSBss{?}(\overline{V}_{\sbullet}):=\bigcap_{m\geq 1}\Bs\aSpan{K}{\aHzs{?}(\overline{V}_m)}
\end{equation}
for $m\geq 1$ and $?=\text{ss}$ and s.

\paragraph{Adelically metrized line bundles}

Let $X$ be a projective variety over $K$.
For each $v\in M_K^{\rm f}$ (respectively, $v=\infty$), we denote the Berkovich analytic space (respectively, complex analytic space) associated to $X_{K_v}:=X\times_{\Spec(K)}\Spec(K_v)$ (respectively, $X_{\CC}:=\coprod_{\sigma:K\to\CC}X\times_{\Spec(K),\sigma}\Spec(\CC)$) by $(X_v^{\rm an},\rho_v:X_v^{\rm an}\to X_{K_v})$ (see \cite[\S 3.4]{BerkovichBook}).
If $W=\Spec(A)$ is affine open subscheme of $X_{K_v}$, then a point $x\in W^{\rm an}=\rho_v^{-1}(W)$ corresponds to a multiplicative seminorm $|\cdot|_x$ on $A$ whose restriction to $K_v$ is $|\cdot|_v$, and the morphism $\rho_v|_{W^{\rm an}}:W^{\rm an}\to W$ is given by $x\to\Ker|\cdot|_x$ (see \cite[Remark~3.4.2]{BerkovichBook}).
In particular, the multiplicative seminorm $|\cdot|_x$ defines a norm on the residue field $k(\rho_v(x))$ at $\rho_v(x)$ (that is, the field of fractions of $A/\Ker|\cdot|_x$), which we denote by the same $|\cdot|_x$.

Let $L$ be a line bundle on $X$.
For each $v\in M_K^{\rm f}$, we put $L_v^{\rm an}:=\rho_v^*L_{K_v}$.
A \emph{continuous metric} $|\cdot|^{L_v}$ on $L_v^{\rm an}$ is a collection $(|\cdot|_x^{L_v})_{x\in X^{\rm an}}$ such that, for every $x\in X_v^{\rm an}$, $|\cdot|_x^{L_v}$ is a $(k(\rho_v(x)),|\cdot|_x)$-norm on $L_v^{\rm an}(x)=L_{K_v}(\rho_v(x))$ and, for every local section $s$ of $L_{K_v}$, the map $x\mapsto |s|^{L_v}(x):=|s(\rho_v(x))|_x^{L_v}$ is continuous in the Berkovich topology.

We consider a finite subset $S$ of $M_K^{\rm f}$ and set $U=\Spec(O_K)\setminus S$.
A \emph{$U$-model} of $X$ is a flat, projective, reduced, and irreducible $U$-scheme $\mathscr{X}_U$ such that $\mathscr{X}_U\times_{U}\Spec(K)$ is $K$-isomorphic to $X$.
A \emph{$U$-model} of the pair $(X,L)$ is a pair $(\mathscr{X}_U,\mathscr{L}_U)$ of a $U$-model $\mathscr{X}_U$ of $X$ and a $\QQ$-line bundle $\mathscr{L}_U$ on $\mathscr{X}_U$ such that $\mathscr{L}_U|_X=L$ as $\QQ$-line bundles.
Let $(\mathscr{X}_U,\mathscr{L}_U)$ be a $U$-model of $(X,L)$, and let $v\in U$.
Given an $x\in X_v^{\rm an}$, we can uniquely extend by the valuative criterion the natural morphism $\Spec(k(\rho_v(x)))\to X_{K_v}$ to
\begin{equation}
 t_x:\Spec(O_{k(\rho_v(x))})\to\mathscr{X}_{O_{K_v}}=\mathscr{X}_U\times_U\Spec(O_{K_v}),
\end{equation}
where $O_{k(\rho_v(x))}:=\{f\in k(\rho_v(x))\,:\,|f|_x\leq 1\}$.
Let $n\geq 1$ be an integer such that $n\mathscr{L}_{O_{K_v}}$ is a line bundle on $\mathscr{X}_{O_{K_v}}$.
For $\ell\in L_v^{\rm an}(x)$, we have $\ell^{\otimes n}\in(n\mathscr{L}_{O_{K_v}})\otimes_{O_{K_v}}k(\rho_v(x))$ and put
\begin{equation}\label{eqn:adelicassociatedtomodel}
 |\ell|^{\mathscr{L}_{O_{K_v}}}(x)=|\ell|_x^{\mathscr{L}_{O_{K_v}}}:=\inf\left\{|f|_x^{1/n}\,:\,f\in k(\rho_v(x)),\,\ell^{\otimes n}\in ft_x^*(n\mathscr{L}_{O_{K_v}})\right\}
\end{equation}
which does not depend on the choice of $n$.
Put $\Red_{\mathscr{X}_{O_{K_v}}}(x)$ as the image of the closed point of $\Spec(O_{k(\rho(x))})$ via $t_x$.
The map
\begin{equation}
 \Red_{\mathscr{X}_{O_{K_v}}}:X_v^{\rm an}\to\mathscr{X}_{O_{K_v}}\times_{\Spec(O_{K_v})}\Spec(\widetilde{K}_v)
\end{equation}
is called the \emph{reduction map} (see \cite[\S 2.4]{BerkovichBook}).
Let $\eta$ be a local frame of $n\mathscr{L}_{O_{K_v}}$ around $\Red_{\mathscr{X}_{O_{K_v}}}(x)$.
For $\ell\in L_v^{\rm an}(x)$, choose an $f\in k(\rho_v(x))$ such that $\ell^{\otimes n}=f\eta(x)$ in $(nL)_v^{\rm an}(x)$.
We then have
\begin{equation}
 |\ell|^{\mathscr{L}_{O_{K_v}}}(x)=|f|_x^{1/n}.
\end{equation}

\begin{lemma}\label{lem:propertyadelicnorm}
\begin{enumerate}
\item[\textup{(1)}] $\Red_{\mathscr{X}_{O_{K_v}}}$ is anti-continuous.
\item[\textup{(2)}] $|\cdot|^{\mathscr{L}_{O_{K_v}}}$ is a continuous metric on $L_v^{\rm an}$.
\item[\textup{(3)}] Let $\iota:X\to\mathscr{X}_U$ be the natural morphism and assume that $\mathscr{X}_U$ is integrally closed in $\iota_*\mathcal{O}_X$ \textup{(}see \cite[\S 6.3]{EGAII}\textup{)}.
Then $s^{\otimes n}$ extends to a global section of $n\mathscr{L}_U$ if and only if
\[
 \sup_{x\in X_v^{\rm an}}|s|^{\mathscr{L}_{O_{K_v}}}(x)\leq 1
\]
for every $v\in U$.
\end{enumerate}
\end{lemma}

\begin{definition}\label{defn:relativelynormal}
We say that $\mathscr{X}_U$ is \emph{relatively normal} in $X$ if the condition in Lemma~\ref{lem:propertyadelicnorm}(3) is satisfied.
Given any $U$-model $\mathscr{X}_U$, there exists an integral morphism $\nu:\mathscr{X}_U'\to\mathscr{X}_U$ of $U$-models of $X$ such that $\mathscr{X}_U'$ is relatively normal in $X$ (see \cite[\S 6.3]{EGAII}).
We call $\mathscr{X}_U'$ a \emph{relative normalization} of $\mathscr{X}_U$ in $X$.
Since the normalization $\widetilde{\nu}:\widetilde{\mathscr{X}}_U\to\mathscr{X}_U$ is a finite morphism, we can see that $\nu$ is also finite.
\end{definition}

\begin{proof}[Proof of Lemma~\ref{lem:propertyadelicnorm}]
(1): We cover $\mathscr{X}_{O_{K_v}}$ with finitely many affine open subschemes $\Spec(\mathscr{A}_i)$ such that each $\mathscr{A}_i$ is a finitely generated $O_{K_v}$-algebra.
Put $A_i:=\mathscr{A}_i\otimes_{O_{K_v}}K_v$ and $\widetilde{\mathscr{A}_i}:=\mathscr{A}_i\otimes_{O_{K_v}}\widetilde{K}_v$.
It is sufficient to show that $\Red_{\mathscr{X}_{O_{K_v}}}|_{\Spec(A_i)^{\rm an}}:\Spec(A_i)^{\rm an}\to\Spec(\widetilde{\mathscr{A}}_i)$ is anti-continuous for each $i$.
A closed subset of $\Spec(\widetilde{\mathscr{A}}_i)$ is written as
\[
 V(\widetilde{\mathfrak{a}})=\left\{p\in\Spec(\widetilde{\mathscr{A}}_i)\,:\,\text{$\widetilde{f}_j(p)=0$ for all $j$}\right\},
\]
where $\widetilde{\mathfrak{a}}$ is an ideal of $\widetilde{\mathscr{A}}_i$ generated by finitely many $\widetilde{f}_j\in\widetilde{\mathscr{A}}_i$.
Choose $f_j\in \mathscr{A}_i$ such that the image in $\widetilde{\mathscr{A}}_i$ is $\widetilde{f}_j$.
One can see that $x\in\Spec(A_i)^{\rm an}$ satisfies $\widetilde{f}_j(\Red_{\mathscr{X}_{O_{K_v}}}(x))=0$ if and only if $f_j\in O_{k(\rho_v(x))}$ belongs to the maximal ideal.
So
\[
 \Red_{\mathscr{X}_{O_{K_v}}}^{-1}(V(\widetilde{\mathfrak{a}}))=\left\{x\in\Spec(A_i)^{\rm an}\,:\,\text{$|f_j|_x<1$ for all $j$}\right\},
\]
which is open.

(2): Let $n\geq 1$ be an integer such that $n\mathscr{L}_{O_{K_v}}$ is a line bundle on $\mathscr{X}_{O_{K_v}}$.
Let $W'$ be an open subscheme of $X_{K_v}$ and let $s$ be a local section of $L_{K_v}$ over $W'$.
We are going to show that ${W'}^{\rm an}\to \RR$, $x\mapsto |s|^{\mathscr{L}_{O_{K_v}}}(x)$, is continuous.

For each $x\in {W'}^{\rm an}$, we choose an affine open neighborhood $\mathscr{W}=\Spec(\mathscr{A})$ of $\Red_{\mathscr{X}_{O_{K_v}}}(x)$ such that $n\mathscr{L}_{O_{K_v}}$ is trivial over $\mathscr{W}$.
Set $A:=\mathscr{A}\otimes_{O_{K_v}}K_v$ and $W:=\Spec(A)$.
Since $\Red_{\mathscr{X}_{O_{K_v}}}(x)$ is contained in the Zariski closure of $\rho_v(x)$ in $\mathscr{X}_{O_{K_v}}$, $W$ contains $\rho_v(x)$.
Fix a local frame $\eta$ of $n\mathscr{L}_{O_{K_v}}$ over $\mathscr{W}$, and write $s^{\otimes n}=f\eta|_W$ with $f\in A$.
Then
\[
 |s|^{\mathscr{L}_{O_{K_v}}}(x)=|f|_x^{1/n}
\]
is continuous over $W^{\rm an}$ by definition of the Berkovich topology.
So we are done.

(3): It suffices to show the ``if'' part.
We can assume $s\neq 0$.
The section $s^{\otimes n}$ can be regarded as a non-zero rational section of $n\mathscr{L}_U$ on $\mathscr{X}_U$, so it suffices to show that the Cartier divisor $\div_{\mathscr{X}_U}(s^{\otimes n})$ on $\mathscr{X}_U$ is effective.

\begin{claim}\label{clm:normal}
It suffices to say
\begin{equation}
 \ord_{Z}(\div_{\mathscr{X}}(s^{\otimes n}))\geq 0
\end{equation}
for every vertical prime divisor $Z$ on $\mathscr{X}_U$.
\end{claim}

\begin{proof}[Proof of Claim~\ref{clm:normal}]
There exists a non-empty open subset $U'\subset U$ such that $s^{\otimes n}$ extends to a global section on $\mathscr{X}_{U'}$.
Let $\mathscr{W}=\Spec(\mathscr{A})$ be an affine open subscheme of $\mathscr{X}_U$ such that $n\mathscr{L}_U$ is trivial with local frame $\eta$.
We can write $s^{\otimes n}=f\eta$ with a rational function $f$ on $\mathscr{W}$.
The domain of $f$ contains $\mathscr{W}\cap\mathscr{X}_{U'}$ and all the generic points of the vertical fibers of $\mathscr{W}$.
So, by Lemma~\ref{lem:ring}(3), we have $f\in\mathscr{A}$.
\end{proof}

Suppose that $Z$ is lying over a place $v\in U$.
We define $|\cdot|_Z$ by
\[
 |\phi|_Z:=\begin{cases} (\sharp\widetilde{K}_v)^{-\frac{\ord_Z(\phi)}{\ord_Z(\varpi_v)}} & \text{if $\phi\neq 0$,} \\ 0 & \text{if $\phi=0$} \end{cases}
\]
for $\phi\in \Rat(X_{K_v})$, where $\Rat(X_{K_v})$ denotes the field of rational functions on $X_{K_v}$.
Then $|\cdot|_Z$ gives a norm on $\Rat(X_{K_v})$ extending $|\cdot|_v$ on $K_v$, so $|\cdot|_Z$ corresponds to a point $x_Z\in X_v^{\rm an}$.
Note that $\rho_v(x_Z)$ is the generic point of $X_{K_v}$ and $\Red_{\mathscr{X}_{O_{K_v}}}(x_Z)$ is the generic point of $Z$, so $x_Z$ belongs to the Shilov boundary of $X_v^{\rm an}$ (see \cite[Proposition~2.4.4]{BerkovichBook}).
Let $f$ be a local equation defining $\div_{\mathscr{X}_U}(s^{\otimes n})$ around the generic point of $Z$.
We then have
\[
 \left(|s|^{\mathscr{L}_{O_{K_v}}}(x_Z)\right)^n=|s^{\otimes n}|^{n\mathscr{L}_{O_{K_v}}}(x_Z)=|f|_Z=(\sharp\widetilde{K}_v)^{-\frac{\ord_Z(f)}{\ord_Z(\varpi_v)}}\leq 1,
\]
so $\ord_Z(\div_{\mathscr{X}}(s^{\otimes n}))=\ord_Z(f)\geq 0$.
\end{proof}

\begin{lemma}\label{lem:ring}
Let $A$ be a Noetherian integral domain and let $S$ be a multiplicative subset of $A\setminus\{0\}$.
Suppose that $A$ is integrally closed in $S^{-1}A$.
\begin{enumerate}
\item[\textup{(1)}] Let $a\in A\setminus\{0\}$ and let $\mathfrak{p}$ be a prime ideal of $A$ associated to $aA$.
If $\mathfrak{p}\cap S$ is non-empty, then $A_{\mathfrak{p}}$ is a discrete valuation ring and $\mathfrak{p}$ has height one.
\item[\textup{(2)}] For any prime ideal $\mathfrak{P}$ of $A$ such that $\mathfrak{P}\cap S\neq\emptyset$, one has
\[
 \depth(A_{\mathfrak{P}})\geq\min\left\{\height(\mathfrak{P}),2\right\}.
\]
\item[\textup{(3)}] One has
\[
 A=S^{-1}A\cap\bigcap_{\substack{\mathfrak{P}\cap S\neq\emptyset \\ \height(\mathfrak{P})=1}}A_{\mathfrak{P}}.
\]
\end{enumerate}
\end{lemma}

\begin{proof}
(1): Let $\varpi\in\mathfrak{p}\cap S$.
There exists a $b\in A$ such that $\mathfrak{p}=\{x\in A\,:\,xb\in aA\}$.
So $\mathfrak{p}A_{\mathfrak{p}}=\{x\in A_{\mathfrak{p}}\,:\,xb\in aA_{\mathfrak{p}}\}$ and $\varpi b\in aA_{\mathfrak{p}}$.

\begin{claim}\label{clm:ring1}
$\mathfrak{p}A_{\mathfrak{p}}$ is an invertible ideal of $A_{\mathfrak{p}}$.
\end{claim}

\begin{proof}[Proof of Claim~\ref{clm:ring1}]
We have $ba^{-1}\in(\mathfrak{p}A_{\mathfrak{p}})^{-1}\cap A_{\mathfrak{p}}[1/\varpi]$ and $ba^{-1}\notin A_{\mathfrak{p}}$.
Suppose that $(ba^{-1})\mathfrak{p}A_{\mathfrak{p}}$ is contained in $\mathfrak{p}A_{\mathfrak{p}}$.
Since $A_{\mathfrak{p}}$ is Noetherian, we can see by the ``determinant trick'' (see \cite[Theorem~2.1]{MatsumuraBook}) that $ba^{-1}\in A_{\mathfrak{p}}[1/\varpi]$ is integral over $A_{\mathfrak{p}}$.
On the other hand, by the hypothesis, $A_{\mathfrak{p}}$ is integrally closed in $A_{\mathfrak{p}}[1/\varpi]$ and it is a contradiction.
Hence $(ba^{-1})\mathfrak{p}A_{\mathfrak{p}}=A_{\mathfrak{p}}$.
\end{proof}

By Claim~\ref{clm:ring1} and \cite[Theorem~11.4]{MatsumuraBook}, $A_{\mathfrak{p}}$ is a discrete valuation ring and $\mathfrak{p}$ has height one.

(2): Since $\mathfrak{P}\cap S\neq\emptyset$, $A_{\mathfrak{P}}$ has depth $\geq 1$.
Suppose that $\height(\mathfrak{P})\geq 2$.
Let $x_1\in\mathfrak{P}\cap S$.
If $\mathfrak{P}\subset\bigcup_{\mathfrak{p}\in\Ass_A(A/x_1A)}\mathfrak{p}$, then $\mathfrak{P}\subset\mathfrak{p}$ for a $\mathfrak{p}\in\Ass_A(A/x_1A)$ (see \cite[Exercise~1.6]{MatsumuraBook}) and it is a contradiction by the assertion (1).
So one can take an
\[
 x_2\in\mathfrak{P}\setminus\left(\bigcup_{\mathfrak{p}\in\Ass_A(A/x_1A)}\mathfrak{p}\right).
\]
By \cite[Theorem~6.1(ii)]{MatsumuraBook}, $x_2\in\mathfrak{P}A_{\mathfrak{P}}$ is ($A_{\mathfrak{P}}/x_1A_{\mathfrak{P}}$)-regular.
So, $A_{\mathfrak{P}}$ has depth $\geq 2$.

(3): Let $a\in A\setminus\{0\}$ and let
\[
 aA=\bigcap_{\mathfrak{p}\in\Ass_A(A/aA)}I(\mathfrak{p})
\]
be a reduced primary decomposition of $aA$, where $I(\mathfrak{p})$ is $\mathfrak{p}$-primary in $A$.

\begin{claim}
One has
\begin{equation}
 S^{-1}(aA)\cap A=\bigcap_{\mathfrak{p}\in\Ass_A(A/aA)\cap\Spec(S^{-1}A)}I(\mathfrak{p}).
\end{equation}
\end{claim}

\begin{proof}
Since
\[
 S^{-1}(aA)=\bigcap_{\mathfrak{p}\in\Ass_A(A/aA)\cap\Spec(S^{-1}A)}S^{-1}I(\mathfrak{p}),
\]
it suffices to show $S^{-1}I(\mathfrak{p})\cap A=I(\mathfrak{p})$ for $\mathfrak{p}\in\Ass_A(A/aA)\cap\Spec(S^{-1}A)$.
Suppose that $sx\in I(\mathfrak{p})$ for $s\in S$ and $x\in A$.
If $x\notin I(\mathfrak{p})$, then $s$ is a zero-divisor for $A/I(\mathfrak{p})$ and $s\notin\mathfrak{p}$, so it is a contradiction.
Hence, $x\in I(\mathfrak{p})$.
\end{proof}

On the other hand, one has by the assertion (1)
\begin{equation}
 I(\mathfrak{p})=aA_{\mathfrak{p}}\cap A
\end{equation}
for every $\mathfrak{p}\in\Ass_A(A/aA)\setminus\Spec(S^{-1}A)$.
So
\begin{equation}\label{eqn:ring_aA}
 aA=(aS^{-1}A\cap A)\cap\bigcap_{\mathfrak{p}\in\Ass_A(A/aA)\setminus\Spec(S^{-1}A)}(aA_{\mathfrak{p}}\cap A).
\end{equation}
If
\[
 ba^{-1}\in S^{-1}A\cap\bigcap_{\substack{\mathfrak{P}\cap S\neq\emptyset \\ \height(\mathfrak{P})=1}}A_{\mathfrak{P}},
\]
then, by the equation (\ref{eqn:ring_aA}), we have
\[
 b\in (aS^{-1}A\cap A)\cap\bigcap_{\mathfrak{P}\in\Ass_A(A/aA)\setminus\Spec(S^{-1}A)}(aA_{\mathfrak{P}}\cap A)=aA
\]
and $ba^{-1}\in A$.
\end{proof}

\begin{lemma}\label{lem:modelofdefn}
Let $L$ be a line bundle on $X$.
\begin{enumerate}
\item[\textup{(1)}] There exists an $O_K$-model $(\mathscr{X},\mathscr{L})$ of $(X,L)$ such that $\mathscr{L}$ is a line bundle on $\mathscr{X}$.
\item[\textup{(2)}] Let $U$ be a non-empty open subset of $\Spec(O_K)$ and let $(\mathscr{X}_U,\mathscr{L}_U)$ and $(\mathscr{X}_U',\mathscr{L}_U')$ be two $U$-models of $(X,L)$.
There then exists a non-empty open subset $U_0\subset U$ such that $\mathscr{X}_{U_0}$ and $\mathscr{X}_{U_0}'$ are $U_0$-isomorphic and $\mathscr{L}_{U_0}=\mathscr{L}_{U_0}'$ as $\QQ$-line bundles.
\end{enumerate}
\end{lemma}

\begin{proof}
(1): Since $X$ is projective, $L$ can be written as a difference $L_2-L_1$ of two effective line bundles.
Each of $-L_i$ can be regarded as a locally principal ideal sheaf $I_i$ of $\mathcal{O}_X$ by a non-zero global section of $L_i$.

Let $\mathscr{X}_1$ be an $O_K$-model of $X$, let $J_1$ be the kernel of $\mathcal{O}_{\mathscr{X}_1}\to\mathcal{O}_X/I_1$, and let $\mu_1:\mathscr{X}_2\to\mathscr{X}_1$ be the blow-up along $J_1$.
So $\mathscr{I}_1:=J_1\mathcal{O}_{\mathscr{X}_2}$ is invertible and $\mathscr{X}_2$ is an $O_K$-model of $X$.
Let $J_2$ be the kernel of $\mathscr{O}_{\mathscr{X}_2}\to\mathcal{O}_X/I_2$, let $\mu_2:\mathscr{X}_3\to\mathscr{X}_2$ be the blow-up along $J_2$, and let $\mathscr{I}_2:=J_2\mathcal{O}_{\mathscr{X}_3}$.
Then $\mathscr{X}:=\mathscr{X}_3$ is an $O_K$-model of $X$ and $\mathscr{L}:=\mu_2^*\mathscr{I}_1\otimes\mathscr{I}_2^{\otimes (-1)}$ is an invertible sheaf extending $L$.

(2): Let $n\geq 1$ be an integer such that both $n\mathscr{L}_U$ and $n\mathscr{L}_U'$ are line bundles.
Let $\mathscr{X}_U''$ be the Zariski closure of the diagonal $\id_X\times\id_X:X\to\mathscr{X}_U\times_U\mathscr{X}_U'$ in $\mathscr{X}_U\times_U\mathscr{X}_U'$.
We then have two birational projective $U$-morphisms $\mathscr{X}_U\xleftarrow{\varphi_U}\mathscr{X}_U''\xrightarrow{\psi_U}\mathscr{X}_U'$.
The two line bundles $\varphi_{U}^*(n\mathscr{L}_{U})$ and $\psi^*_{U}(n\mathscr{L}_{U}')$ on $\mathscr{X}_U''$ are isomorphic over the generic fiber $X$.
So, by \cite[Corollaires~(8.8.2.4) et (8.8.2.5)]{EGAIV_3}, one can find a $U_0\subset U$ such that $\varphi_{U_0}$ and $\psi_{U_0}$ are isomorphisms and $\varphi_{U_0}^*(n\mathscr{L}_{U_0})$ and $\psi^*_{U_0}(n\mathscr{L}_{U_0}')$ are isomorphic.
\end{proof}

\begin{lemma}\label{lem:pullbackadelic}
Let $j:Y\to X$ be a morphism of projective varieties over $K$ and let $L$ be a line bundle on $X$.
Let $U\subset\Spec(O_K)$ be a non-empty open subset.
Let $(\mathscr{X}_U,\mathscr{L}_U)$ be a $U$-model of $(X,L)$, let $\mathscr{Y}_U$ be a $U$-model of $Y$, and let $j_U:\mathscr{Y}_U\to\mathscr{X}_U$ be a $U$-morphism that extends $j$.
Then, for any $P\in U$, any local section $s$ of $L$, and any $y\in Y_P^{\rm an}$,
\[
 |j^*s|^{j_U^*\mathscr{L}_{O_{K_P}}}(y)=|s|^{\mathscr{L}_{O_{K_P}}}(j^{\rm an}_P(y)).
\]
\end{lemma}

\begin{proof}
Set $x:=j_P^{\rm an}(y)$.
Note that
\[
 \rho_P(x)=j(\rho_P(y))\quad\text{and}\quad\Red_{\mathscr{X}_{O_{K_P}}}(x)=j_U(\Red_{\mathscr{Y}_{O_{K_P}}}(y)).
\]
Let $n\geq 1$ be an integer such that $n\mathscr{L}_U$ is a line bundle.
Let $\eta$ be a local frame of $n\mathscr{L}_{O_{K_P}}$ around $\Red_{\mathscr{X}_{O_{K_P}}}(x)$ and set $s(\rho_P(x))=f\eta(\rho_P(x))$ in $L_P^{\rm an}(x)$ with $f\in k(\rho_P(x))$.
Then $j_U^*\eta$ is a local frame of $j_U^*(n\mathscr{L}_{O_{K_P}})$ around $\Red_{\mathscr{Y}_{O_{K_P}}}(y)$ and $(j^*s)(\rho_P(y))=(j^*f)(j_U^*\eta)(\rho_P(y))$.
So
\[
 |j^*s|^{j_U^*\mathscr{L}_{O_{K_P}}}(y)=|j^*f|_y=|f|_x=|s|^{\mathscr{L}_{O_{K_P}}}(x).
\]
\end{proof}

\begin{definition}\label{defn:adelicallymetrizedlinebdl}
Let $S$ be a finite subset of $M_K$ containing $\infty$.
An \emph{$S$-adelically metrized line bundle} on $X$ is a pair $\overline{L}=\left(L,(|\cdot|_v^{\overline{L}})_{v\in M_K\setminus S}\right)$ of a line bundle $L$ on $X$ and a collection of metrics $(|\cdot|_v^{\overline{L}})_{v\in M_K\setminus S}$ having the following properties.
\begin{enumerate}
\item[(a)] For every $v\in M_K\setminus S$, $|\cdot|_v^{\overline{L}}$ is a continuous metric on $L_v^{\rm an}$.
\item[(b)] There exist a non-empty open set $U\subset\Spec(O_K)\setminus S$ and a $U$-model $(\mathscr{X}_U,\mathscr{L}_U)$ of $(X,L)$ such that, for every $P\in U$, $|\cdot|_P^{\overline{L}}(x)=|\cdot|^{\mathscr{L}_{O_{K_P}}}(x)$ holds on $X_P^{\rm an}$.
\end{enumerate}
By Lemma~\ref{lem:modelofdefn}, there exist a non-empty open subset $U\subset\Spec(O_K)$ and a $U$-model $(\mathscr{X}_U,\mathscr{L}_U)$ of $(X,L)$ such that $(\mathscr{X}_U,\mathscr{L}_U)$ satisfies the condition (b) above and $\mathscr{L}_U$ is a line bundle on $\mathscr{X}_U$.
We call such a $U$-model a \emph{$U$-model of definition for $\overline{L}$}.

An \emph{adelically metrized line bundle} on $X$ is a pair $\overline{L}=\left(L,(|\cdot|_v^{\overline{L}})_{v\in M_K}\right)$ such that $\overline{L}^{\{\infty\}}:=\left(L,(|\cdot|_v^{\overline{L}})_{v\in M_K^{\rm f}}\right)$ is an $\{\infty\}$-adelically metrized line bundle on $X$ and $|\cdot|_{\infty}^{\overline{L}}$ is a continuous Hermitian metric on $L_{\infty}^{\rm an}$ such that for every $x\in X_{\infty}^{\rm an}$ and every local section $s$ of $L_{\infty}^{\rm an}$
\begin{equation}
 |F_{\infty}^*s|(\overline{x})=|s|(x)
\end{equation}
holds, where $F_{\infty}:x\mapsto\overline{x}$ is the complex conjugation on $X_{\infty}^{\rm an}$.
The $\ZZ$-module of all the adelically metrized line bundles on $X$ is denoted by $\aPic(X)$, and an element in $\aPic_{\QQ}(X):=\aPic(X)\otimes_{\ZZ}\QQ$ is called an \emph{adelically metrized $\QQ$-line bundle} on $X$.
\end{definition}

Let $\overline{L}=\left(L,(|\cdot|_v^{\overline{L}})_{v\in M_K\setminus S}\right)$ be an adelically metrized line bundle on $X$, let $s$ be a non-zero rational section of $L$, and let $v\in M_K$.
Suppose that $x\notin\Supp(\div(s))_v^{\rm an}$, so $\rho_v(x)\notin\Supp(\div s)$.
Let $f$ be a regular function defined around $\rho_v(x)$ such that $fs$ is a local section of $L$ around $\rho_v(x)$ and $f(\rho_v(x))\neq 0$.
Then we can define
\[
 |s|_v^{\overline{L}}(x):=\frac{|fs|_v^{\overline{L}}(x)}{|f|_x},
\]
which does not defend on the choice of $f$.

\begin{remark}\label{rem:remadelic}
\begin{enumerate}
\item[(1)] 
Let $\overline{L}$ be an adelically metrized $\QQ$-line bundle in our sense, let $n\geq 1$ be an integer such that $n\overline{L}$ is an adelically metrized line bundle, and let $s$ be a non-zero rational section of $nL$.
Then
\[
 (1/n)\left(\div(s),(-2\log|s|_v^{n\overline{L}})_{v\in M_K}\right)
\]
is an adelic arithmetic $\RR$-divisor in the sense of Moriwaki \cite{Moriwaki13}.
\item[(2)] Let $K_X:=\Hz(\mathcal{O}_X)$ and let $K\subset K'\subset K_X$ be a subextension.
Given a finite place $v$ of $K$ and a point $x\in X_v^{\rm an}$, we can restrict $|\cdot|_x$ to $K'$ and obtain a place $w$ of $K'$ lying over $v$.
Thus we have
\[
 X_v^{\rm an}=\bigcup_{w|v}X_w^{\rm an},
\]
where $w$ runs over all the finite places of $K'$ lying over $v$.
So, in particular, the notion of adelically metrized line bundles does not depend on the choice of $K$.
\item[(3)] Let $\mathscr{X}$ be a projective arithmetic variety over $O_K$, let $\mathscr{L}$ be a $\QQ$-line bundle on $\mathscr{X}$, and let $n\geq 1$ be an integer such that $n\mathscr{L}_K$ is a line bundle on $\mathscr{X}_K$.
To $\mathscr{L}$, we can associate an $\{\infty\}$-adelically metrized line bundle
\[
 \mathscr{L}^{\rm ad}:=(1/n)\left(n\mathscr{L}_{K},(|\cdot|^{n\mathscr{L}_{O_{K_P}}})_{P\in M_K^{\rm f}}\right)
\]
and, to a continuous Hermitian $\QQ$-line bundle $\overline{\mathscr{L}}$ on $\mathscr{X}$, we can associate an adelically metrized $\QQ$-line bundle
\[
 \overline{\mathscr{L}}^{\rm ad}:=(1/n)\left(n\mathscr{L}_{K},(|\cdot|^{n\mathscr{L}_{O_{K_P}}})_{P\in M_K^{\rm f}}\cup(|\cdot|_{\infty}^{n\overline{\mathscr{L}}})\right).
\]
\item[(4)] For each $v\in M_K$, we define the \emph{trivial metric} $|\cdot|_v^{\rm triv}$ on $\mathcal{O}_{X,v}^{\rm an}$ as the collection $(|\cdot|_x)_{x\in X_v^{\rm an}}$ of the norms $|\cdot|_x$ on $k(\rho_v(x))$.
Then $\overline{\mathcal{O}}_X^{\rm triv}:=\left(\mathcal{O}_X,(|\cdot|_v^{\rm triv})_{v\in M_K}\right)$ is an adelically metrized line bundle on $X$.
For any continuous function $\lambda:X_P^{\rm an}\to\RR$, we set
\[
 \overline{\mathcal{O}}_X(\lambda[P]):=\left(\mathcal{O}_X,(|\varpi_P|_P^{\lambda}|\cdot|_P^{\rm triv})\cup(|\cdot|_v^{\rm triv})_{v\in M_K\setminus\{P\}}\right)
\]
and, for any continuous function $\lambda:X_{\infty}^{\rm an}\to\RR$ that is invariant under the complex conjugation, we set
\[
 \overline{\mathcal{O}}_X(\lambda[\infty]):=\left(\mathcal{O}_X,(\exp(-\lambda)|\cdot|_{\infty}^{\rm triv})\cup(|\cdot|_v^{\rm triv})_{v\in M_K^{\rm f}}\right).
\]
Let $\pi:\mathscr{X}\to\Spec(O_K)$ be an $O_K$-model of $X$.
If $\lambda\in\QQ$, then the adelically metrized $\QQ$-line bundle on $X$ associated to $\lambda\left(\mathcal{O}_{\mathscr{X}}(\pi^{-1}(P)),|\cdot|^{\rm triv}_{\infty}\right)$ is $\overline{\mathcal{O}}_X(\lambda[P])$.
\end{enumerate}
\end{remark}

Let $\overline{L}=\left(L,(|\cdot|_v^{\overline{L}})_{v\in M_K}\right)$ be an adelically metrized line bundle on $X$.
For each $v\in M_K$, we define the supremum norm of $s\in\Hz(L)\otimes_{K}K_v$ by
\begin{equation}
 \|s\|_{v,\sup}^{\overline{L}}:=\sup_{x\in X_v^{\rm an}}\left\{|s|_v^{\overline{L}}(x)\right\}.
\end{equation}
Then
\begin{equation}\label{eqn:adelicgradedlinser}
 \overline{V}_{\sbullet}=\bigoplus_{m\geq 0}\left(\Hz(mL),(\|\cdot\|_{v,\sup}^{\overline{L}})_{v\in M_K}\right)
\end{equation}
is an adelically normed graded $K$-linear series belonging to $L$, and we write, for short, $\aHzf(m\overline{L}):=\aHzf(\overline{V}_m)$,
\[
 \aHzs{?}(m\overline{L}):=\aHzs{?}(\overline{V}_m),\quad \aBss{?}(m\overline{L}):=\aBss{?}(\overline{V}_m),\quad\text{and}\quad\aSBss{?}(\overline{L}):=\aSBss{?}(\overline{V}_{\sbullet}),
\]
where $?=\text{ss}$ or s.
We refer to a section in $\aHzst(\overline{L})$ as a \emph{strictly small section of $\overline{L}$}.
Since $\aSBss{?}(m\overline{L})=\aSBss{?}(\overline{L})$ for every $m\geq 1$, we can define $\aSBss{?}(\overline{L})$ for every adelically metrized $\QQ$-line bundle $\overline{L}$.

Let $(\mathscr{X}_U,\mathscr{L}_U)$ be a $U$-model of definition for $\overline{L}^{\{\infty\}}$ and let $\nu_U:\mathscr{X}_U'\to\mathscr{X}_U$ be a relative normalization in $X$ (Definition~\ref{defn:relativelynormal}).
Then we have a natural injection
\begin{equation}\label{eqn:injmodelofdef}
 \aHzf(\overline{L})\to\Hz(\nu_U^*\mathscr{L}_U)
\end{equation}
(Lemma~\ref{lem:propertyadelicnorm}(3)).

\begin{definition}
The \emph{arithmetic volume} of an adelically metrized line bundle $\overline{L}$ is defined as
\begin{equation}\label{eqn:defnarithvolume}
 \avol(\overline{L}):=\limsup_{m\to+\infty}\frac{\log\sharp\aHzst(m\overline{L})}{m^{\dim X+1}/(\dim X+1)!}.
\end{equation}
(Note that in the original definition in \cite{Moriwaki13}, Moriwaki uses $\aHzsm(m\overline{L})$ instead of $\aHzst(m\overline{L})$.
These two definitions are actually equivalent by continuity of the volume function.)
We recall the following results (see \cite{Moriwaki13} for detail).
\begin{enumerate}
\item[(1)] The $\limsup$ in (\ref{eqn:defnarithvolume}) is actually a limit.
\item[(2)] For any adelically metrized line bundle $\overline{L}$ and for any $a\geq 1$, $\avol(a\overline{L})=a^{\dim X+1}\avol(\overline{L})$.
In particular, we can define $\avol$ for adelically metrized $\QQ$-line bundles on $X$.
\end{enumerate}
Positivity notions for adelically metrized $\QQ$-line bundles are defined as follows.
\begin{description}
\item[(effective)] An adelically metrized line bundle $\overline{L}$ on $X$ is said to be \emph{effective} if $\aHzsm(\overline{L})\neq\{0\}$.
For $\overline{L}_1,\overline{L}_2\in\aPic(X)$, we write $\overline{L}_1\leq\overline{L}_2$ if $\overline{L}_2-\overline{L}_1$ is effective.
\item[(big)] We say that an adelically metrized $\QQ$-line bundle $\overline{L}$ is \emph{big} if $\avol(\overline{L})>0$.
\item[(pseudoeffective)] We say that an adelically metrized $\QQ$-line bundle $\overline{L}$ is \emph{pseudoeffective} if $\overline{L}+\overline{A}$ is big for every big adelically metrized $\QQ$-line bundle $\overline{A}$.
For $\overline{L}_1,\overline{L}_2\in\aPic_{\QQ}(X)$, we write $\overline{L}_1\preceq \overline{L}_2$ if $\overline{L}_2-\overline{L}_1$ is pseudoeffective.
\end{description}
\end{definition}

\begin{example}
Unlike the geometric case (Zariski's theorem on removable base loci), $\aSBs(\overline{L})$ can contain an isolated closed point.
Let $\mathscr{X}:=\PP_{\ZZ}^d$ be the projective space and let $\mathscr{L}:=\mathcal{O}_{\mathscr{X}}(1)$ be the hyperplane line bundle.
We consider the Hermitian metric $|\cdot|^{\overline{\mathscr{L}}}$ on $\mathscr{L}$ defined by
\[
 |X_0|^{\overline{\mathscr{L}}}(x_0:\dots:x_d)^2:=\frac{|x_0|^2}{\max\left\{a_0^2|x_0|^2,\dots,a_d^2|x_d|^2\right\}},
\]
and set $\overline{\mathscr{L}}:=(\mathscr{L},|\cdot|^{\overline{\mathscr{L}}})$.
If $a_j\geq 1$ for $j\neq i$ and $0<a_i<1$, then $\aSBs(\overline{\mathscr{L}}^{\rm ad})=\{(0:\cdots:0:\overset{i}{1}:0:\cdots:0)\}$.
\end{example}

\begin{proposition}\label{prop:adelicapproximation}
Let $\overline{L}=\left(L,(|\cdot|_v^{\overline{L}})_{v\in M_K^{\rm f}}\right)$ be an $\{\infty\}$-adelically metrized line bundle on $X$ and let $U$ be a non-empty open subset of $\Spec(O_K)$.
\begin{enumerate}
\item[\textup{(1)}] For any $U$-model of definition $(\mathscr{X}_U,\mathscr{L}_U)$ for $\overline{L}$ and for any $\varepsilon>0$, there exists an $O_K$-model $(\mathscr{X}_{\varepsilon},\mathscr{L}_{\varepsilon})$ such that $\mathscr{X}_{\varepsilon}\times_{\Spec(O_K)}U$ is $U$-isomorphic to $\mathscr{X}_U$, $\mathscr{L}_{\varepsilon}|_{\mathscr{X}_U}=\mathscr{L}_U$ as $\QQ$-line bundles, and
\[
 \exp(-\varepsilon)|\cdot|^{\mathscr{L}_{\varepsilon,O_{K_P}}}(x)\leq |\cdot|_P^{\overline{L}}(x)\leq \exp(\varepsilon)|\cdot|^{\mathscr{L}_{\varepsilon,O_{K_P}}}(x)
\]
for every $P\in\Spec(O_K)\setminus U$ and every $x\in X_P^{\rm an}$.
\item[\textup{(2)}] If $L$ is nef, then the following are equivalent.
\begin{enumerate}
\item[\textup{(a)}] For any $U$-model of definition $(\mathscr{X}_U,\mathscr{L}_U)$ for $\overline{L}$ and for any $\varepsilon>0$, there exists an $O_K$-model $(\mathscr{X}_{\varepsilon},\mathscr{L}_{\varepsilon})$ such that $\mathscr{X}_{\varepsilon}\times_{\Spec(O_K)}U$ is $U$-isomorphic to $\mathscr{X}_U$, $\mathscr{L}_{\varepsilon}|_{\mathscr{X}_U}=\mathscr{L}_U$ as $\QQ$-line bundles, $\mathscr{L}_{\varepsilon}$ is relatively nef, and
\[
 \exp(-\varepsilon)|\cdot|^{\mathscr{L}_{\varepsilon,O_{K_P}}}(x)\leq |\cdot|_P^{\overline{L}}(x)\leq \exp(\varepsilon)|\cdot|^{\mathscr{L}_{\varepsilon,O_{K_P}}}(x)
\]
for every $P\in\Spec(O_K)\setminus U$ and every $x\in X_P^{\rm an}$.
\item[\textup{(b)}] For any $U$-model of definition $(\mathscr{X}_U,\mathscr{L}_U)$ for $\overline{L}$ and for any rational number $\varepsilon>0$, there exist $O_K$-models $(\mathscr{X}_{\varepsilon},\mathscr{L}_{\varepsilon,1})$ and $(\mathscr{X}_{\varepsilon},\mathscr{L}_{\varepsilon,1})$ such that $\mathscr{X}_{\varepsilon}\times_{\Spec(O_K)}U$ is $U$-isomorphic to $\mathscr{X}_U$, $\mathscr{L}_{\varepsilon,i}|_{\mathscr{X}_U}=\mathscr{L}_U$ as $\QQ$-line bundles, $\mathscr{L}_{\varepsilon,i}$ are relatively nef, and
\[
 \exp(-\varepsilon)|\cdot|_P^{\overline{L}}(x)\leq |\cdot|^{\mathscr{L}_{\varepsilon,1,O_{K_P}}}(x)\leq |\cdot|_P^{\overline{L}}(x)\leq |\cdot|^{\mathscr{L}_{\varepsilon,2,O_{K_P}}}(x)\leq\exp(\varepsilon)|\cdot|_P^{\overline{L}}(x)
\]
for every $P\in\Spec(O_K)\setminus U$ and every $x\in X_P^{\rm an}$.
\end{enumerate}
\end{enumerate}
\end{proposition}

\begin{proof}
For the assertion (1), we refer to \cite[Theorem~4.1.3]{Moriwaki13} (see also \cite{Bou_Fav_Mat11}).
We are going to show the equivalence (a) $\Leftrightarrow$ (b) in the assertion (2).
The implication (b) $\Rightarrow$ (a) is clear since we can take a rational number $0<\varepsilon'\leq\varepsilon$ and set $\mathscr{L}_{\varepsilon}:=\mathscr{L}_{\varepsilon',1}$.

(a) $\Rightarrow$ (b): By the condition (a), given any rational number $\varepsilon>0$ we can find an $O_K$-model $(\mathscr{X}_{\varepsilon},\mathscr{L}_{\varepsilon})$ such that $\mathscr{X}_{\varepsilon}\times_{\Spec(O_K)}U$ is $U$-isomorphic to $\mathscr{X}_U$, $\mathscr{L}_{\varepsilon}|_{\mathscr{X}_U}=\mathscr{L}_U$ as $\QQ$-line bundles, $\mathscr{L}_{\varepsilon}$ is relatively nef, and
\[
 |\varpi_P|_P^{\varepsilon'}|\cdot|^{\mathscr{L}_{\varepsilon,O_{K_P}}}(x)\leq |\cdot|_P^{\overline{L}}(x)\leq |\varpi_P|_P^{-\varepsilon'}|\cdot|^{\mathscr{L}_{\varepsilon,O_{K_P}}}(x)
\]
for every $P\in\Spec(O_K)\setminus U$ and $x\in X_P^{\rm an}$, where $\varepsilon'$ is a rational number such that
\[
 0<\varepsilon'\leq\frac{\varepsilon}{-\log|\varpi_P|_P}.
\]
Let $\pi_{\varepsilon}:\mathscr{X}_{\varepsilon}\to\Spec(O_K)$ denote the structure morphism and set
\[
 \mathscr{L}_{\varepsilon,1}:=\mathscr{L}_{\varepsilon}-\varepsilon'\sum_{P\notin U}\mathcal{O}_{\mathscr{\mathscr{X}_{\varepsilon}}}(\pi_{\varepsilon}^{-1}(P))\quad\text{and}\quad\mathscr{L}_{\varepsilon,2}:=\mathscr{L}_{\varepsilon}+\varepsilon'\sum_{P\notin U}\mathcal{O}_{\mathscr{\mathscr{X}_{\varepsilon}}}(\pi_{\varepsilon}^{-1}(P)).
\]
Then
\[
 |\cdot|^{\mathscr{L}_{\varepsilon,1,O_{K_P}}}=|\varpi_P|_P^{\varepsilon'}|\cdot|^{\mathscr{L}_{\varepsilon,O_{K_P}}}\quad\text{and}\quad|\cdot|^{\mathscr{L}_{\varepsilon,2,O_{K_P}}}=|\varpi_P|_P^{-\varepsilon'}|\cdot|^{\mathscr{L}_{\varepsilon,O_{K_P}}}.
\]
\end{proof}

\begin{definition}\label{defn:verticallynef}
Let $\overline{L}=\left(L,(|\cdot|_v^{\overline{L}})_{v\in M_K}\right)$ be an adelically metrized line bundle on $X$ and let $x\in X(\overline{K})$ be an algebraic point on $X$.
Let $K(x)$ be a field of definition for $x$ such that $K(x)/\QQ$ is finite, that is, $K(x)$ is a number field that contains the residue field $k(x)$ of the image of $x$.
For each $P\in M_K^{\rm f}$, we have a canonical isomorphism $K(x)\otimes_KK_P=\bigoplus_{Q|P}K(x)_Q$ (see \cite[Chap.\ VI, \S 8.2, Corollary~2]{BourbakiCA72}), so
\begin{equation}
 [K(x):K]=[K(x)\otimes_KK_P:K_P]=\sum_{Q|P}[K(x)_Q:K_P].
\end{equation}
We can choose a non-zero rational section $s$ of $L$ such that $x\notin\Supp(\div(s))$.
For $v\in M_K^{\rm f}$, we set
\begin{equation}\label{eqn:defnheight1}
 \adeg_v\left(\widehat{\div}(s)|_x\right):=-\sum_{w|v} [K(x)_w:K_v]\log|s|_v^{\overline{L}}(x^w),
\end{equation}
where $w$ runs over all the finite places of $K(x)$ lying over $v$ and $x^w\in X_v^{\rm an}$ is the point corresponding to $(k(x)_w,w|_{k(x)})$.
For $v=\infty$,
\begin{equation}\label{eqn:defnheight2}
 \adeg_{\infty}\left(\widehat{\div}(s)|_x\right):=-\sum_{\sigma:K(x)\to\CC}\log|s|_{\infty}^{\overline{L}}(x^{\sigma}).
\end{equation}
We then define the \emph{height} of $x$ by
\begin{equation}\label{eqn:defnheight}
 h_{\overline{L}}(x):=\frac{1}{[K(x):K]}\sum_{v\in M_K}\adeg_v\left(\widehat{\div}(s)|_x\right),
\end{equation}
which does not depend on the choice of $K(x)$ and $s$.
For $\overline{L},\overline{M}\in\aPic(X)$, we have $h_{\overline{L}+\overline{M}}(x)=h_{\overline{L}}(x)+h_{\overline{M}}(x)$, so we can define $h_{\overline{L}}(x)$ for every $\overline{L}\in\aPic_{\QQ}(X)$.

\begin{description}
\item[(nef)] We say that $\overline{L}\in\aPic(X)$ is \emph{nef} if the following are satisfied.
\begin{enumerate}
\item[\textup{(a)}] $L$ is nef.
\item[\textup{(b)}] The $\{\infty\}$-adelically metrized line bundle $\overline{L}^{\{\infty\}}$ satisfies the equivalent conditions in Proposition~\ref{prop:adelicapproximation}(2).
\item[\textup{(c)}] The curvature current of $(L_{\infty}^{\rm an},|\cdot|_{\infty}^{\overline{L}})$ is semipositive.
\item[\textup{(d)}] For every $x\in X(\overline{K})$, the height $h_{\overline{L}}(x)$ is non-negative.
\end{enumerate}
An adelically metrized $\QQ$-line bundle $\overline{L}$ is said to be \emph{nef} if some multiple of $\overline{L}$ is nef.
\item[(integrable)] We say that $\overline{L}\in\aPic_{\QQ}(X)$ is \emph{integrable} if $\overline{L}$ is a difference of two nef adelically metrized $\QQ$-line bundles.
Denote the $\QQ$-vector space of all the integrable adelically metrized $\QQ$-line bundles on $X$ by $\aInt(X)$.
\end{description}
\end{definition}

\begin{proposition}\label{prop:aintnum}
\begin{enumerate}
\item[\textup{(1)}] There exists a unique map
\begin{gather*}
 \adeg:\aPic_{\QQ}(X)\times\aInt(X)^{\times \dim X}\to\RR, \\
 (\overline{L}_0;\overline{L}_1\dots,\overline{L}_{\dim X})\to\adeg(\overline{L}_0\cdot\overline{L}_1\cdots\overline{L}_{\dim X}), \nonumber
\end{gather*}
having the following properties.
\begin{enumerate}
\item[\textup{(a)}] $\adeg$ is multilinear and the restriction $\adeg:\aInt(X)^{\times (\dim X+1)}\to\RR$ is symmetric.
\item[\textup{(b)}] If $\overline{L}\in\aPic_{\QQ}(X)$ is nef, then $\adeg(\overline{L}^{\cdot (\dim X+1)})=\avol(\overline{L})$.
\item[\textup{(c)}] If $\overline{L}_0$ is pseudoeffective and $\overline{L}_1,\dots,\overline{L}_{\dim X}$ are nef, then
\[
 \adeg(\overline{L}_0\cdot\overline{L}_1\cdots\overline{L}_{\dim X})\geq 0.
\]
\end{enumerate}
\item[\textup{(2)}] For $\lambda\in\RR$ and for $\overline{L}_1,\dots,\overline{L}_{\dim X}\in\aInt(X)$, we have
\[
 \adeg\left(\overline{\mathcal{O}}_X(\lambda[\infty])\cdot\overline{L}_1\cdots\overline{L}_{\dim X}\right)=[K:\QQ]\lambda\deg_K(L_1\cdots L_{\dim X}).
\]
\item[\textup{(3)}] Suppose that $X$ is normal, and let $\varphi:X'\to X$ be a birational projective $K$-morphism.
Then
\[
\xymatrix{ \aPic_{\QQ}(X')\times\aInt(X')^{\times \dim X} \ar[r]^-{\adeg} & \RR \\ \aPic_{\QQ}(X)\times\aInt(X)^{\times \dim X} \ar[r]^-{\adeg} \ar[u]^-{\varphi^{*\times (\dim X+1)}} & \RR \ar@{=}[u]
}
\]
is commutative.
\end{enumerate}
\end{proposition}

\begin{remark}
\begin{enumerate}
\item[\textup{(1)}] The above map $\adeg$ gives a unique extension of the classical arithmetic intersection numbers of $C^{\infty}$-Hermitian line bundles.
\item[\textup{(2)}] If $\overline{L}_0,\dots,\overline{L}_{\dim X}$ and $\overline{M}_0,\dots,\overline{M}_{\dim X}$ are nef adelically metrized $\QQ$-line bundle on $X$ and $\overline{L}_i\preceq\overline{M}_i$ for every $i$, then, by successively use of the property (c) of Proposition~\ref{prop:aintnum}(1), we have
\[
 \adeg(\overline{L}_0\cdots\overline{L}_{\dim X})\leq\adeg(\overline{M}_0\cdots\overline{M}_{\dim X}).
\]
\end{enumerate}
\end{remark}

\begin{proof}
(1): We refer to \cite[\S 4.5]{Moriwaki13}, \cite[Proof of Lemma~2.6]{IkomaCon}, and \cite[Proposition~4.5.4]{Moriwaki13}.

(2): If $\overline{L}_i$ are associated to $C^{\infty}$-Hermitian line bundles on some $O_K$-model of $X$, then the assertion is clear.
In general, we can show the result by approximation (Proposition~\ref{prop:adelicapproximation}).

(3): If $\overline{L}_i$ are associated to continuous Hermitian line bundles on an $O_K$-model of $X$, then the assertion follows from the projection formula \cite[Proposition~2.4.1]{Kawaguchi_Moriwaki} (see also \cite[Lemma~2.3]{IkomaCon}).
In general, we can assume that  $\overline{L}_0,\dots,\overline{L}_{\dim X}$ are nef and approximate them by nef continuous Hermitian line bundles on a suitable $O_K$-model of $X$.
\end{proof}

\section{Basic properties of base loci}\label{sec:prelim}

In this section, we collect some elementary properties of the augmented base loci of general graded linear series.
Let $X$ be a projective variety over a field $k$, and let $L$ be a line bundle on $X$.
Recall that the base locus, the stable base locus, and the augmented base locus of $L$ are respectively defined as
\[
 \Bs(L):=\Bs\Hz(L),\quad\SBs(L):=\bigcap_{m\geq 1}\Bs(mL),\quad\text{and}\quad\Bsp(L):=\bigcap_{a\geq 1}\SBs(aL-A),
\]
where $A$ is a fixed ample line bundle on $X$ and $\Bsp(L)$ does not depend on the choice of $A$.
By homogeneity, we can define the stable base locus and the augmented base locus for all $\QQ$-line bundles on $X$.

To a $k$-linear series $V\subset\Hz(L)$, we can associate a $k$-morphism denoted by
\begin{equation}
 \Phi_{V}:X\setminus\Bs V\to\PP_k(V).
\end{equation}
A \emph{graded $k$-linear series $V_{\sbullet}=\bigoplus_{m\geq 0}V_m$ belonging to $L$} is a sub-graded $k$-algebra of $\bigoplus_{m\geq 0}\Hz(mL)$.
Let $V_{\sbullet}$ be a graded $k$-linear series belonging to $L$, let $A$ be a line bundle on $X$, and let $a\geq 1$ be an integer.
We define a $k$-linear series belonging to $aL-A$ as
\begin{equation}
 \Lambda(V_{\sbullet};A,a):=\left\langle s\,:\,\begin{array}{l} \Image\left(\Hz(pA)\xrightarrow{\otimes s^{\otimes p}}\Hz(paL)\right)\subset V_{pa}, \\ \text{for every sufficiently large $p$}\end{array}\right\rangle_k.
\end{equation}
It is clear that $\Bs\Lambda(V_{\sbullet};nA,na)\subset\Bs\Lambda(V_{\sbullet};A,a)$ for every $n\geq 1$.
Thus, if we set
\begin{equation}
 \SBs(V_{\sbullet};A,a):=\bigcap_{n\geq 1}\Bs\Lambda(V_{\sbullet};nA,na),
\end{equation}
then $\SBs(V_{\sbullet};A,a)=\Bs\Lambda(V_{\sbullet};nA,na)$ for every sufficiently divisible $n$.

Let $\mu:X'\to X$ be a birational $k$-morphism of projective $k$-varieties and let $V_{\sbullet}$ be a graded $k$-linear series belonging to a line bundle $L$ on $X$.
If $X$ is normal, then $\Hz(mL)\overset{\mu^*}{=}\Hz(\mu^*(mL))$ for every $m\geq 1$.
We define the \emph{pull-back} of $V_{\sbullet}$ via $\mu$ as
\begin{equation}
 \mu^*V_{\sbullet}:=\bigoplus_{m\geq 0}\Image\left(V_m\xrightarrow{\mu^*}\Hz(\mu^*(mL))\right).
\end{equation}

\begin{definition}[\text{see \cite[Definition~2.2]{ChenSesh}}]\label{defn:Chens}
We define the \emph{augmented base locus} of $V_{\sbullet}$ as
\begin{equation}
 \Bsp(V_{\sbullet}):=\bigcap_{\substack{\text{$A$: ample,} \\ a\geq 1}}\Bs\Lambda(V_{\sbullet};A,a),
\end{equation}
where the intersection is taken over all the ample line bundles $A$ on $X$ and all the positive integers $a$.
\end{definition}

\begin{lemma}\label{lem:augbs}
Let $V_{\sbullet}$ be a graded linear series belonging to $L$.
\begin{enumerate}
\item[\textup{(1)}] If $A,B$ are two line bundles on $X$, then for $a,b\geq 1$
\[
 \SBs(V_{\sbullet};A,a)\subset\SBs(V_{\sbullet};B,b)\cup\SBs((1/b)B-(1/a)A).
\]
\item[\textup{(2)}] For any ample line bundle $A$, $\Bsp(V_{\sbullet})=\SBs(V_{\sbullet};A,a)$ holds for every $a\gg 1$.
\item[\textup{(3)}] If $A$ is semiample, then $\Bs\Lambda(V_{\sbullet};A,a)\supset\SBs(V_{\sbullet})$ for every $a\geq 1$.
\item[\textup{(4)}] Let $\mu:X'\to X$ be a birational $k$-morphism of projective $k$-varieties.
If $X$ is normal, then
\[
 \Bs\Lambda(\mu^*V_{\sbullet};\mu^*A,a)=\mu^{-1}\Bs\Lambda(V_{\sbullet};A,a).
\]
\item[\textup{(5)}] The following are equivalent.
\begin{enumerate}
\item[\textup{(a)}] $L$ is ample and $V_m=\Hz(mL)$ for every sufficiently divisible $m\gg 1$.
\item[\textup{(b)}] For some $a\geq 1$, $\Phi_{V_a}:X\to\PP_k(V_a)$ is a closed immersion.
\item[\textup{(c)}] $\Bsp(V_{\sbullet})=\emptyset$.
\end{enumerate}
\end{enumerate}
\end{lemma}

\begin{proof}
(1): Suppose that $x\notin\SBs(V_{\sbullet};B,b)\cup\SBs((1/b)B-(1/a)A)$.
There exists an $n\geq 1$ such that $x\notin\Bs(n(aB-bA))$ and $x\notin\Bs\Lambda(V_{\sbullet};naB,nab)$.
We can find an $s\in\Lambda(V_{\sbullet};naB,nab)$ and a $t\in\Hz(n(aB-bA))$ such that $s(x)\neq 0$ and $t(x)\neq 0$.
Set $s':=s\otimes t\in\Hz(nb(aL-A))$.
Then $s'(x)\neq 0$ and
\[
 \Image\left(\Hz(pnbA)\xrightarrow{\otimes t^{\otimes p}}\Hz(pnaB)\xrightarrow{\otimes s^{\otimes p}}\Hz(pnabL)\right)\subset V_{pnab}
\]
for every $p\gg 1$.

(2): Since $X$ is a Noetherian topological space, there exist ample line bundles $B_1,\dots,B_r$ and positive integers $b_1,\dots,b_r$ such that
\[
 \Bsp(V_{\sbullet})=\bigcap_{i=1}^r\Bs\Lambda(V_{\sbullet};B_i,b_i).
\]
We can find an $a_0\geq 1$ such that $(1/b_i)B_i-(1/a_0)A$ are ample for all $i$.
Then by the assertion (1), we have $\Bsp(V_{\sbullet})=\Bs\Lambda(V_{\sbullet};A,a)$ for every $a\geq a_0$.

(3): Suppose that $x\notin\Bs\Lambda(V_{\sbullet};A,a)$.
Since $A$ is semiample, there exist a $p\geq 1$, a $t\in\Hz(pA)$, and an $s\in\Hz(aL-A)$ such that $t(x)\neq 0$, $s(x)\neq 0$, and $t\otimes s^{\otimes p}\in V_{pa}$.
Thus $x\notin\SBs(V_{\sbullet})$.
The assertion (4) is clear.

(5): The implications (a) $\Rightarrow$ (b) and (a) $\Rightarrow$ (c) are clear.

(b) $\Rightarrow$ (a) Let $P:=\PP_k(V_a)$ and let $\mathcal{O}_P(1)$ be the hyperplane line bundle on $P$.
Since $aL=\Phi_{V_a}^*\mathcal{O}_P(1)$, $\Hz(\mathcal{O}_P(p))=\Sym_k^pV_a\to\Hz(paL)$ is surjective for every $p\gg 1$.
Thus $V_{pa}=\Hz(paL)$ for every $p\gg 1$.

(c) $\Rightarrow$ (b) One can find a very ample line bundle $A$ and an $a\geq 1$ such that
\[
 \Bsp(V_{\sbullet})=\Bs\Lambda(V_{\sbullet};A,a)=\emptyset.
\]
There exist $s_0,\dots,s_r\in\Hz(aL-A)$ and a $p\geq 1$ such that
\[
 \left\{x\in X\,:\,s_0(x)=\dots=s_r(x)=0\right\}=\emptyset
\]
and
\[
 \Image\left(\Hz(pA)\xrightarrow{\otimes s_i^{\otimes p}}\Hz(paL)\right)\subset V_{pa}
\]
for all $i$.
Denote the $k$-morphism defined by $s_0^{\otimes p},\dots,s_r^{\otimes p}$ by $\Phi:X\to\PP_k^r$.
Since $X\overset{\Phi_{pA}\times\Phi}{\longrightarrow}\PP_k(\Hz(pA))\times_k\PP_k^r\xrightarrow{\text{Segre emb.}}\PP_k(\Hz(pA)\otimes_k\aSpan{k}{s_0^{\otimes p},\dots,s_r^{\otimes p}})$ is a closed immersion, so is $\Phi_{V_{pa}}$.
\end{proof}

Let $\mu:X'\to X$ be a morphism of $k$-varieties.
The \emph{exceptional locus} of $\mu$ is defined as the minimal Zariski closed subset $\Exc(\mu)\subset X'$ such that
\begin{equation}
 \mu|_{X'\setminus \Exc(\mu)}:X'\setminus\Exc(\mu)\to X
\end{equation}
is an immersion.
If $\mu:X'\to \overline{\mu(X')}$ is not birational, then $\Exc(\mu)$ is defined as $X'$.

\begin{lemma}
Let $\mu:X'\to X$ be a birational projective $k$-morphism of $k$-varieties.
If $X$ is normal, then the exceptional locus of $\mu$ is given as
\[
 \Exc(\mu)=\bigcup_{\substack{Z'\subset X', \\ \dim Z'>\dim\mu(Z')}}Z'.
\]
\end{lemma}

\begin{proof}
The inclusion $\supset$ is clear, so we are going to show the reverse.
By \cite[Proposition~(4.4.1)]{EGAIII_1}, one has
\[
 \Exc(\mu)=\left\{x'\in X'\,:\,\dim_{x'}\mu^{-1}\left(\mu(x')\right)\geq 1\right\}
\]
and $\mu^{-1}\left(\mu\left(\Exc(\mu)\right)\right)=\Exc(\mu)$.
Given any closed point $x'\in\Exc(\mu)$, there exists an irreducible component $Z'$ of $\mu^{-1}\left(\mu(x')\right)$ such that $Z'$ passes through $x'$ and $\dim Z'\geq 1$.
Since $\mu(x')$ is a closed point of $X$, one has $\dim Z'>\dim\mu(Z')$.
\end{proof}

\begin{lemma}\label{lem:augbs_main}
Suppose that $X$ is normal.
Let $V_{\sbullet}$ be a graded $k$-linear series belonging to a line bundle $L$ on $X$.
\begin{enumerate}
\item[\textup{(1)}] $\Bsp(V_{\sbullet})=\bigcap_{\mu:X'\to X}\mu\left(\Bsp(\mu^*V_{\sbullet})\right)$, where the intersection is taken over all the projective birational $k$-morphisms $\mu$ onto $X$.
\item[\textup{(2)}] Let $\mu:X'\to X$ be a birational $k$-morphism of projective $k$-varieties.
Then
\[
 \Bsp(\mu^*V_{\sbullet})=\mu^{-1}\Bsp(V_{\sbullet})\cup\Exc(\mu).
\]
\end{enumerate}
\end{lemma}

\begin{proof}
(1): The inclusion $\supset$ is obvious by definition.
The reverse follows from the following claim.

\begin{claim}\label{clm:inclusion}
For any $\mu:X'\to X$, we have $\Bsp(V_{\sbullet})\subset\mu\left(\Bsp(\mu^*V_{\sbullet})\right)$.
\end{claim}

\begin{proof}[Proof of Claim~\ref{clm:inclusion}]
Let $A$ be an ample line bundle on $X$.
Suppose that $x\notin\mu\left(\SBs(\mu^*V_{\sbullet};A',a)\right)$ for a positive integer $a$ and an ample line bundle $A'$ on $X'$ such that $A'-\mu^*A$ is ample.
Since
\[
 \mu^{-1}\SBs(V_{\sbullet};A,a)=\SBs(\mu^*V_{\sbullet};\mu^*A,a)\subset\SBs(\mu^*V_{\sbullet};A',a)
\]
by Lemma~\ref{lem:augbs}(1),(4), we have $x\notin\SBs(V_{\sbullet};A,a)$.
\end{proof}

(2): Suppose that $x'\notin\SBs(\mu^*V_{\sbullet};A',a)$ for an ample line bundle $A'$ on $X'$ and an $a\geq 1$.
Since $\SBs(\mu^*V_{\sbullet};A',a)\supset\SBs\left(\mu^*L-\frac{1}{a}A'\right)\supset\Exc(\mu)$ (see \cite[Lemma~3.39(2)]{Kollar_Mori}), $x'\notin\Exc(\mu)$.
Thus by the assertion (1), we have $\mu(x')\notin\Bsp(V_{\sbullet})$.

To show the reverse, we assume that $x'\notin\mu^{-1}\Bsp(V_{\sbullet})\cup\Exc(\mu)$.
By \cite[Proposition~2.3]{Bou_Bro_Pac13}, we can find an ample line bundle $A$ on $X$, an ample line bundle $A'$ on $X'$, and an $a\geq 1$ such that $x'\notin\Bs\left(\mu^*A-A'\right)=\Exc(\mu)$ and that $x'\notin\SBs(\mu^*V_{\sbullet};\mu^*A,a)=\mu^{-1}\Bsp(V_{\sbullet})$.
Thus we have $x'\notin\SBs(\mu^*V_{\sbullet};A',a)$ by Lemma~\ref{lem:augbs}(1).
\end{proof}

\begin{lemma}\label{lem:excep}
\begin{enumerate}
\item[\textup{(1)}] For positive integers $m,n$, we have
\[
 \Exc(\Phi_{V_{mn}})\cup\Bs(V_{mn})\subset\Exc(\Phi_{V_m})\cup\Bs(V_m).
\]
\item[\textup{(2)}] For every sufficiently divisible $n$, we have
\[
 \Exc(\Phi_{V_n})\cup\Bs(V_n)=\bigcap_{m\geq 1}\left(\Exc(\Phi_{V_m})\cup\Bs(V_m)\right).
\]
\end{enumerate}
\end{lemma}

\begin{proof}
(1): Set $Q:=\Coker\left(V_m^{\otimes n}\to V_{mn}\right)$.
Considering the commutative diagram
\[
\xymatrix{ X\setminus\left(\Exc(\Phi_{V_m})\cup\Bs(V_m)\right) \ar[rr]^-{\Phi_{V_{mn}}} \ar[d]_-{\Phi_{V_m}^{\times n}} && \PP_k(V_{mn})\setminus\PP_k(Q) \ar[d] \\ \PP_k(V_m)^{\times n} \ar[rr]^-{\text{Segre emb.}} && \PP_k(V_m^{\otimes n}),
}
\]
we know that $\Phi_{V_{mn}}:X\setminus\left(\Exc(\Phi_{V_m})\cup\Bs(V_m)\right)\to\PP_k(V_{mn})$ is an immersion.

(2): Since $X$ is a Noetherian topological space, one can find positive integers $m_1,\dots,m_r$ such that
\[
 \bigcap_{m\geq 1}\left(\Exc(\Phi_{V_m})\cup\Bs(V_m)\right)=\bigcap_{i=1}^r\left(\Exc(\Phi_{V_{m_i}})\cup\Bs(V_{m_i})\right).
\]
Thus by (1) we have
\[
 \bigcap_{m\geq 1}\left(\Exc(\Phi_{V_m})\cup\Bs(V_m)\right)=\Exc(\Phi_{V_{am_1\cdots m_r}})\cup\Bs(V_{am_1\cdots m_r})
\]
for every $a\geq 1$.
\end{proof}

\begin{lemma}\label{lem:Kodairatype}
For any graded $k$-linear series $V_{\sbullet}$ belonging to $L$, we have
\[
 \Bsp(V_{\sbullet})\supset\bigcap_{m\geq 1}\left(\Exc(\Phi_{V_m})\cup\Bs(V_m)\right).
\]
\end{lemma}

\begin{proof}
We choose a very ample line bundle $A$ on $X$ and an $a\geq 1$ such that $\SBs(V_{\sbullet})=\Bs(V_a)$ and $\Bsp(V_{\sbullet})=\Bs\Lambda(V_{\sbullet};A,a)$.
Fix a basis $s_0,\dots,s_r$ for $\Lambda(V_{\sbullet};A,a)$ and take a $p\geq 1$ such that
\[
 \Image\left(\Hz(pA)\xrightarrow{\otimes s_i^{\otimes p}}\Hz(paL)\right)\subset V_{pa}
\]
for all $i$.
Denote the $k$-morphism defined by $s_0^{\otimes p},\dots,s_r^{\otimes p}$ by $\Phi:X\setminus\Bs\Lambda(V_{\sbullet};A,a)\to\PP_k^r$ and set $Q:=\Coker\left(\Hz(pA)\otimes_k\aSpan{k}{s_0^{\otimes p},\dots,s_r^{\otimes p}}\to V_{pa}\right)$.
Considering a commutative diagram
\[
\xymatrix{ X\setminus\Bs(V_{pa}) \ar[rr]^-{\Phi_{V_{pa}}} & & \PP_k(V_{pa})\setminus\PP_k(Q) \ar[d] \\
 X\setminus\Bs\Lambda(V_{\sbullet};A,a) \ar[r]^-{\Phi_{pA}\times\Phi} \ar[u] & \PP_k(\Hz(pA))\times_k\PP_k^r \ar[r]^-{\text{Segre}} & \PP_k(\Hz(pA)\otimes_k\aSpan{k}{s_0^{\otimes p},\dots,s_r^{\otimes p}})
}
\]
we have that $\Phi_{V_{pa}}$ is an immersion over $X\setminus\Bsp(V_{\sbullet})$.
\end{proof}

The following was proved in \cite{Ein_Laz_Mus_Nak_Pop06,Bou_Cac_Lop13}.

\begin{theorem}\label{thm:charaugbs}
Suppose that $X$ is normal, and let $L$ be a line bundle on $X$.
\begin{enumerate}
\item[\textup{(1)}] $\Bsp(L)$ is characterized as the minimal Zariski closed subset of $X$ such that the restriction of the Kodaira map
\[
 \Phi_{mL}:X\setminus\Bs(mL)\to\PP_k(\Hz(mL))
\]
to $X\setminus\Bsp(L)$ is an immersion for every sufficiently divisible $m\gg 1$.
\item[\textup{(2)}] We have
\[
 \Bsp(L)=\bigcup_{\substack{Z\subset X, \\ \volq{X|Z}(L)=0}}Z.
\]
\end{enumerate}
\end{theorem}

\begin{proof}
If $k=\Hz(\mathcal{O}_X)$, then by passing to the algebraic closure $\overline{k}$, we can apply \cite[Theorems~A and B]{Bou_Cac_Lop13}.
In general, since the natural morphism $\PP_{\Hz(\mathcal{O}_X)}(\Hz(mL))\to\PP_k(\Hz(mL))\times_k\Spec(\Hz(\mathcal{O}_X))$ is a closed immersion, we can easily deduce the result from the case of $k=\Hz(\mathcal{O}_X)$.
\end{proof}

\begin{remark}
In Theorem~\ref{thm:charaaugbs} and Corollary~\ref{cor:propertiesarestvol}, we show analogous results of Theorem~\ref{thm:charaugbs} under the conditions that $X$ is a normal projective variety over a number field $K$ and $V_{\sbullet}=\bigoplus_{m\geq 0}\aSpan{K}{\aHzst(m\overline{L})}$ for an $\overline{L}\in\aPic(X)$.
\end{remark}

\section{A result of Zhang--Moriwaki}\label{sec:reprove}

In this section, we give a simple proof to the Zhang--Moriwaki theorem (see \cite[Corollary~B]{MoriwakiFree} and Theorem~\ref{thm:reprove} below).
The proof is independent of the previous sections except Lemma~\ref{lem:Kodairatype}.
Our method can recover \cite[Theorem~A]{MoriwakiFree} but not so powerful as recovering the arithmetic Nakai-Moishezon criterion (see \cite[Theorem~3.1]{MoriwakiFree}).
Let $X$ be a projective variety over a number field $K$, and let $L$ be a line bundle on $X$.

\begin{lemma}\label{lem:Veronese}
Let $\overline{V}_{\sbullet}$ be an adelically normed graded $K$-linear series belonging to $L$.
If $V_{\sbullet}$ is Noetherian, then the following are equivalent.
\begin{enumerate}
\item[\textup{(1)}] $V_m=\aSpan{K}{\aHzst(\overline{V}_m)}$ for every $m\gg 1$.
\item[\textup{(2)}] For an integer $a\geq 1$, $V_{ma}=\aSpan{K}{\aHzst(\overline{V}_{ma})}$ for every $m\gg 1$.
\end{enumerate}
\end{lemma}

\begin{proof}
The implication (1) $\Rightarrow$ (2) is clear and we are going to show (2) $\Rightarrow$ (1).

\begin{claim}\label{clm:Veronese}
The Veronese subalgebra
\begin{equation}
 V_{\sbullet}^{(a)}:=\bigoplus_{m\geq 0}V_{ma}
\end{equation}
is also Noetherian and $V_{\sbullet}$ is a finitely generated $V_{\sbullet}^{(a)}$-module.
\end{claim}

\begin{proof}
This is well-known (see \cite[Chap.\ III, \S 1.3, Proposition~2]{BourbakiCA72}).
\end{proof}

By \cite[Chap.\ III, \S 1.3, Proposition~3]{BourbakiCA72}, we can replace $a$ with a high multiple of $a$ and assume that $V_{\sbullet}^{(a)}$ is generated by $V_a$ over $V_0$ and $V_a=\aSpan{K}{\aHzst(\overline{V}_{a})}$.

Let $\alpha_1,\dots,\alpha_p\in V_0$ be generators of $V_0$ over $K$ and let
\[
 \aHzst(\overline{V}_{a})\setminus\{0\}=\left\{s_1,\dots,s_q\right\}.
\]
By Claim~\ref{clm:Veronese}, we can take generators $t_1\in V_{n_1},\dots,t_r\in V_{n_r}$ of $V_{\sbullet}$ as a $V_{\sbullet}^{(a)}$-module.
Then
\[
 V_m=\sum_{\substack{a(i_1+\dots+i_q)+n_j=m}}(K\alpha_1+\dots+K\alpha_q)s_1^{\otimes i_1}\otimes\cdots\otimes s_q^{\otimes i_q}\otimes t_j
\]
for every $m\gg 1$.
If $i_1+\dots+i_q$ is sufficiently large, we can see that
\[
 s_1^{\otimes i_1}\otimes\cdots\otimes s_q^{\otimes i_q}\otimes (\alpha_kt_j)\in\aSpan{K}{\aHzst(\overline{V}_m)}
\]
for every $k,j$.
So we have $V_m=\aSpan{K}{\aHzst(\overline{V}_m)}$ for every $m\gg 1$.
\end{proof}

\begin{theorem}\label{thm:reprove}
Let $X$ be a projective variety over a number field $K$, and let $L$ be a line bundle on $X$.
Let $\overline{V}_{\sbullet}$ be an adelically normed graded $K$-linear series belonging to $L$.
Suppose the following.
\begin{enumerate}
\item[\textup{(a)}] $V_{\sbullet}$ is Noetherian.
\item[\textup{(b)}] $\aSBs(\overline{V}_{\sbullet})=\emptyset$.
\end{enumerate}
Then $V_m=\aSpan{K}{\aHzst(\overline{V}_m)}$ for every $m\gg 1$.
\end{theorem}

\begin{proof}
Thanks to Lemma~\ref{lem:Veronese}, we can assume $\Bs(V_m)=\aBs(\overline{V}_m)=\emptyset$ for every $m\geq 1$.
We divide the proof into two steps.
\medskip

\paragraph{Step 1.}
First, we suppose that $V_{\sbullet}$ is generated by $V_1$ over $K$, so, in particular, $V_0=K$.
Let $P:=\PP_K(V_1)$, let $Y:=\Phi_{V_1}(X)$ be the image, and decompose the morphism $\Phi_{V_1}$ into $X\xrightarrow{\Phi}Y\to P$.
Let $\mathcal{O}_P(1)$ (respectively $\mathcal{O}_Y(1)$) be the hyperplane line bundle on $P$ (respectively $Y$).
For every $m\geq 1$, the image of the homomorphism
\begin{equation}\label{eqn:reprove1}
 \Phi_{V_1}^*:\Hz(\mathcal{O}_P(m))=\Sym_K^mV_1\xrightarrow{\text{rest.}}\Hz(\mathcal{O}_Y(m))\xrightarrow{\Phi^*}\Hz(mL)
\end{equation}
coincides with $V_{m}$.

For each $m\geq 1$, let
\begin{equation}\label{eqn:reprove2}
 W_m:=(\Phi^*)^{-1}\left(\aSpan{K}{\aHzst(\overline{V}_{m})}\right)
\end{equation}
be the $K$-subspace of $\Hz(\mathcal{O}_Y(m))$.
Then $W_{\sbullet}$ forms a graded $K$-linear series belonging to $\mathcal{O}_Y(1)$.
We can choose an $a\geq 1$ such that the restriction
\begin{equation}\label{eqn:reprove3}
 \Hz(\mathcal{O}_P(m))\twoheadrightarrow\Hz(\mathcal{O}_Y(m))
\end{equation}
is surjective for every $m\geq a$ and the homomorphism
\begin{equation}\label{eqn:reprove4}
 \Sym_K^p\Hz(\mathcal{O}_Y(a))\twoheadrightarrow\Hz(\mathcal{O}_Y(pa))
\end{equation}
is surjective for every $p\geq 1$.

\begin{claim}\label{clm:reprove1}
For any $y\in Y$, there exists an $s\in W_1$ such that $s(y)\neq 0$ and $\Phi^*s\in\aHzst(\overline{V}_1)$.
\end{claim}

\begin{proof}[Proof of Claim~\ref{clm:reprove1}]
We can find an $s_0\in\Hz(\mathcal{O}_P(1))$ such that $s_0(y)\neq 0$ and $\Phi_{V_1}^*s_0\in\aHzst(\overline{V}_1)$.
Then the restriction $s:=s_0|_Y$ has the desired properties.
\end{proof}

\begin{claim}\label{clm:reprove2}
Let $y\in Y$ and $s\in W_1$ be as in Claim~\ref{clm:reprove1}.
There exists an integer $b\geq 1$ such that
\[
 \Image\left(\Hz(\mathcal{O}_Y(pa))\xrightarrow{\otimes s^{\otimes pb}}\Hz(\mathcal{O}_Y(p(a+b)))\right)\subset W_{p(a+b)}
\]
for every $p\gg 1$.
In other words, $s^{\otimes b}\in\Lambda(W_{\sbullet};\mathcal{O}_Y(a),a+b)$.
\end{claim}

\begin{proof}[Proof of Claim~\ref{clm:reprove2}]
Fix a $K$-basis $e_1,\dots,e_r$ for $\Hz(\mathcal{O}_Y(a))$.
By (\ref{eqn:reprove1}) and (\ref{eqn:reprove3}), $\Phi^*e_j\in V_{a}$ for every $j$.
Hence we can find a $b\geq 1$ such that
\[
 \Phi^*(e_j\otimes s^{\otimes b})=(\Phi^*e_j)\otimes (\Phi^*s)^{\otimes b}\in\aSpan{K}{\aHzst(\overline{V}_{a+b})}
\]
for every $j$.
Since by (\ref{eqn:reprove4}) $\Sym_K^p\Hz(\mathcal{O}_Y(a))\to\Hz(\mathcal{O}_Y(pa))$ is surjective for every $p$, we have the claim.
\end{proof}

By Claim~\ref{clm:reprove2}, we have $y\notin\Bs\Lambda(W_{\sbullet};\mathcal{O}_Y(a),a+b)$, and $\Bsp(W_{\sbullet})=\emptyset$.
By Lemma~\ref{lem:Kodairatype}, a $K$-morphism $Y\to Q:=\PP_K(W_{c})$ associated to $W_{c}$ is a closed immersion for some $c\geq 1$.
Let $\mathcal{O}_Q(1)$ be the hyperplane line bundle on $Q$.
Since the upper arrow of the diagram
\[
\xymatrix{
 \Hz(\mathcal{O}_Q(m))=\Sym_K^m(W_{c}) \ar[rr]^-{\rm rest.} \ar[d] && \Hz(\mathcal{O}_Y(mc)) \ar@{->>}[d]^-{\Phi^*} \\
 \aSpan{K}{\aHzst(\overline{V}_{mc})} \ar[rr] && V_{mc}
}
\]
is surjective for every $m\gg 1$, we have $\aSpan{K}{\aHzst(\overline{V}_{mc})}=V_{mc}$ for every $m\gg 1$.
By using Lemma~\ref{lem:Veronese} again, we conclude.
\medskip

\paragraph{Step 2.}
Next, we consider the general case.
By \cite[Chap.\ III, \S 1.3, Proposition~3]{BourbakiCA72} and Lemma~\ref{lem:Veronese}, we can assume without loss of generality that $V_{\sbullet}$ is generated by $V_1$ over $V_0$.
Let $W_{\sbullet}$ be the sub-graded $K$-algebra of $V_{\sbullet}$ generated by $V_1$.
Each $W_m$ ($m\geq 0$) is endowed with the subspace norms $(\|\cdot\|_v^{\overline{W}_m})_{v\in M_K}$ induced from $(\|\cdot\|_v^{\overline{V}_{m}})_{v\in M_K}$.

We choose a suitably small $\varepsilon>0$ and replace $\|\cdot\|_{\infty}^{\overline{W}_m}$ with $\exp(\varepsilon m)\|\cdot\|_{\infty}^{\overline{W}_m}$.
By applying the above arguments, we can find an $m_0\geq 1$ such that, for every $m\geq m_0$, $W_m$ is generated over $K$ by its strictly small sections with $\|\cdot\|_{\infty}^{\overline{W}_m}<\exp(-\varepsilon m)$.
Let $\alpha_1,\dots,\alpha_p\in\aHzf(\overline{V}_0)$ be generators of $V_0$ over $K$.
Then
\[
 V_m=(K\alpha_1+\dots+K\alpha_p)W_m
\]
is generated by its strictly small sections for every $m$ with
\[
 m\geq \max\left\{m_0,\frac{\log\|\alpha_1\|_{\infty}^{\overline{V}_0}}{\varepsilon},\dots,\frac{\log\|\alpha_p\|_{\infty}^{\overline{V}_0}}{\varepsilon}\right\}.
\]
\end{proof}

\begin{remark}
There are many graded linear series $V_{\sbullet}$ such that $V_{\sbullet}$ are Noetherian, $\SBs(V_{\sbullet})=\emptyset$, and $V_m\neq\Hz(mL)$ for every $m\geq 1$.
Suppose that a line bundle $L$ on $X$ is free and let $\mu:Y\to X$ be a morphism such that $\Hz(mL)\to\Hz(\mu^*(mL))$ is not surjective for every $m\geq 1$.
Then $\bigoplus_{m\geq 0}\mu^*(\Hz(mL))$ gives an example.
\end{remark}

In the rest of this section, we apply Theorem~\ref{thm:reprove} to the case of adelically metrized line bundles (Corollary~\ref{cor:reproveZM}).
Suppose that $X$ is normal.
Let $\overline{L}=\left(L,(|\cdot|_v^{\overline{L}})_{v\in M_K}\right)$ be an adelically metrized line bundle on $X$ such that $\aHzst(m\overline{L})\neq\{0\}$ for some $m\geq 1$.
For each $m\geq 1$, let
\begin{equation}
 \widehat{\mathfrak{b}}_m:=\Image\left(\aSpan{K}{\aHzst(m\overline{L})}\otimes_K(-mL)\to\mathcal{O}_X\right)
\end{equation}
be the ideal sheaf on $X$.
Let $\mu_m:X_m\to X$ be the normalized blow-up along $\widehat{\mathfrak{b}}_m$, let $F_m:=\SHom_{\mathcal{O}_{X_m}}(\widehat{\mathfrak{b}}_m\mathcal{O}_{X_m},\mathcal{O}_{X_m})$, and let $1_{F_m}\in\Hz(F_m)$ be the natural inclusion.

\begin{proposition}\label{prop:adelicmetric}
For each $m\geq 1$, we can endow $F_m$ with an adelic metric $(|\cdot|_v^{\overline{F}_m})_{v\in M_K}$ such that, for each $v\in M_K^{\rm f}$ and each $x\in X_{m,v}^{\rm an}$,
\begin{equation}\label{eqn:adelicmetricfin}
 |1_{F_m}|^{\overline{F}_m}_v(x)=\max_{s\in\aHzst(m\overline{L})}\left\{|s|^{m\overline{L}}_v(\mu_{m,v}^{\rm an}(x))\right\}
\end{equation}
and, for each $x\in X_{m,\infty}^{\rm an}$,
\begin{equation}\label{eqn:adelicmetriccomp}
 |1_{F_m}|^{\overline{F}_m}_{\infty}(x)=\max_{s\in\aHzst(m\overline{L})}\left\{\frac{|s|^{m\overline{L}}_{\infty}(\mu_{m,\infty}^{\rm an}(x))}{\|s\|_{\infty,\sup}^{m\overline{L}}}\right\}.
\end{equation}
We set $\overline{F}_m:=\left(F_m,(|\cdot|^{\overline{F}_m}_v)_{v\in M_K}\right)$ and $\overline{M}_m:=\mu_m^*(m\overline{L})-\overline{F}_m$.
\begin{enumerate}
\item[\textup{(1)}] We have
\begin{equation}\label{eqn:adelicmetricsscorresp}
 \aHzst(\overline{M}_m)\overset{\otimes 1_{F_m}}{=}\Image\left(\aHzst(m\overline{L})\to\Hz(\mu_m^*(mL))\right)\quad\text{and}\quad\aBs(\overline{M}_m)=\emptyset.
\end{equation}
\item[\textup{(2)}] $\overline{M}_m$ is a nef adelically metrized line bundle on $X_m$.
\end{enumerate}
\end{proposition}

\begin{proof}
We take an affine open covering $\{W_{\lambda}\}$ of $X_m$ such that $\mu_m^*(mL)|_{W_{\lambda}}$ is trivial with local frame $\eta_{\lambda}$, and $\Supp(1_{F_m})\cap W_{\lambda}$ is defined by a local equation $f_{\lambda}$.
Since any $s\in\aSpan{K}{\aHzst(m\overline{L})}$ satisfies $s\in\Hz(mL\otimes \widehat{\mathfrak{b}}_m)$, we can find an $s'\in\Hz(\mu_m^*(mL)-F_m)$ such that $\mu_m^*s=s'\otimes 1_{F_m}$.
So we can write $\mu_m^*s|_{W_{\lambda}}=\phi_{s,\lambda}\cdot f_{\lambda}\cdot\eta_{\lambda}$ on $W_{\lambda}$, where $\phi_{s,\lambda}$ is a regular function on $W_{\lambda}$ and $\{x\in W_{\lambda}\,|\,\phi_{s,\lambda}(x)=0,\,\forall s\in\aHzst(mL)\}=\emptyset$.
For each $v\in M_K^{\rm f}$ and each $x\in X_{m,v}^{\rm an}$,
\[
 \max_{s\in\aHzst(m\overline{L})}\left\{|s|^{m\overline{L}}_v(\mu_{m,v}^{\rm an}(x))\right\}=\max_{s\in\aHzst(m\overline{L})}\left\{|\phi_{s,\lambda}|(x)\right\}\cdot |f_{\lambda}|(x)\cdot |\eta_{\lambda}|_v^{\mu_m^*(m\overline{L})}(x).
\]
Therefore we can define $|1_{F_m}|^{\overline{F}_m}_v(x)$ by the formula (\ref{eqn:adelicmetricfin}), which defines a continuous metric on $F_{m,v}^{\rm an}$ (see section \ref{sec:adelicallymetrizedlbd}).
By the same way, we can show that the formula (\ref{eqn:adelicmetriccomp}) defines a continuous Hermitian metric on $F_{m,\infty}^{\rm an}$.

Let $(\mathscr{X}_U,\mathscr{L}_U)$ be a $U$-model of definition for $\overline{L}^{\{\infty\}}$ (Definition~\ref{defn:adelicallymetrizedlinebdl}).
Since $X$ is normal, we can assume that $\mathscr{X}_U$ is also normal.
By Proposition~\ref{prop:adelicapproximation}(1), given any $\varepsilon>0$, we can find an $O_K$-model $(\mathscr{X}_{\varepsilon},\mathscr{L}_{\varepsilon})$ such that $\mathscr{X}_{\varepsilon}$ is normal, $\mathscr{X}_{\varepsilon}\times_{\Spec(O_K)}U$ is $U$-isomorphic to $\mathscr{X}_U$, $\mathscr{L}_{\varepsilon}|_{\mathscr{X}_U}=\mathscr{L}_U$ as $\QQ$-line bundles, and
\begin{equation}
 |\cdot|^{\mathscr{L}_{\varepsilon,O_{K_P}}}(x)\leq |\cdot|_P^{\overline{L}}(x)\leq\exp(\varepsilon)|\cdot|^{\mathscr{L}_{\varepsilon,O_{K_P}}}(x)
\end{equation}
for $P\in\Spec(O_K)\setminus U$ and $x\in X_P^{\rm an}$.
For each $\varepsilon$, we fix an integer $n=n_{\varepsilon}\geq 1$ such that $n\mathscr{L}_{\varepsilon}$ is a line bundle on $\mathscr{X}_{\varepsilon}$.
By Lemma~\ref{lem:propertyadelicnorm}(3), we have $\aHzf(mn\overline{L})\subset\Hz(mn\mathscr{L}_{\varepsilon})$.
Let $\nu_{m,n,\varepsilon}:\mathscr{X}_{m,n,\varepsilon}\to\mathscr{X}_{\varepsilon}$ be the normalized blow-up along
\[
 \widehat{\mathfrak{b}}_{m,n,\varepsilon}:=\Image\left(\Sym_{O_K}^n\aSpan{O_K}{\aHzst(m\overline{L})}\otimes_{O_K}(-mn\mathscr{L}_{\varepsilon})\to\mathcal{O}_{\mathscr{X}_{\varepsilon}}\right).
\]
Note that $\widehat{\mathfrak{b}}_{m,n,\varepsilon}\mathcal{O}_X=(\widehat{\mathfrak{b}}_m)^n$.
Choose a normal $O_K$-model $\mathscr{X}_{m,\varepsilon}$ of $X_m$ dominating $\mathscr{X}_{m,n,\varepsilon}$, namely, we have a morphism
\begin{equation}\label{eqn:decomposemum}
 \mu_{m,\varepsilon}:\mathscr{X}_{m,\varepsilon}\xrightarrow{\psi_{m,n,\varepsilon}}\mathscr{X}_{m,n,\varepsilon}\xrightarrow{\nu_{m,n,\varepsilon}} \mathscr{X},
\end{equation}
whose restriction to the generic fiber is $\mu_m:X_m\to X$.
Let
\[
 \mathscr{F}_{m,n,\varepsilon}:=\SHom_{\mathcal{O}_{\mathscr{X}_{m,n,\varepsilon}}}(\widehat{\mathfrak{b}}_{m,n,\varepsilon}\mathcal{O}_{\mathscr{X}_{m,n,\varepsilon}},\mathcal{O}_{\mathscr{X}_{m,n,\varepsilon}})
\]
and let $\mathscr{F}_{m,\varepsilon}:=(1/n)\psi_{m,n,\varepsilon}^*\mathscr{F}_{m,n,\varepsilon}$ and $\mathscr{M}_{m,\varepsilon}:=\mu_{m,\varepsilon}^*(m\mathscr{L}_{\varepsilon})-\mathscr{F}_{m,\varepsilon}$ be the $\QQ$-line bundles on $\mathscr{X}_{m,\varepsilon}$.
Then $(\mathscr{X}_{m,\varepsilon},\mathscr{F}_{m,\varepsilon})$ is an $O_K$-model of $(X_m,F_m)$.

\begin{claim}\label{clm:adelicmetric1}
\begin{enumerate}
\item[\textup{(1)}] For every $P\in \Spec(O_K)$ and for every $x\in X_{m,P}^{\rm an}$,
\[
 |1_{F_m}|^{\mathscr{F}_{m,\varepsilon,O_{K_P}}}(x)=\max_{s\in\aHzst(m\overline{L})}\left\{|s|^{m\mathscr{L}_{\varepsilon,O_{K_P}}}(\mu_{m,P}^{\rm an}(x))\right\}.
\]
\item[\textup{(2)}] For every $P\in U$ and for every $x\in X_{m,P}^{\rm an}$, $|\cdot|^{\overline{F}_m}_P(x)=|\cdot|^{\mathscr{F}_{m,\varepsilon,O_{K_P}}}(x)$.
\item[\textup{(3)}] For every $P\in \Spec(O_K)\setminus U$ and for every $x\in X_{m,P}^{\rm an}$,
\[
 |\cdot|^{\mathscr{F}_{m,\varepsilon,O_{K_P}}}(x)\leq |\cdot|^{\overline{F}_m}_P(x)\leq\exp(\varepsilon)|\cdot|^{\mathscr{F}_{m,\varepsilon,O_{K_P}}}(x).
\]
\end{enumerate}
In particular, $(\mathscr{X}_{m,\varepsilon,U},\mathscr{F}_{m,\varepsilon,U})$ and $(\mathscr{X}_{m,\varepsilon,U},\mathscr{M}_{m,\varepsilon,U})$ give a $U$-model of definition for $\overline{F}_m^{\{\infty\}}$ and $\overline{M}_m^{\{\infty\}}$, respectively.
\end{claim}

\begin{proof}[Proof of Claim~\ref{clm:adelicmetric1}]
It suffices to show the assertion (1).
We decompose $\mu_m$ as $X_m\xrightarrow{\psi_{m,n}}X_{m,n}\xrightarrow{\nu_{m,n}}X$ according to (\ref{eqn:decomposemum}).
For $x\in X_{m,P}^{\rm an}$, we set $x':=\psi_{m,n,P}^{\rm an}(x)$ and $x'':=\mu_{m,P}^{\rm an}(x)$.
By Lemma~\ref{lem:pullbackadelic} and the formula (\ref{eqn:adelicassociatedtomodel}), we have
\begin{align*}
 &|1_{F_m}|^{\mathscr{F}_{m,\varepsilon,O_{K_P}}}(x) \\
 &\qquad =\inf\left\{|f|(x')^{1/n}\,:\,f\in k(\rho_P(x')),\,1_{\mathscr{F}_{m,n,\varepsilon}}\in ft_{x'}^*\mathscr{F}_{m,n,\varepsilon,O_{K_P}}\right\} \\
 &\qquad =\max\left\{|f|(x'')^{1/n}\,:\,f\in k(\rho_P(x'')),\,f\in t_{x''}^*\widehat{\mathfrak{b}}_{m,n,\varepsilon,O_{K_P}}\right\} \\
 &\qquad =\max_{s\in\aHzst(m\overline{L})}\left\{|s|^{m\mathscr{L}_{\varepsilon,O_{K_P}}}(x'')\right\}.
\end{align*}
\end{proof}

By Claim~\ref{clm:adelicmetric1}, the collection $(|\cdot|_v^{\overline{F}_m})_{v\in M_K}$ is in fact an adelic metric on $F_m$.

(1): It is obvious that $\aHzst(\overline{M}_m)\subset\left\{s'\,:\,s\in\aHzst(m\overline{L})\right\}$, and, by Claim~\ref{clm:adelicmetric2} below, we have $\aHzst(\overline{M}_m)=\left\{s'\,:\,s\in\aHzst(m\overline{L})\right\}$.

\begin{claim}\label{clm:adelicmetric2}
To each $s\in\aSpan{K}{\aHzst(m\overline{L})}$, we can associate an $s'\in\Hz(M_m)$ as above.
\begin{enumerate}
\item[\textup{(1)}] For any $v\in M_K$, we have $\|s'\|_{v,\sup}^{\overline{M}_m}\geq\|s\|_{v,\sup}^{m\overline{L}}$.
\item[\textup{(2)}] If $s\in\aHzst(m\overline{L})$, then $\|s'\|_{\infty,\sup}^{\overline{M}_m}=\|s\|_{\infty,\sup}^{m\overline{L}}<1$.
\end{enumerate}
\end{claim}

\begin{proof}[Proof of Claim~\ref{clm:adelicmetric2}]
The assertion (1) is clear since $\|1_{F_m}\|_{v,\sup}^{\overline{F}_m}\leq 1$ for every $v\in M_K$, so we are going to show the assertion (2).
For each $x\in (X_m\setminus\Supp(1_{F_m}))_{\infty}^{\rm an}$,
\[
 |s'|^{\overline{M}_m}_{\infty}(x)=|s|^{m\overline{L}}_{\infty}(\mu_m(x))\cdot\min_{t\in\aHzst(m\overline{L})}\left\{\frac{\|t\|_{\infty,\sup}^{m\overline{L}}}{|t|_{\infty}^{m\overline{L}}(\mu_m(x))}\right\}\leq \|s\|_{\infty,\sup}^{m\overline{L}}.
\]
\end{proof}

(2): We are going to verify the conditions (a)--(d) in Definition~\ref{defn:verticallynef}.
The condition (a) is clear and the condition (d) follows from the formula (\ref{eqn:adelicmetricsscorresp}).
Since $n\mathscr{M}_{m,\varepsilon}$ is free, the condition (b) follows from Claim~\ref{clm:adelicmetric1}.

(c): For each $x\in X_{m,\infty}^{\rm an}$, we choose a $\lambda$ and an $s_0\in\aHzst(m\overline{L})$ such that $x\in W_{\lambda,\infty}^{\rm an}$ and $\phi_{s_0,\lambda}(x)\neq 0$, where $\{W_{\lambda}\}$ and $\{\phi_{s,\lambda}\}$ are chosen as in the proof of Proposition~\ref{prop:adelicmetric}.
Then $s_0'$ gives a local frame of $M_m$ around $x$ and
\[
 -\log|s_0'|^{\overline{M}_m}_{\infty}(x')^2=\max_{s\in\aHzst(m\overline{L})}\left\{\log|\phi_{s,\lambda}/\phi_{s_0,\lambda}|(x')^2-\log\left(\|s\|_{\infty,\sup}^{m\overline{L}}\right)^2\right\}
\]
is plurisubharmonic over $\left\{x'\in W_{\lambda,\infty}^{\rm an}\,:\,\phi_{s_0,\lambda}(x')\neq 0\right\}$. 
\end{proof}

\begin{corollary}\label{cor:reproveZM}
Let $X$ be a normal projective variety that is geometrically irreducible over $K$, and let $\overline{L}$ be an adelically metrized line bundle on $X$.
If $\bigoplus_{m\geq 0}\Hz(mL)$ is a finitely generated $K$-algebra, then the following are equivalent.
\begin{enumerate}
\item[\textup{(a)}] $\Hz(mL)=\aSpan{K}{\aHzst(m\overline{L})}$ holds for every $m\gg 1$.
\item[\textup{(b)}] There exists an $a\geq 1$ such that $\Hz(aL)$ generates $\bigoplus_{m\geq 0}\Hz(maL)$ over $K$ and
\[
 \Image\left(\Hz(aL)\otimes_K(-aL)\to\mathcal{O}_X\right)=\Image\left(\aSpan{K}{\aHzst(a\overline{L})}\otimes_K(-aL)\to\mathcal{O}_X\right).
\]
\end{enumerate}
\end{corollary}

\begin{proof}
By \cite[Chap.\ III, \S 1.3, Proposition~3]{BourbakiCA72}, there exists an $e\geq 1$ such that $\bigoplus_{m\geq 0}\Hz(mneL)$ is generated by $\Hz(neL)$ over $\Hz(\mathcal{O}_X)=K$ for every $n\geq 1$.
So the implication (a) $\Rightarrow$ (b) is obvious.

We are going show the reverse.
For $m\geq 1$, we set
\[
 \mathfrak{b}_m:=\Image\left(\Hz(mL)\otimes_K(-mL)\to\mathcal{O}_X\right).
\]
Since $\Sym_K^m\Hz(aL)\to\Hz(maL)$ is surjective, we have $\mathfrak{b}_{ma}=(\mathfrak{b}_a)^m$ for every $m\geq 1$.
Let $\mu_a:X_a\to X$, $\overline{M}_a$, $\overline{F}_a$ be as in Proposition~\ref{prop:adelicmetric}.
Since $\aBs(\overline{M}_a)=\emptyset$, Theorem~\ref{thm:reprove} says that $\Hz(mM_a)=\aSpan{K}{\aHzst(m\overline{M}_a)}\subset\aSpan{K}{\aHzst(ma\overline{L})}$ for every $m\gg 1$.
Since $\mu_{a*}(-mF_a)\supset(\mathfrak{b}_a)^m$, we have
\[
 \Hz(mM_a)=\Hz(maL\otimes\mu_{a*}(-mF_a))\supset\Hz(maL\otimes(\mathfrak{b}_a)^m)=\Hz(maL)
\]
for every $m\geq 1$.
Thus $\Hz(maL)=\aSpan{K}{\aHzst(ma\overline{L})}$ for every $m\gg 1$.
We conclude the proof by Lemma~\ref{lem:Veronese}.
\end{proof}

\section{Arithmetic augmented base loci}\label{sec:arithaugbs}

\begin{definition}\label{defn:aaugbs}
Let $X$ be a projective variety over a number field $K$.
Let $L$ be a line bundle on $X$, and let $\overline{V}_{\sbullet}$ be an adelically normed graded $K$-linear series belonging to $L$.
We define the \emph{arithmetic augmented base locus} of $\overline{V}_{\sbullet}$ as
\begin{equation}
 \aBsp(\overline{V}_{\sbullet}):=\Bsp(\aSpan{K}{\aHzst(\overline{V}_{\sbullet})}),
\end{equation}
which clearly does not depend on $K$ (see Definition~\ref{defn:Chens}).
If $\overline{L}$ is an adelically metrized line bundle on $X$ and $\overline{V}_{\sbullet}$ is given by the formula (\ref{eqn:adelicgradedlinser}), then we write $\aBsp(\overline{L}):=\aBsp(\overline{V}_{\sbullet})$.
Since $\aBsp(m\overline{L})=\aBsp(\overline{L})$ for every $m\geq 1$, we can define $\aBsp(\overline{L})$ for every $\overline{L}\in\aPic_{\QQ}(X)$.
Let $\overline{L}$ be an adelically metrized line bundle on $X$.
\begin{description}
\item[(semiample)] We say that $\overline{L}$ is \emph{free} if the homomorphism
\[
 \aSpan{K}{\aHzst(\overline{L})}\otimes_{K}\mathcal{O}_X\to L
\]
is surjective.
We say that an $\overline{L}\in\aPic_{\QQ}(X)$ is \emph{semiample} if some multiple of $\overline{L}$ is free.
\item[(w-ample)] We say that $\overline{L}$ is \emph{weakly ample} or \emph{w-ample} for short if $L$ is ample and $\Hz(mL)=\aSpan{K}{\aHzst(m\overline{L})}$ for every $m\gg 1$.
We say that an $\overline{L}\in\aPic_{\QQ}(X)$ is \emph{w-ample} if some multiple of $\overline{L}$ is w-ample.
If $\overline{L}$ is ample in the sense of Moriwaki (\cite[\S 0.3 Conventions and terminology (7)]{MoriwakiEst}) or ample in the sense of Zhang (\cite[Corollary~(4.8)]{ZhangVar}, \cite[(1.3) and Theorem~(1.8)]{ZhangAdelic}), then $\overline{L}$ is w-ample.
\end{description}

\begin{remark}\label{rem:w-ample}
The notion of w-ampleness does not depend on the choice of $K$.
In fact, by Theorem~\ref{thm:reprove}, the following are equivalent.
\begin{enumerate}
\item[(a)] $\overline{L}$ is w-ample.
\item[(b)] $L$ is ample and there exist an $m\geq 1$ and $s_1,\dots,s_N\in\aHzst(m\overline{L})$ such that
\[
 \left\{x\in X\,:\,s_1(x)=\dots=s_N(x)=0\right\}=\emptyset.
\]
\end{enumerate}
\end{remark}

We denote the Kodaira map
\begin{equation}
 \Phi_{\aSpan{K}{\aHzst(\overline{L})}}:X\setminus\aBs(\overline{L})\to\PP_K(\aSpan{K}{\aHzst(\overline{L})})
\end{equation}
associated to $\aSpan{K}{\aHzst(\overline{L})}$ by $\aPhi_{\overline{L},K}$.
\end{definition}

\begin{lemma}\label{lem:w-ample}
Let $\overline{A}$ be an adelically metrized line bundle on $X$.
\begin{enumerate}
\item[\textup{(1)}] The following are equivalent.
\begin{enumerate}
\item[\textup{(a)}] $\overline{A}$ is w-ample.
\item[\textup{(b)}] $a\overline{A}$ is w-ample for an $a\geq 1$.
\item[\textup{(c)}] For an $a\geq 1$,
\[
 \aPhi_{a\overline{A},K}:X\to\PP_K(\aSpan{K}{\aHzst(a\overline{A})})
\]
is a closed immersion.
\item[\textup{(d)}] Given any adelically metrized line bundle $\overline{L}$ on $X$, $mA+L$ is very ample and $\Hz(mA+L)=\aSpan{K}{\aHzst(m\overline{A}+\overline{L})}$ for every $m\gg 1$.
\item[\textup{(e)}] Given any adelically metrized line bundle $\overline{L}$ on $X$, $m\overline{A}+\overline{L}$ is w-ample for every $m\gg 1$.
\item[\textup{(f)}] Given any adelically metrized line bundle $\overline{L}$ on $X$, $m\overline{A}+\overline{L}$ is semiample for every $m\gg 1$.
\end{enumerate}
\item[\textup{(2)}] If $\overline{A}$ is w-ample and $\overline{F}$ is semiample, then $\overline{A}+\overline{F}$ is w-ample.
\end{enumerate}
\end{lemma}

\begin{proof}
(1): The implications (a) $\Rightarrow$ (b), (a) $\Rightarrow$ (c), (e) $\Rightarrow$ (b), and (e) $\Rightarrow$ (f) are trivial.
The implication (b) $\Rightarrow$ (a) is obvious (see for example Remark~\ref{rem:w-ample}).


(c) $\Rightarrow$ (a): By the implication (b) $\Rightarrow$ (a), we can assume $a=1$.
Let $P:=\PP_K(\aSpan{K}{\aHzst(\overline{A})})$, and let $\mathcal{O}_P(1)$ be the hyperplane line bundle on $P$.
Since $A$ is isomorphic to $\aPhi_{\overline{A},K}^*\mathcal{O}_P(1)$, the homomorphism
\[
 \Hz(\mathcal{O}_P(m))=\Sym_K^m\aSpan{K}{\aHzst(\overline{A})}\to\Hz(mA)
\]
is surjective for every $m\gg 1$.
Hence $\aSpan{K}{\aHzst(m\overline{A})}=\Hz(mA)$ for every $m\gg 1$.

(a) $\Rightarrow$ (d): There exists an $\varepsilon>0$ such that $\overline{A}-\mathcal{O}_X(\varepsilon[\infty])$ is also w-ample.
We can find positive integers $a,b$ such that $\aSpan{K}{\aHzst(ma(\overline{A}-\mathcal{O}_X(\varepsilon[\infty]))}=\Hz(maA)$ for every $m\geq 1$, $mA+L$ is very ample for every $m\geq a$, and
\[
 \Hz(maA)\otimes_{K}\Hz((ab+r)A+L)\to\Hz(((m+b)a+r)A+L)
\]
are surjective for all $m,r$ with $m\geq 1$ and $0\leq r<a$.
Then
\[
 \aSpan{K}{\aHzst(ma(\overline{A}-\mathcal{O}_X(\varepsilon[\infty]))}\otimes_K\Hz((ab+r)A+L)\to\Hz(((m+b)a+r)A+L)
\]
are surjective for all $m,r$ with $m\geq 1$ and $0\leq r<a$, so we have $\aSpan{K}{\aHzst(m\overline{A}+\overline{L})}=\Hz(mA+L)$ for every $m\gg 1$.

The implication (d) $\Rightarrow$ (e) is clear by (c) $\Rightarrow$ (a).

Before proving the implication (f) $\Rightarrow$ (a), we show the assertion (2).
Take an $m\geq 1$ such that $\aPhi_{m\overline{A},K}$ is a closed immersion and $m\overline{F}$ is free.
Set
\[
 Q_m:=\Coker\left(\aSpan{K}{\aHzst(m\overline{F})}\otimes_K\aSpan{K}{\aHzst(m\overline{A})}\to\aSpan{K}{\aHzst(m(\overline{A}+\overline{F}))}\right).
\]
By considering the commutative diagram
\[
\xymatrix{ X \ar[rrrr]^-{\aPhi_{m(\overline{A}+\overline{F}),K}} \ar@{=}[d] &&&& \PP_K(\aSpan{K}{\aHzst(m(\overline{A}+\overline{F}))})\setminus\PP_K(Q_m) \ar[d] \\ X \ar[rrrr]^-{(\text{Segre emb.})\circ(\aPhi_{m\overline{A},K}\times\aPhi_{m\overline{F},K})} &&&& \PP_K(\aSpan{K}{\aHzst(m\overline{A})}\otimes_K\aSpan{K}{\aHzst(m\overline{F})}),
}
\]
we know that $\aPhi_{m(\overline{A}+\overline{F}),K}$ is a closed immersion.

(f) $\Rightarrow$ (a): Let $\overline{B}$ be a w-ample adelically metrized line bundle on $X$.
By the condition (f), $m\overline{A}-\overline{B}$ is semiample for every $m\gg 1$.
Thus by the assertion (2), $m\overline{A}$ is w-ample for every $m\gg 1$.
\end{proof}

\begin{proposition}\label{prop:aaugbs}
Let $\overline{L}$ be an adelically metrized line bundle on $X$.
\begin{enumerate}
\item[\textup{(1)}] The augmented base locus of $\overline{L}$ satisfies
\[
 \aBsp(\overline{L})=\bigcap_{\text{$\overline{A}$: w-ample}}\aSBs(\overline{L}-\overline{A})=\Bsp(L)\cup\aSBs(\overline{L}),
\]
where the intersection in the middle is taken over all the w-ample adelically metrized $\QQ$-line bundles $\overline{A}$ on $X$.
In particular, given any w-ample adelically metrized line bundle $\overline{A}$, $\aBsp(\overline{L})=\aSBs(\overline{L}-\varepsilon\overline{A})$ for every sufficiently small rational number $\varepsilon>0$.
\item[\textup{(2)}] Let $\overline{M}$ be another adelically metrized line bundle on $X$.
If $s\in\aHzsm(\overline{M})$, then
\[
 \aBsp(\overline{L}+\overline{M})\subset\aBsp(\overline{L})\cup\Supp(s).
\]
\item[\textup{(3)}] $\overline{L}$ is w-ample if and only if $\aBsp(\overline{L})=\emptyset$.
\item[\textup{(4)}] Let $\overline{A}_1,\dots,\overline{A}_r$ be w-ample adelically metrized line bundles on $X$.
Then there exists an $\varepsilon>0$ such that
\[
 \aBsp(\overline{L})=\aBsp(\overline{L}-\varepsilon_1\overline{A}_1-\dots-\varepsilon_r\overline{A}_r)
\]
for any rational numbers $\varepsilon_1,\dots,\varepsilon_r$ with $0\leq\varepsilon_i\leq\varepsilon$.
\item[\textup{(5)}] Let $\overline{A}_1,\dots,\overline{A}_r$ be adelically metrized line bundles on $X$.
If $\overline{L}$ is w-ample, then there exists a $\delta>0$ such that $\overline{L}+\delta_1\overline{A}_1+\dots+\delta_r\overline{A}_r$ is w-ample for any rational numbers $\delta_1,\dots,\delta_r$ with $|\delta_i|\leq\delta$.
In particular, the w-ampleness is an open condition.
\end{enumerate}
\end{proposition}

\begin{proof}
(1): Let $\overline{V}_{\sbullet}$ be as in (\ref{eqn:adelicgradedlinser}).
For any $a\geq 1$ and for any w-ample adelically metrized line bundle $\overline{A}$ on $X$, we have $\Hz(aL-A)\supset\Lambda(\aSpan{K}{\aHzst(\overline{V}_{\sbullet})};A,a)\supset\aSpan{K}{\aHzst(a\overline{L}-\overline{A})}$.
Thus by Lemma~\ref{lem:augbs}(3) we have
\begin{equation}\label{eqn:aaugbsimplications}
 \Bsp(L)\cup\aSBs(\overline{L})\subset\Bsp(\aSpan{K}{\aHzst(\overline{V}_{\sbullet})})\subset\bigcap_{\text{$\overline{A}$: w-ample}}\aSBs(\overline{L}-\overline{A}).
\end{equation}
Suppose that $x\notin\Bsp(L)\cup\aSBs(\overline{L})$.
We can find a w-ample adelically metrized line bundle $\overline{A}$ on $X$, an $s\in\aHzst(b\overline{L})$, and a $t\in\aHzf(c\overline{L}-\overline{A})$ such that $s(x)\neq 0$ and $t(x)\neq 0$.
For a $p\geq 1$, we have
\[
 s^{\otimes p}\otimes t\in\aHzst((pb+c)\overline{L}-\overline{A}).
\]
Thus $x\notin\aBs((pb+c)\overline{L}-\overline{A})$.
This implies that the equalities hold in (\ref{eqn:aaugbsimplications}).

The assertion (2) is clear since
\[
 \aSBs(\overline{L}+\overline{M}-\overline{A})\subset\aSBs(\overline{L}-\overline{A})\cup\Supp(s)
\]
for every w-ample adelically metrized $\QQ$-line bundle $\overline{A}$.

(3): If $\overline{L}$ is w-ample, then $\aBsp(\overline{L})=\emptyset$.
Conversely, if $\aBsp(\overline{L})=\emptyset$, then there exists a w-ample $\QQ$-line bundle $\overline{A}$ on $X$ such that $\aSBs(\overline{L}-\overline{A})=\emptyset$.
Thus by Lemma~\ref{lem:w-ample}(2), $\overline{L}$ is w-ample.

(4): The inclusion $\subset$ is clear.
Since $\overline{A}_1+\dots+\overline{A}_r$ is w-ample, there exists a rational number $\varepsilon>0$ such that
\[
 \aBsp(\overline{L})=\aSBs\left(\overline{L}-\varepsilon(\overline{A}_1+\dots+\overline{A}_r)\right).
\]
Then for any $\varepsilon_1,\dots,\varepsilon_r$ with $0\leq\varepsilon_i\leq\varepsilon$, we have
\[
 \aBsp(\overline{L})\subset\aSBs(\overline{L}-\varepsilon_1\overline{A}_1-\dots-\varepsilon_r\overline{A}_r)\subset\aSBs\left(\overline{L}-\varepsilon(\overline{A}_1+\dots+\overline{A}_r)\right).
\]
Thus we conclude.

(5): Since every adelically metrized line bundle is a difference of two w-ample adelically metrized line bundles (Lemma~\ref{lem:w-ample}(1)), we can assume that $\overline{A}_1,\dots,\overline{A}_r$ are all w-ample.
By the assertions (3) and (4) above, there exists an $\varepsilon>0$ such that
\[
 \aBsp(\overline{L}-\varepsilon_1\overline{A}_1-\dots-\varepsilon_r\overline{A}_r)=\emptyset
\]
for any $\varepsilon_1,\dots,\varepsilon_r$ with $0\leq\varepsilon_i\leq\varepsilon$.
Then, for any $\delta_1,\dots,\delta_r$ with $\delta_i\geq -\varepsilon$, $\overline{L}+\delta_1\overline{A}_1+\dots+\delta_r\overline{A}_r$ is w-ample.
\end{proof}

\begin{remark}\label{rem:aaugbscompare}
There are several different definitions of the arithmetic augmented base locus.
For an adelically metrized line bundle $\overline{L}$, we set
\[
 \aHzsm(\overline{L}):=\left\{s\in\Hz(L)\,:\,\|s\|_{v,\sup}^{\overline{L}}\leq 1,\,\forall v\in M_K\right\}.
\]
An section in $\aHzsm(\overline{L})$ is referred to as a \emph{small section} of $\overline{L}$.
In \cite[\S 3]{MoriwakiEst}, Moriwaki defined
\[
 \aSBsn_+'(\overline{L}):=\bigcap_{\text{$\overline{A}$: w-ample}}\aSBssm(\overline{L}-\overline{A}),
\]
where the intersection is taken over all the w-ample adelically metrized $\QQ$-line bundles $\overline{A}$ (see section \ref{sec:adelicallymetrizedlbd}).
We can easily see that $\Bsp(L)\cup\aSBs(\overline{L})\subset\aSBsn_+'(\overline{L})\subset\aBsp(\overline{L})$.
Hence Moriwaki's $\aSBsn_+'(\overline{L})$ is identical to our $\aBsp(\overline{L})$.
Of course, $\aSBssm(\overline{L})\subset\aSBs(\overline{L})$ and the equality does not hold in general.

In \cite[\S 4]{ChenSesh}, Chen defined the augmented base locus by using the sections with normalized Arakelov degree not less than zero.
If $\overline{L}$ is associated to a continuous Hermitian line bundle on an $O_K$-model of $X$, then by \cite[Remark~3.8(ii)]{Boucksom_Chen} one can see that Chen's definition also coincides with ours.
\end{remark}

\begin{theorem}\label{thm:charaaugbs}
Let $X$ be a normal projective variety over a number field and let $\overline{L}$ be an adelically metrized line bundle on $X$.
The arithmetic augmented base locus $\aBsp(\overline{L})$ is characterized as the minimal Zariski closed subset of $X$ such that the restriction of the Kodaira map
\[
 \aPhi_{m\overline{L},K}:X\setminus\aBs(m\overline{L})\to\PP_K(\aSpan{K}{\aHzst(m\overline{L})})
\]
to $X\setminus\aBsp(\overline{L})$ is an immersion for every sufficiently divisible $m\gg 1$.
\end{theorem}

\begin{proof}
We choose an $a\geq 1$ such that $\aSBs(\overline{L})=\aBs(ma\overline{L})$ and
\[
 \Exc(\aPhi_{ma\overline{L},K})\cup\aBs(ma\overline{L})=\bigcap_{n\geq 1}\left(\Exc(\aPhi_{n\overline{L},K})\cup\aBs(n\overline{L})\right)
\]
for every $m\geq 1$ (Lemma~\ref{lem:excep}(2)).
By Lemma~\ref{lem:Kodairatype},
\begin{equation}\label{eqn:charaug0}
 \aBsp(\overline{L})\supset\Exc(\aPhi_{ma\overline{L},K})\cup\aBs(ma\overline{L})
\end{equation}
for every $m\geq 1$.
To show the reverse inclusion, let $\mu_a:X_a\to X$, $\overline{F}_a$, and $\overline{M}_a$ be as in Proposition~\ref{prop:adelicmetric} and let $1_{F_a}\in\aHzsm(\overline{F}_a)$ be the natural inclusion.
Then $\overline{M}_a$ is free.
By Theorems~\ref{thm:reprove} and \ref{thm:charaugbs}, there exists a $p\geq 1$ such that
\begin{equation}\label{eqn:charaug1}
 \aBsp(\overline{M}_a)=\Exc(\aPhi_{p\overline{M}_a,K})=\Bsp(M_a).
\end{equation}
Set $Y:=\aPhi_{pa\overline{L},K}(X)$ and $Y':=\aPhi_{p\overline{M}_a,K}(X_a)$.
We consider the commutative diagram
\[
\xymatrix{ X_a \ar[rr]^-{\aPhi_{p\overline{M}_a,K}} \ar[d]_-{\mu_a} && Y' \ar[d]^-{\nu} \\
 X \ar@{-->}[rr]^-{\aPhi_{pa\overline{L},K}} && Y.
}
\]
By applying Proposition~\ref{prop:aaugbs}(2) to the decomposition $p\mu_a^*(a\overline{L})=p\overline{M}_a+p\overline{F}_a$, we have
\begin{equation}\label{eqn:charaug3}
 \aBsp(p\mu_a^*(a\overline{L}))\subset\aBsp(p\overline{M}_a)\cup\Supp(1_{F_a}^{\otimes p}).
\end{equation}
Since
\[
 X_a\setminus\mu_a^{-1}\left(\Exc(\aPhi_{pa\overline{L},K})\cup\aBs(pa\overline{L})\right)\xrightarrow{\mu_a} X\setminus\left(\Exc(\aPhi_{pa\overline{L},K})\cup\aBs(pa\overline{L})\right)\xrightarrow{\aPhi_{pa\overline{L},K}} Y
\]
is an immersion, we have
\begin{equation}\label{eqn:charaug2}
 \Exc(\aPhi_{p\overline{M}_a,K})\subset\mu_a^{-1}\left(\Exc(\aPhi_{pa\overline{L},K})\cup\aBs(pa\overline{L})\right).
\end{equation}
Moreover, we have
\begin{equation}\label{eqn:charaug4}
 \aBsp(\mu_a^*\overline{L})=\mu_a^{-1}\aBsp(\overline{L})
\end{equation}
thanks to Lemma~\ref{lem:augbs_main}(2).
Therefore,
\[
 \mu_a^{-1}\aBsp(\overline{L})\subset\mu_a^{-1}\left(\Exc(\aPhi_{pa\overline{L},K})\cup\aBs(pa\overline{L})\right)
\]
by (\ref{eqn:charaug1})--(\ref{eqn:charaug4}), and
\[
 \aBsp(\overline{L})\subset\Exc(\aPhi_{pa\overline{L},K})\cup\aBs(pa\overline{L}).
\]
\end{proof}

\section{Yuan's estimation}\label{sec:YuansEst}

The main result of this section is Theorem~\ref{thm:Yuan}, which is the key to prove Theorem~B and to prove fundamental properties of the arithmetic restricted volumes and the arithmetic multiplicities.
The ideas to construct arithmetic Okounkov bodies can be traced back to Yuan's paper \cite{Yuan09}.
Later, Yuan \cite{Yuan12} largely simplified the construction, and Boucksom--Chen \cite{Boucksom_Chen} presented another method.
In this paper, we decide to rewrite the arguments in \cite{Yuan09,MoriwakiEst} with necessary changes.
We remark that most part of the arguments in \cite{Yuan12} is also applicable to the general case except some relations before \cite[Lemma~3.2]{Yuan12}.

The arithmetic restricted volume we study below was first introduced by Moriwaki \cite{MoriwakiEst}.
Let $M$ be a free $\ZZ$-module of finite rank.
A subset $\Gamma$ of $M$ is called a \emph{CL-subset of $M$} if the following equivalent conditions are satisfied (see \cite[Proposition~1.2.1(2)]{MoriwakiEst}).
\begin{enumerate}
\item[(a)] There exist a $\ZZ$-submodule $N$ of $M$ and a convex subset $\Delta\subset M\otimes_{\ZZ}\RR$ such that $\Gamma=N\cap\Delta$.
\item[(b)] Let $\aSpan{\RR}{\Gamma}$ be the $\RR$-vector subspace of $M\otimes_{\ZZ}\RR$ generated by $\Gamma$ and let $\Conv_{\aSpan{\RR}{\Gamma}}(\Gamma)$ be the minimal convex body containing $\Gamma$ in $\aSpan{\RR}{\Gamma}$.
Then
\[
 \Gamma=\aSpan{\ZZ}{\Gamma}\cap\Conv_{\aSpan{\RR}{\Gamma}}(\Gamma).
\]
Note that $\aSpan{\ZZ}{\Gamma}\otimes_{\ZZ}\RR\xrightarrow{\sim}\aSpan{\RR}{\Gamma}$ in this case.
\item[(c)] $\Gamma=\bigcup_{l\geq 1}\left\{\frac{\gamma_1+\cdots+\gamma_l}{l}\in\aSpan{\ZZ}{\Gamma}\,:\,\gamma_1,\dots,\gamma_l\in\Gamma\right\}$.
\end{enumerate}
A subset $S\subset M\otimes_{\ZZ}\RR$ is said to be \emph{symmetric} if $\gamma\in S$ implies $-\gamma\in S$.
Given any subset $S$ in $M$, we define the \emph{CL-hull of $S$ in $M$} as the smallest CL-subset of $M$ containing $S$, which we shall denote by $\CL_M(S)$.
Moreover, we set
\begin{equation}
 m\ast S:=\left\{\gamma_1+\cdots+\gamma_m\,:\,\gamma_1,\dots,\gamma_m\in S\right\}
\end{equation}
for an integer $m\geq 1$ and a subset $S$ of $M$.

\begin{lemma}[\text{\cite[Proposition~2.8]{Yuan09}, \cite[Lemma~1.2.2]{MoriwakiEst}}]\label{lem:Yuansineq}
Let $M$ be a free $\ZZ$-module of finite rank, and let $r:M\to N$ be a surjective homomorphism of $\ZZ$-modules.
\begin{enumerate}
\item[\textup{(1)}] Let $\Gamma$ be a symmetric finite subset of $M$.
Then
\[
 \log\sharp\Gamma-\log\sharp(\Ker(r)\cap(2\ast\Gamma))\leq \log\sharp r(\Gamma)\leq\log\sharp(2\ast\Gamma)-\log(\Ker(r)\cap\Gamma).
\]
\item[\textup{(2)}] Let $\Delta$ be a bounded symmetric convex subset of $M\otimes_{\ZZ}\RR$, and let $a\geq 1$ be a real number.
Then
\[
 0\leq \log\sharp(M\cap a\Delta)-\log\sharp(M\cap\Delta)\leq\log(\lceil 2a\rceil)\rk_{\ZZ}M.
\]
\end{enumerate}
\end{lemma}

\begin{lemma}\label{lem:rkcomparison}
Let $K$ be a number field, let $M$ be a projective $O_K$-module of finite rank, and let $\Gamma$ be a finite subset of $M$.
Then
\[
 \rk_{O_K}\aSpan{O_K}{\Gamma}\leq\rk_{\ZZ}\aSpan{\ZZ}{\Gamma}\leq [K:\QQ]\rk_{O_K}\aSpan{O_K}{\Gamma}.
\]
\end{lemma}

\begin{proof}
The first inequality is clear.
Since $\aSpan{\ZZ}{\Gamma}\subset\aSpan{O_K}{\Gamma}$, we have $\rk_{\ZZ}\aSpan{\ZZ}{\Gamma}\leq\rk_{\ZZ}\aSpan{O_K}{\Gamma}=[K:\QQ]\rk_{O_K}\aSpan{O_K}{\Gamma}$.
\end{proof}

In the rest of this section, let $X$ be a projective variety that is geometrically irreducible over a number field $K$.
Let $U$ be a non-empty open subset of $\Spec(O_K)$ and let $\pi_U:\mathscr{X}_U\to U$ be a $U$-model of $X$, so that $\mathscr{X}_U$ is reduced and irreducible and $\pi_U$ is flat and projective.

\begin{definition}\label{defn:goodflag}
A \emph{flag on $\mathscr{X}_U$} is a sequence of reduced irreducible closed subschemes of $\mathscr{X}_U$,
\[
 F_{\sbullet}:\mathscr{X}_U=F_{-1}\supsetneq F_0\supsetneq F_1\supsetneq\dots\supsetneq F_{\dim X}=\{\xi\},
\]
such that each $F_i$ has codimension $i+1$ in $\mathscr{X}_U$, $F_{\dim X}$ consists of a closed point $\xi\in\mathscr{X}_U$, and each $F_{i+1}$ is locally principal in $F_i$ around $\xi$.
The closed point $\xi=\xi_{F_{\sbullet}}$ is called the \emph{center of the flag $F_{\sbullet}$}.

Let $\Psi$ be a Zariski closed subset of $\mathscr{X}_U$.
We say that $F_{\sbullet}$ is a \emph{$\Psi$-good flag on $\mathscr{X}_U$ over a prime number $p$} if the following conditions are satisfied.
\begin{enumerate}
\item[(a)] There exists a prime ideal $\mathfrak{p}\in U$ such that $\mathfrak{p}\cap\ZZ=p\ZZ$ and $[O_{K_{\mathfrak{p}}}/\mathfrak{p}O_{K_{\mathfrak{p}}}:\FF_p]=1$.
\item[(b)] $F_0=\pi^{-1}(\mathfrak{p})$ and the center $\xi$ is $\FF_p$-rational.
\item[(c)] The center $\xi$ is not contained in $\Psi$.
\end{enumerate}
An $\emptyset$-good flag shall be simply called a \emph{good flag} (see \cite[\S 1.4]{MoriwakiEst}).
Note that $F_0$ is a Cartier divisor on $\mathscr{X}_U$ and $F_0,\dots,F_{\dim X}$ are all geometrically irreducible over $\FF_p$.
\end{definition}

Let $\Rat(X)$ be the field of rational functions on $X$ and let $F_{\sbullet}:\mathscr{X}_U=F_{-1}\supset F_0\supset\dots\supset F_{\dim X}=\{\xi\}$ be any flag on $\mathscr{X}_U$.
We then define the \emph{valuation map} $\bm{w}_{F_{\sbullet}}:\Rat(X)^{\times}\to\ZZ^{\dim X+1}$ associated to $F_{\sbullet}$ as follows.
For each $i=0,\dots,\dim X$, we fix a local equation $f_i$ defining $F_i$ in $F_{i-1}$ around $\xi$.
For $\phi\in\Rat(X)^{\times}$, we set $\phi_0:=\phi$, $w_0(\phi):=\ord_{F_0}(\phi_0)$,
\[
 \phi_i:=\left.\left(f_{i-1}^{-w_{i-1}(\phi)}\phi_{i-1}\right)\right|_{F_{i-1}},\quad\text{and}\quad w_i(\phi):=\ord_{F_i}(\phi_i)
\]
for $i=1,\dots,\dim X$, inductively.
Then define
\begin{equation}
 \bm{w}_{F_{\sbullet}}(\phi):=(w_0(\phi),w_1(\phi),\dots,w_{\dim X}(\phi)),
\end{equation}
which does not depend on the choice of $f_0,\dots,f_{\dim X}$.
Note that
\begin{equation}
 \bm{v}_{F_{\sbullet}}(\phi_1)=(w_1(\phi),\dots,w_{\dim X}(\phi))
\end{equation}
is the valuation vector of $\phi_1\in\Rat(F_0)^{\times}$ associated to the flag $F_0\supset\dots\supset F_{\dim X}$ on $F_0$ (see \cite[\S 1.1]{Lazarsfeld_Mustata08}).

\begin{lemma}\label{lem:imageflag}
Let $\mathscr{X}_U'$ be another $U$-model of $X$ and let $\varphi_U:\mathscr{X}_U'\to\mathscr{X}_U$ be a projective birational $U$-morphism.
\begin{enumerate}
\item[\textup{(1)}] Let $\Psi'$ be a Zariski closed subset of $\mathscr{X}_U'$ and let $F_{\sbullet}':\mathscr{X}_U'\supset F_0'\supset\dots\supset F_{\dim X}'=\{\xi'\}$ be a $\Psi'$-good flag on $\mathscr{X}_U'$ over a prime number $p$.
If $\varphi_U$ is isomorphic around $\xi'$, then the sequence of the images
\[
 \varphi_U(F_{\sbullet}'):\mathscr{X}_U\supset\varphi_U(F_0')\supset\dots\supset\varphi_U(F_{\dim X}')=\{\varphi_U(\xi')\}
\]
is a $\varphi_U(\Psi')$-good flag on $\mathscr{X}_U$ over $p$, and $\bm{w}_{\varphi_U(F_{\sbullet}')}=\bm{w}_{F_{\sbullet}}\circ\varphi_U^*$.
\item[\textup{(2)}] Let $\Psi$ be a Zariski closed subset of $\mathscr{X}_U$ and let $F_{\sbullet}:\mathscr{X}_U\supset F_0\supset\dots\supset F_{\dim X}=\{\xi\}$ be a $\Psi$-good flag on $\mathscr{X}_U$ over a prime number $p$.
If $\varphi_U$ is isomorphic around $\xi$, then the sequence of the strict transforms
\[
 \varphi_{U*}^{-1}(F_{\sbullet}):\mathscr{X}_U'\supset\varphi_{U*}^{-1}(F_0)\supset\dots\supset\varphi_{U*}^{-1}(F_{\dim X})=\{\varphi_U^{-1}(\xi)\}
\]
is a $\varphi_U^{-1}(\Psi)$-good flag on $\mathscr{X}_U'$ over $p$ and $\bm{w}_{\varphi_{U*}^{-1}(F_{\sbullet})}\circ\varphi_U^*=\bm{w}_{F_{\sbullet}}$.
\end{enumerate}
\end{lemma}

\begin{lemma}[\text{\cite[\S 2.2]{Yuan09}, \cite[Proposition~1.4.1]{MoriwakiEst}}]\label{lem:existgoodflags}
Given any Zariski closed subset $\Psi$ of $\mathscr{X}_U$ such that $\Psi\neq\mathscr{X}_U$, there exist $\Psi$-good flags on $\mathscr{X}$ over all but finitely many prime numbers.
\end{lemma}

\begin{proof}
Let $\varphi_U:\mathscr{X}_U'\to\mathscr{X}_U$ be a projective birational $U$-morphism such that $X':=\mathscr{X}_U'\times_U\Spec(K)$ is smooth and geometrically irreducible over $K$.
Set $\pi_U':=\pi_U\circ\varphi_U$ and $\widetilde{\Psi}:=\varphi_U^{-1}(\Psi)\cup\Exc(\varphi_U)$.
We can choose a sequence of reduced irreducible closed subschemes of $\mathscr{X}_U'$, $\mathscr{F}_{\sbullet}:\mathscr{X}_U'\supsetneq\mathscr{F}_{0}\supsetneq\dots\supsetneq\mathscr{F}_{\dim X-1}$, such that, for any $i=0,\dots,\dim X-1$, $\pi_U'|_{\mathscr{F}_i}:\mathscr{F}_{i}\to U$ is flat, $\mathscr{F}_{i,K}:=\mathscr{F}_{i}\times_{U}\Spec(K)$ is smooth of codimension $i+1$ in $X$, and $\mathscr{F}_{i}$ is not contained in $\widetilde{\Psi}$.
Let $U_0\subset U\subset\Spec(O_K)$ be the set of all the prime ideals $\mathfrak{p}$ having the properties that
\begin{enumerate}
\item[(a)] the prime ideal $p\ZZ:=\mathfrak{p}\cap\ZZ$ completely splits in $K$,
\item[(b)] $\mathscr{F}_{i}\cap{\pi'}^{-1}(\mathfrak{p})$ is smooth and is not contained in $\widetilde{\Psi}$ for every $i$, and
\item[(c)] $\sharp(\mathscr{F}_{\dim X-1}\cap\widetilde{\Psi})<p+1-2g\sqrt{p}$, where $g$ is the genus of the smooth curve $\mathscr{F}_{\dim X-1,K}$.
\end{enumerate}
Thanks to Chebotarev's density theorem, $\Spec(O_{K})\setminus U_0$ is a finite set.
By the property (c) and Weil's theorem, given any $\mathfrak{p}\in U_0$, we can take a $\xi\in\mathscr{F}_{\dim X-1}(\FF_p)$ that is not contained in $\widetilde{\Psi}$.
Therefore, for each $\mathfrak{p}\in U_0$, the sequence
\[
 {\pi_U'}^{-1}(\mathfrak{p})\supset\mathscr{F}_{0}\cap{\pi_U'}^{-1}(\mathfrak{p})\supset\dots\supset\mathscr{F}_{\dim X-1}\cap{\pi_U'}^{-1}(\mathfrak{p})\supset\{\xi\}
\]
is a $\varphi_U^{-1}(\Psi)$-good flag on $\mathscr{X}_U'$ over $p$ and $\varphi_U$ is isomorphic around $\xi$.
Thus the assertion follows from Lemma~\ref{lem:imageflag}(1).
\end{proof}

Let $R$ be an order of $K$ such that $\Spec(O_K)\to\Spec(R)$ is isomorphic over $U$.
Let $\mathscr{X}$ be an $R$-model of $X$ that extends $\mathscr{X}_U$ (for example, embed $\mathscr{X}_U$ into a projective space $\PP_U$ over $U$ and take $\mathscr{X}$ as the Zariski closure of $\mathscr{X}_U$ in $\PP_R$).
Let $\nu:\mathscr{X}'\to\mathscr{X}$ be a relative normalization in $X$ and let $F_{\sbullet}$ be a flag on $\mathscr{X}_U'$.
Let $\overline{L}$ be an adelically metrized line bundles on $X$ such that $(\mathscr{X}_U,\mathscr{L}_U)$ gives a $U$-model of definition for $\overline{L}^{\{\infty\}}$.
We fix a local frame $\eta$ of $\nu^*\mathscr{L}_U$ around $\xi=\xi_{F_{\sbullet}}$.
Any $s\in\aHzf(\overline{L})\setminus\{0\}$ can be written as $\nu^*s=\phi\eta$ with a non-zero local function $\phi$ around $\xi$.
We define the \emph{valuation map} associated to $F_{\sbullet}$ as
\begin{equation}\label{eqn:definitionvalvector}
 \bm{w}_{F_{\sbullet}}:\aHzf(\overline{L})\to\Hz(\nu^*\mathscr{L}_U)\to\ZZ^{\dim X+1},\quad s\mapsto (w_0(\phi),\dots,w_{\dim X}(\phi)),
\end{equation}
(see (\ref{eqn:injmodelofdef})), which does not depend on the choice of the frame $\eta$.

\begin{lemma}\label{lem:valmapissurj}
\begin{enumerate}
\item[\textup{(1)}] Let $\mathscr{A}$ be any ample line bundle on $\mathscr{X}$.
Then for every $m\gg 1$, the image $\bm{w}_{F_{\sbullet}}\left(\Hz(m\mathscr{A})\setminus\{0\}\right)$ contains all of the vectors
\[
 (0,\dots,0),\,(1,0,\dots,0),\,\dots,\,(0,\dots,0,1)\,\in\ZZ^{\dim X+1}.
\]
\item[\textup{(2)}] The valuation map $\bm{w}_{F_{\sbullet}}:\Rat(X)^{\times}\to\ZZ^{\dim X+1}$ is surjective.
\item[\textup{(3)}] Let $E$ be a subextension of $\Rat(X)/K$.
Then $\bm{w}_{F_{\sbullet}}(E^{\times})$ is a free $\ZZ$-module of rank $\trdeg_{K}E+1$.
\end{enumerate}
\end{lemma}

\begin{proof}
The assertion (1) is nothing but \cite[Lemma~5.4]{MoriwakiEst} and the assertion (2) follows from the assertion (1).

(3): The restriction of $\bm{w}_{F_{\sbullet}}$ to $K^{\times}$ is the same as the usual $\mathfrak{p}$-adic valuation.
Thus $\rk_{\ZZ}\bm{w}_{F_{\sbullet}}(K^{\times})=1$.
The map $\bm{w}_{F_{\sbullet}}$ satisfies the axiom of valuations \cite[Chap.\ VI, \S 3.1]{BourbakiCA72}, so we have, by the arguments in \cite[Chap.\ VI, \S 10.3, Theorem~1]{BourbakiCA72}, $\rk_{\ZZ}\left(\bm{w}_{F_{\sbullet}}(\Rat(X)^{\times})/\bm{w}_{F_{\sbullet}}(E^{\times})\right)\leq\trdeg_E\Rat(X)$ and $\rk_{\ZZ}\left(\bm{w}_{F_{\sbullet}}(E^{\times})/\bm{w}_{F_{\sbullet}}(K^{\times})\right)\leq\trdeg_{K}E$.
Since
\begin{align*}
 \rk_{\ZZ}\bm{w}_{\sbullet}(\Rat(X)^{\times}) &=\rk_{\ZZ}\bm{w}_{F_{\sbullet}}(E^{\times})+\rk_{\ZZ}\left(\bm{w}_{F_{\sbullet}}(\Rat(X)^{\times})/\bm{w}_{F_{\sbullet}}(E^{\times})\right) \\
 &\leq\trdeg_{K}E+1+\trdeg_E\Rat(X) \\
 &=\trdeg_{K}\Rat(X)+1,
\end{align*}
we have $\rk_{\ZZ}\bm{w}_{\sbullet}(E^{\times})=\trdeg_{K}E+1$.
\end{proof}

Given any adelically metrized line bundle $\overline{L}$ on $X$, we set
\begin{equation}\label{eqn:defnsigmainv}
 \delta(\overline{L}):=\inf_{\overline{A}}\frac{\adeg\left(\overline{L}\cdot\overline{A}^{\cdot\dim X}\right)}{\vol(A)}
\end{equation}
(see Proposition~\ref{prop:aintnum}(1)), where the infimum is taken over all the nef adelically metrized line bundles $\overline{A}$ on $X$ such that $\vol(A)$ is positive.

\begin{theorem}[\text{\cite[\S 2.4]{Yuan09}, \cite[Theorem~2.2]{MoriwakiEst}}]\label{thm:Yuan}
Let $X$ be a projective variety that is geometrically irreducible over $K$ and let $\overline{L}$ be an adelically metrized line bundle on $X$ having a $U$-model of definition $(\mathscr{X}_U,\mathscr{L}_U)$.
Let $F_{\sbullet}$ be a good flag on $\mathscr{X}_U'$ over a prime number $p$.
Let $\Gamma$ be any symmetric CL-subset of $\aHzf(\overline{L})$ such that $\Gamma\neq\{0\}$.
Set $\beta:=p\rk_{O_{K}}\aSpan{O_{K}}{\Gamma}\geq 2$.
We then have
\[
 \left|\log\sharp\Gamma-\sharp\bm{w}_{F_{\sbullet}}\left(\Gamma\setminus\{0\}\right)\log(p)\right|\leq\left(\delta(\overline{L})\log(4)+\log(4p)\log(4\beta)\right)\frac{\rk_{\ZZ}\aSpan{\ZZ}{\Gamma}}{\log(p)}.
\]
\end{theorem}

\begin{proof}
We borrow the proof from \cite{Yuan09,MoriwakiEst}.
We divide the proof into four steps.
\medskip

\paragraph{Step 1.}
Set $\overline{F_0}:=\overline{\mathcal{O}}_{X}([\mathfrak{p}])$ (Remark~\ref{rem:remadelic}(4)), which is the adelically metrized line bundle on $X$ associated to the Hermitian line bundle $\left(\mathcal{O}_{\mathscr{X}'}(F_0),|\cdot|_{\infty}^{\rm triv}\right)$ on $\mathscr{X}'$.
Set $M:=\aSpan{\ZZ}{\Gamma}\subset\aHzf(\overline{L})$ and $\Delta:=\Conv_{\aSpan{\RR}{\Gamma}}(\Gamma)$.
Then $\Delta$ is a compact symmetric convex body in $\aSpan{\RR}{\Gamma}$.
For each $n\geq 0$, we set
\begin{equation}
 M_n:=\left\{s\in M\,:\,\ord_{F_0}(s)\geq n\right\}=M\cap\aHzf(\overline{L}-n\overline{F_0})
\end{equation}
and let
\[
 r_n:\aHzf(\overline{L}-n\overline{F_0})\to\Hz(\nu^*\mathscr{L}_U-n\mathcal{O}_{\mathscr{X}_U'}(F_0))\xrightarrow{\text{rest.}}\Hz(\nu^*\mathscr{L}_U|_{F_0}-n\mathcal{O}_{\mathscr{X}_U'}(F_0)|_{F_0})
\]
be the natural homomorphism.
Then
\begin{equation}\label{eqn:decompbymult}
 \sharp\bm{w}_{F_{\sbullet}}\left(\Gamma\setminus\{0\}\right)=\sum_{n\geq 0}\sharp\bm{v}_{F_{\sbullet}}\left(r_n(M_n\cap\Delta)\setminus\{0\}\right).
\end{equation}
\medskip

\paragraph{Step 2.}
In this step, we show that for each $n\geq 0$
\begin{multline}
 \log\sharp(M_n\cap(1/\beta)\Delta)-\log\sharp(M_{n+1}\cap(2/\beta)\Delta) \\
 \leq\sharp\bm{v}_{F_{\sbullet}}\left(r_n(M_n\cap\Delta)\setminus\{0\}\right)\log(p) \\
 \leq\log\sharp(M_n\cap 2\beta\Delta)-\log\sharp(M_{n+1}\cap\beta\Delta).\label{eqn:Yuansineqs1}
\end{multline}
First, we have
\begin{align}
 \sharp\bm{v}_{F_{\sbullet}}\left(r_n(M_n\cap\Delta)\setminus\{0\}\right)\log(p) &=\sharp\bm{v}_{F_{\sbullet}}\left(\aSpan{\FF_p}{r_n(M_n\cap\Delta)}\setminus\{0\}\right)\log(p) \label{eqn:Yuansineqs1-1}\\
 &=\log\sharp\aSpan{\FF_p}{r_n(M_n\cap\Delta)}\nonumber
\end{align}
thanks to \cite[Lemma~1.4]{Lazarsfeld_Mustata08}.
We choose $\left\{s_1,\dots,s_l\right\}\subset M_n\cap\Delta$ such that the image forms an $\FF_p$-basis for $\aSpan{\FF_p}{r_n(M_n\cap\Delta)}$.
Since $l\leq\rk_{O_{K}}\aSpan{O_{K}}{\Gamma}$ and $r_n$ maps
\[
 S:=\left\{\sum_{i=1}^l a_is_i\,:\,a_i=0,\dots,p-1\right\}\subset M_n\cap\beta\Delta
\]
onto $\aSpan{\FF_p}{r_n(M_n\cap\Delta)}$, we have
\begin{equation}\label{eqn:Yuansineqs1-2}
 \log\sharp\aSpan{\FF_p}{r_n(M_n\cap\Delta)}\leq\log\sharp r_n(S)\leq\log\sharp r_n(M_n\cap\beta\Delta).
\end{equation}
Note that $2\ast(M_n\cap\beta\Delta)\subset M_n\cap 2\beta\Delta$ and $\Ker(r_n)=M_{n+1}$.
By applying Lemma~\ref{lem:Yuansineq}(1) to $r_n(M_n\cap\beta\Delta)$, we have
\begin{equation}\label{eqn:Yuansineqs1-3}
 \log\sharp r_n(M_n\cap\beta\Delta)\leq\log\sharp(M_n\cap 2\beta\Delta)-\log\sharp(M_{n+1}\cap\beta\Delta).
\end{equation}
By (\ref{eqn:Yuansineqs1-1})--(\ref{eqn:Yuansineqs1-3}), we have the second inequality of (\ref{eqn:Yuansineqs1}).

Next, we choose $\left\{t_1,\dots,t_{l'}\right\}\subset M_n\cap(1/\beta)\Delta$ such that the image forms an $\FF_p$-basis for $\aSpan{\FF_p}{r_n(M_n\cap(1/\beta)\Delta)}$.
Since $l'\leq\rk_{O_{K}}\aSpan{O_{K}}{\Gamma}$ and $r_n$ maps
\[
 S':=\left\{\sum_{j=1}^{l'} a_jt_j\,:\,a_j=0,\dots,p-1\right\}\subset M_n\cap\Delta
\]
onto $\aSpan{\FF_p}{r_n(M_n\cap(1/\beta)\Delta)}$, we have
\begin{align}
 \sharp\bm{v}_{F_{\sbullet}}\left(r_n(M_n\cap\Delta)\setminus\{0\}\right)\log(p) &\geq\sharp\bm{v}_{F_{\sbullet}}\left(r_n(S')\setminus\{0\}\right)\log(p) \label{eqn:Yuansineqs1-4}\\
 &=\sharp\bm{v}_{F_{\sbullet}}\left(\aSpan{\FF_p}{r_n(M_n\cap(1/\beta)\Delta)}\setminus\{0\}\right)\log(p)\nonumber \\
 &=\log\sharp\aSpan{\FF_p}{r_n(M_n\cap(1/\beta)\Delta)}\nonumber \\
 &\geq\log\sharp r_n(M_n\cap(1/\beta)\Delta)\nonumber
\end{align}
thanks to \cite[Lemma~1.4]{Lazarsfeld_Mustata08} again.
By applying Lemma~\ref{lem:Yuansineq}(1) to $r_n(M_n\cap(1/\beta)\Delta)$,
\begin{equation}\label{eqn:Yuansineqs1-5}
 \log\sharp r_n(M_n\cap(1/\beta)\Delta)\geq\log\sharp(M_n\cap (1/\beta)\Delta)-\log\sharp(M_{n+1}\cap(2/\beta)\Delta).
\end{equation}
By (\ref{eqn:Yuansineqs1-4})--(\ref{eqn:Yuansineqs1-5}), we have the first inequality of (\ref{eqn:Yuansineqs1}).
\medskip

\paragraph{Step 3.}
By (\ref{eqn:decompbymult}) and (\ref{eqn:Yuansineqs1}), we have
\begin{align}
 &\log\sharp(M\cap(1/\beta)\Delta)-\sum_{n\geq 1}\left(\log\sharp(M_n\cap(2/\beta)\Delta)-\log\sharp(M_n\cap(1/\beta)\Delta)\right) \label{eqn:Yuansineqs2}\\
 &\qquad\qquad \leq\sharp\bm{w}_{F_{\sbullet}}\left(\Gamma\setminus\{0\}\right)\log(p) \nonumber\\
 &\qquad\qquad \leq\log\sharp(M\cap 2\beta\Delta)+\sum_{n\geq 1}\left(\log\sharp(M_n\cap 2\beta\Delta)-\log\sharp(M_n\cap\beta\Delta)\right).\nonumber
\end{align}
Thanks to Lemma~\ref{lem:Yuansineq}(2), we have, for each $n\geq 1$,
\begin{align}
 &\log\sharp(M\cap 2\beta\Delta)\leq\log\sharp\Gamma+\log(4\beta)\rk_{\ZZ}M, \label{eqn:Yuansineq1}\\
 &\log\sharp(M\cap(1/\beta)\Delta)\geq\log\sharp\Gamma-\log(2\beta)\rk_{\ZZ}M \label{eqn:Yuansineq2}
\end{align}
and
\begin{align}
 &\log\sharp(M_n\cap 2\beta\Delta)-\log\sharp(M_n\cap\beta\Delta)\leq\log(4)\rk_{\ZZ}M, \label{eqn:Yuansineq3}\\
 &\log\sharp(M_n\cap(2/\beta)\Delta)-\log\sharp(M_n\cap(1/\beta)\Delta)\leq\log(4)\rk_{\ZZ}M. \label{eqn:Yuansineq4}
\end{align}
We set $T:=\left\{n\geq 1\,:\,M_n\cap 2\beta\Delta\neq\{0\}\right\}\supset\left\{n\geq 1\,:\,M_n\cap(2/\beta)\Delta\neq\{0\}\right\}$.
Then, by (\ref{eqn:Yuansineqs2})--(\ref{eqn:Yuansineq4}), we have
\begin{multline}
 -\left(\log(2\beta)+\sharp T\log(4)\right)\rk_{\ZZ}M \\
 \leq\sharp\bm{w}_{F_{\sbullet}}\left(\Gamma\setminus\{0\}\right)\log(p)-\log\sharp\Gamma\leq\left(\log(4\beta)+\sharp T\log(4)\right)\rk_{\ZZ}M.\label{eqn:Yuansineqs3}
\end{multline}
\medskip

\paragraph{Step 4.}
If $M_n\cap 2\beta\Delta\neq\{0\}$, then $\overline{L}-n\overline{F_0}+\overline{\mathcal{O}}_X(\log(2\beta)[\infty])$ is pseudoeffective and
\begin{equation}\label{eqn:Yuansineq5}
 \adeg\left(\left(\overline{L}-n\overline{F_0}+\overline{\mathcal{O}}_X(\log(2\beta)[\infty])\right)\cdot \overline{A}^{\cdot\dim X}\right)\geq 0
\end{equation}
for every nef adelically metrized line bundle $\overline{A}$ on $X$ (Proposition~\ref{prop:aintnum}(1)).
Suppose that $\vol(A)>0$.
By (\ref{eqn:Yuansineq5}) and Lemma~\ref{lem:invofdegree} below, we have
\[
 n\leq\left(\frac{\adeg\left(\overline{L}\cdot\overline{A}^{\cdot\dim X}\right)}{\vol(A)}+\log(2\beta)\right)\frac{1}{\log(p)},
\]
so $\sharp T$ has the same upper bound.
Therefore, we have
\begin{align*}
 &\left|\sharp\bm{w}_{F_{\sbullet}}\left(\Gamma\setminus\{0\}\right)\log(p)-\log\sharp\Gamma\right| \\
 &\qquad\qquad \leq\left(\frac{\adeg\left(\overline{L}\cdot\overline{A}^{\cdot\dim X}\right)}{\vol(A)}\log(4)+\log(4p)\log(4\beta)\right)\frac{\rk_{\ZZ}M}{\log(p)}
\end{align*}
for every nef adelically metrized line bundle $\overline{A}$ on $X$ with $\vol(A)>0$.
\end{proof}

\begin{lemma}\label{lem:invofdegree}
Let $X$ be a projective variety that is geometrically irreducible over $K$.
For any $\mathfrak{p}\in\Spec(O_K)$ and for any nef adelically metrized line bundle $\overline{A}$ on $X$, we have
\[
 \adeg\left(\overline{\mathcal{O}}_X([\mathfrak{p}])\cdot\overline{A}^{\cdot\dim X}\right)=\vol(A)\log(\sharp O_K/\mathfrak{p}).
\]
\end{lemma}

\begin{proof}
Fix a rational number $\varepsilon>0$.
Thanks to Proposition~\ref{prop:adelicapproximation}(2), we can find a finite subset $S\subset M_K^{\rm f}$ and $O_K$-models $(\mathscr{X}',\mathscr{A}_1)$ and $(\mathscr{X}',\mathscr{A}_2)$ of $(X,A)$ such that $\mathscr{A}_i$ are relatively nef and
\begin{equation}
 \left(\mathscr{A}_1^{\rm ad},|\cdot|_{\infty}^{\overline{A}}\right)\leq\overline{A}\leq\left(\mathscr{A}_2^{\rm ad},|\cdot|_{\infty}^{\overline{A}}\right)\leq\left(\mathscr{A}_1^{\rm ad},|\cdot|_{\infty}^{\overline{A}}\right)+\varepsilon\sum_{P\in S}\overline{\mathcal{O}}_X([P]).
\end{equation}
Let $\pi':\mathscr{X}'\to\Spec(O_K)$ denote the structure morphism and fix an adelically metrized line bundle $\overline{H}$ associated to an ample $C^{\infty}$-Hermitian line bundle $\overline{\mathscr{H}}$ on $\mathscr{X}'$.
By invariance of degree, we have
\begin{align*}
 \adeg\left(\overline{\mathcal{O}}_X([\mathfrak{p}])\cdot(\overline{\mathscr{A}}_{i}^{\rm ad}+\delta\overline{\mathscr{H}}^{\rm ad})^{\cdot\dim X}\right) &=\vol((\mathscr{A}_{i}+\delta\mathscr{H})|_{{\pi'}^{-1}(\mathfrak{p})})\log(\sharp O_K/\mathfrak{p}) \\
 &=\vol(A+\delta H)\log(\sharp O_K/\mathfrak{p})
\end{align*}
for every rational number $\delta>0$ and for $i=1,2$.
Therefore, by continuity,
\[
 \adeg\left(\overline{\mathcal{O}}_X([\mathfrak{p}])\cdot(\overline{\mathscr{A}}_{i}^{\rm ad})^{\cdot\dim X}\right)=\vol(A)\log(\sharp O_K/\mathfrak{p})
\]
for $i=1,2$.
Set $\lambda:=-(\varepsilon/[K:\QQ])\sum_{P\in S}\log|\varpi_P|_P$.
Then $\overline{\mathscr{A}}_1^{\rm ad}+\overline{\mathcal{O}}_X(\lambda[\infty])$ is nef.
By Proposition~\ref{prop:aintnum}(1),(2)
\begin{align*}
 \adeg\left(\overline{\mathcal{O}}_X([\mathfrak{p}])\cdot(\overline{\mathscr{A}}_{1}^{\rm ad})^{\cdot\dim X}\right) &=\adeg\left(\overline{\mathcal{O}}_X([\mathfrak{p}])\cdot(\overline{\mathscr{A}}_{1}^{\rm ad}+\overline{\mathcal{O}}_X(\lambda[\infty]))^{\cdot\dim X}\right) \\
 &\leq\adeg\left(\overline{\mathcal{O}}_X([\mathfrak{p}])\cdot(\overline{A}+\overline{\mathcal{O}}_X(\lambda[\infty]))^{\cdot\dim X}\right) \\
 &=\adeg\left(\overline{\mathcal{O}}_X([\mathfrak{p}])\cdot\overline{A}^{\cdot\dim X}\right) \\
 &\leq\adeg\left(\overline{\mathcal{O}}_X([\mathfrak{p}])\cdot(\overline{\mathscr{A}}_{2}^{\rm ad})^{\cdot\dim X}\right).
\end{align*}
Hence we have the assertion.
\end{proof}

\section{Numbers of restricted sections}\label{sec:arithrestvol}

In this section, we study the asymptotic behavior of the numbers of restricted strictly small sections in general, and show Theorem~B (Theorems~\ref{thm:convergence}).

\begin{definition}\label{defn:arithrestvol}
Let $X$ be a projective variety over a number field $K$, let $Y$ be a closed subvariety of $X$ with number field $K_Y:=\Hz(\mathcal{O}_Y)$, and let $\overline{L}$ be an adelically metrized line bundle on $X$.
We set
\begin{equation}
 \aHzsq{?}{X|Y}(\overline{L}):=\Image\left(\aHzs{?}(\overline{L})\to\aHzf(\overline{L}|_Y)\right)
\end{equation}
for $?=\text{s}$ and ss, and set
\begin{equation}
 \CLq{X|Y}(\overline{L}):=\CL_{\aHzf(\overline{L}|_Y)}\left(\aHzstq{X|Y}(\overline{L})\right),
\end{equation}
where $\aHzf(\overline{L}|_Y)$ is regarded as a free $\ZZ$-module containing $\aHzstq{X|Y}(\overline{L})$.
Moreover, we set
\begin{equation}
 \aN_{X|Y}(\overline{L}):=\left\{m\in\NN\,:\,\aHzstq{X|Y}(m\overline{L})\neq\{0\}\right\}
\end{equation}
and
\begin{equation}
 \akappaq{X|Y}(\overline{L}):=\begin{cases} \trdeg_{K_Y}\left(\bigoplus_{m\geq 0}\aSpan{K_Y}{\aHzstq{X|Y}(m\overline{L})}\right)-1 & \text{if $\aN_{X|Y}(\overline{L})\neq\emptyset$,} \\ -\infty & \text{if $\aN_{X|Y}(\overline{L})=\emptyset$.}\end{cases}
\end{equation}
Then we define the \emph{arithmetic restricted volume of $\overline{L}$ along $Y$} as
\begin{equation}
 \avolq{X|Y}(\overline{L}):=\limsup_{m\to+\infty}\frac{\log\sharp\CLq{X|Y}(m\overline{L})}{m^{\dim Y+1}/(\dim Y+1)!}.
\end{equation}
\begin{description}
\item[($Y$-effective)] We say that an adelically metrized line bundle $\overline{L}$ is \emph{$Y$-effective} if there exists an $s\in\aHzsm(\overline{L})$ such that $s|_Y$ is non-zero.
We write $\overline{L}_1\leq_Y\overline{L}_2$ if $\overline{L}_2-\overline{L}_1$ is $Y$-effective.
\item[($Y$-big)] We say that an adelically metrized $\QQ$-line bundle $\overline{L}$ is \emph{$Y$-big} if there exist an $a\geq 1$ and a w-ample adelically metrized line bundle $\overline{A}$ such that $a\overline{L}\geq_Y\overline{A}$.
\item[($Y$-pseudoeffective)] We say that an adelically metrized $\QQ$-line bundle $\overline{L}$ is \emph{$Y$-pseudoeffective} if $\overline{L}+\overline{A}$ is $Y$-big for every $Y$-big adelically metrized $\QQ$-line bundle $\overline{A}$.
We write $\overline{L}_1\preceq_Y\overline{L}_2$ if $\overline{L}_2-\overline{L}_1$ is $Y$-pseudoeffective.
\end{description}
\end{definition}

\begin{remark}
Let $\overline{L}$ be an adelically metrized line bundle on $X$.
\begin{enumerate}
\item[(1)] If $s\in\CLq{X|Y}(m\overline{L})$ and $t\in\CLq{X|Y}(n\overline{L})$, then $s\otimes t\in\CLq{X|Y}((m+n)\overline{L})$.
In fact, we can write $s=\sum a_is_i$ and $t=\sum b_jt_j$, where $a_i,b_j\in\QQ_{\geq 0}$ with $\sum a_i=\sum b_j=1$, $s_i\in\aHzstq{X|Y}(m\overline{L})$, and $t_j\in\aHzstq{X|Y}(n\overline{L})$.
Then $s\otimes t=\sum a_ib_js_i\otimes t_j$ and $\sum a_ib_j=1$.
\item[(2)] For any subfield $K'$ of $K_Y$,
\[
 \trdeg_{K'}\left(\bigoplus_{m\geq 0}\aSpan{K'}{\aHzstq{X|Y}(m\overline{L})}\right)-1
\]
does not depend on the choice of $K'$ and coincides with $\akappaq{X|Y}(\overline{L})$.
\item[(3)] In \cite{MoriwakiEst}, Moriwaki defined the arithmetic restricted volume of $\overline{L}$ along $Y$ as
\[
 \avolsmq{X|Y}(\overline{L}):=\limsup_{m\to+\infty}\frac{\log\sharp\CL_{\aHzf(m\overline{L}|_Y)}(\aHzsmq{X|Y}(m\overline{L}))}{m^{\dim Y+1}/(\dim Y+1)!}.
\]
Obviously, $\avolq{X|Y}(\overline{L})\leq\avolsmq{X|Y}(\overline{L})$ and the equality holds if $\overline{L}$ is $Y$-big.
\end{enumerate}
\end{remark}

\begin{lemma}\label{lem:effectiveincrease}
\begin{enumerate}
\item[\textup{(1)}] Let $\overline{L}_1,\overline{L}_2\in\aPic(X)$.
If $\overline{L}_1\leq_Y\overline{L}_2$, then
\[
 \avolq{X|Y}(\overline{L}_1)\leq\avolq{X|Y}(\overline{L}_2)\quad\text{and}\quad\akappaq{X|Y}(\overline{L}_1)\leq\akappaq{X|Y}(\overline{L}_2).
\]
\item[\textup{(2)}] Suppose that $X$ is normal and let $\overline{L}\in\aPic(X)$.
Let $\varphi:X'\to X$ be a birational projective $K$-morphism and let $Y'$ be a closed subvariety of $X'$ such that $\varphi(Y')=Y$.
Then
\[
 \avolq{X'|Y'}(\varphi^*\overline{L})=\avolq{X|Y}(\overline{L})\quad\text{and}\quad\akappaq{X'|Y'}(\varphi^*\overline{L})=\akappaq{X|Y}(\overline{L}).
\]
\end{enumerate}
\end{lemma}

\begin{proof}
(1): An $s\in\aHzsm(\overline{L}_2-\overline{L}_1)$ with $s|_Y\neq 0$ determines for all $m\geq 1$ injections $\aHzstq{X|Y}(m\overline{L}_1)\xrightarrow{\otimes (s|_Y)^{\otimes m}}\aHzstq{X|Y}(m\overline{L}_2)$.
So we have the assertion.

(2): Since $X$ is normal, $\Hz(L)=\Hz(\varphi^*L)$ as $K$-vector spaces.
Since $|\cdot|^{\varphi^*\overline{L}}_v(x)=|\cdot|_v^{\overline{L}}(\varphi_v^{\rm an}(x))$ for $x\in {X_v'}^{\rm an}$ and $v\in M_K$,
\begin{equation}
 \varphi^*:\left(\Hz(L),(\|\cdot\|_{v,\sup}^{\overline{L}})_{v\in M_K}\right)\xrightarrow{\sim}\left(\Hz(\varphi^*L),(\|\cdot\|_{v,\sup}^{\varphi^*\overline{L}})_{v\in M_K}\right).
\end{equation}
as adelically normed $K$-vector spaces.
By considering the commutative diagram
\[
\xymatrix{\Hz(mL|_Y) \ar[rr] && \Hz(\varphi^*(mL)|_{Y'}) \\ \Hz(mL) \ar[u] \ar[rr]^-{\sim} && \Hz(\varphi^*(mL)), \ar[u]
}
\]
we have $\aHzstq{X|Y}(m\overline{L})=\aHzstq{X'|Y'}(\varphi^*(m\overline{L}))$ for every $m\geq 1$.
\end{proof}

Set $\Hzq{X|Y}(L):=\Image(\Hz(L)\to\Hz(L|_Y))$,
\begin{equation}
 \mathbf{N}_{X|Y}(L):=\left\{m\in\NN\,:\,\Hzq{X|Y}(mL)\neq\{0\}\right\},
\end{equation}
and
\begin{equation}
 \kappa_{X|Y}(L):=\begin{cases} \trdeg_{K_Y}\left(\bigoplus_{m\geq 0}\aSpan{K_Y}{\Hzq{X|Y}(mL)}\right)-1 & \text{if $\mathbf{N}_{X|Y}(L)\neq\emptyset$,} \\ -\infty & \text{if $\mathbf{N}_{X|Y}(L)=\emptyset$.} \end{cases}
\end{equation}

\begin{lemma}\label{lem:kappaqequal}
We have
\[
\akappaq{X|Y}(\overline{L})=\begin{cases} \kappa_{X|Y}(L) & \text{if $Y\not\subset\aSBs(\overline{L})$,} \\ -\infty & \text{if $Y\subset\aSBs(\overline{L})$.} \end{cases}
\]
\end{lemma}

\begin{proof}
We can assume that $Y\not\subset\aSBs(\overline{L})$ and the inequality $\leq$ is obvious.
Since $\kappa_{X|Y}(aL)=\kappa_{X|Y}(L)$ and $\akappaq{X|Y}(a\overline{L})=\akappaq{X|Y}(\overline{L})$ for every $a\geq 1$, we can also assume that $\aHzstq{X|Y}(\overline{L})\ni s\neq 0$.
Set
\[
 E_1:=\left\{\frac{\phi}{\psi}\in\Rat(Y)\,:\,\text{$\exists m\geq 1$ s.t.\ $\phi (s|_Y)^{\otimes m},\psi (s|_Y)^{\otimes m}\in\aSpan{K_Y}{\Hzq{X|Y}(mL)}$}\right\}
\]
and
\[
 E_2:=\left\{\frac{\phi}{\psi}\in\Rat(Y)\,:\,\text{$\exists m\geq 1$ s.t.\ $\phi (s|_Y)^{\otimes m},\psi (s|_Y)^{\otimes m}\in\aSpan{K_Y}{\aHzstq{X|Y}(m\overline{L})}$}\right\}.
\]
Then
\[
 E_1(T)\to\Frac\left(\bigoplus_{m\geq 0}\aSpan{K_Y}{\Hzq{X|Y}(mL)}\right),\quad Q(T)\mapsto Q(s|_Y)
\]
and
\[
 E_2(T)\to\Frac\left(\bigoplus_{m\geq 0}\aSpan{K_Y}{\aHzst(m\overline{L})}\right),\quad Q(T)\mapsto Q(s|_Y)
\]
give isomorphisms of fields of finite type over $K_Y$.
So what we have to show is $E_1=E_2$.
Given any $t\in\aSpan{K_Y}{\Hzq{X|Y}(mL)}$, $t\otimes (s|_Y)^{\otimes p}\in\aSpan{K_Y}{\aHzstq{X|Y}(m\overline{L})}$ for sufficiently large $p\geq1$.
Thus we conclude.
\end{proof}

Let $X$ be a projective variety over $K$, let $Y$ be a closed subvariety of $X$, and let $\overline{L}$ be an adelically metrized line bundle on $X$.
Put $K_Y:=\Hz(\mathcal{O}_Y)$ and $\overline{M}:=\overline{L}|_Y$.
In the rest of this section, we fix models of $X$ and $(Y,M)$ as follows.
By Lemma~\ref{lem:modelofdefn}(1), there exists an $O_K$-model $(\mathscr{X},\mathscr{L})$ of $(X,L)$ such that $\mathscr{X}$ is relatively normal in $X$ and $\mathscr{L}$ is a line bundle on $\mathscr{X}$.
Let $\mathscr{Y}$ be the Zariski closure of $Y$ in $\mathscr{X}$ and let $\mathscr{M}:=\mathscr{L}|_{\mathscr{Y}}$.
By Lemma~\ref{lem:modelofdefn}(2), one can find a non-empty open subset $U_0$ of $\Spec(O_K)$ such that $(\mathscr{Y}_{U_0},\mathscr{M}_{U_0})$ gives a $U_0$-model of definition for $\overline{M}^{\{\infty\}}$, where we have set
\begin{equation}
 \mathscr{Y}_{U_0}:=\mathscr{Y}\times_{\Spec(O_K)}U_0.
\end{equation}
By Lemma~\ref{lem:modelofdefn} and Remark~\ref{rem:remadelic}(2), one can also find a non-empty open subset $U$ of $\Spec(O_{K_Y})$ such that the morphism $\Spec(O_{K_Y})\to\Spec(\Hz(\mathcal{O}_{\mathscr{Y}}))$ is isomorphic over $U$, $U$ is mapped into $U_0$ via $\Spec(O_{K_Y})\to\Spec(O_K)$, and $(\mathscr{Y}_{U},\mathscr{M}_{U})$ gives a $U$-model of definition for $\overline{M}^{\{\infty\}}$, where we have set
\begin{equation}
 \mathscr{Y}_{U}:=\mathscr{Y}\times_{\Spec(\Hz(\mathcal{O}_{\mathscr{Y}}))}U\subset\mathscr{Y}_{U_0}.
\end{equation}

Let $\nu:\mathscr{Y}'\to\mathscr{Y}$ be a relative normalization in $Y$.
Then $\mathscr{Y}'\to\Spec(\mathcal{O}_{\mathscr{Y}})$ factorizes through $\Spec(O_{K_Y})$ (see \cite[(6.3.3)]{EGAII}).
We fix a good flag
\begin{equation}
 F_{\sbullet}:\mathscr{Y}_U'\supset F_0\supset F_1\supset\dots\supset F_{\dim Y}=\{\xi\}
\end{equation}
on $\mathscr{Y}_U'$ over a prime number $p$.

\begin{definition}\label{defn:languageofS}
As in (\ref{eqn:definitionvalvector}), we denote for each $m\geq 1$ the valuation map associated to $F_{\sbullet}$ by $\bm{w}_{F_{\sbullet}}:\aHzf(m\overline{M})\to\Hz(\nu^*(m\mathscr{M}_U))\to\ZZ^{\dim Y+1}$.
Let
\[
 \aS_{X|Y}(\overline{L}):=\left\{(m,\bm{w}_{F_{\sbullet}}(s))\,:\,\forall m\in\aN_{X|Y}(\overline{L}),\,\forall s\in\aHzstq{X|Y}(m\overline{L})\setminus\{0\}\right\}\subset\NN\times\ZZ^{\dim Y+1}.
\]
Let $\pr_1:\RR\times\RR^{\dim Y+1}\to\RR$ and $\pr_2:\RR\times\RR^{\dim Y+1}\to\RR^{\dim Y+1}$ be the natural projections, and set $\aS_{X|Y}(\overline{L})_m:=\aS_{X|Y}(\overline{L})\cap\pr_1^{-1}(m)$ for $m\geq 1$.
The \emph{base of $\aS_{X|Y}(\overline{L})$} is defined as
\[
 \aDelta_{X|Y}(\overline{L}):=\overline{\left(\bigcup_{m\geq 1}\frac{1}{m}\aS_{X|Y}(\overline{L})_m\right)}\subset\RR^{\dim Y+1},
\]
and the affine space in $\RR^{\dim Y+1}$ spanned by $\aDelta_{X|Y}(\overline{L})$ is
\[
 \Aff(\aS_{X|Y}(\overline{L})):=\pr_2(\aSpan{\RR}{\aS_{X|Y}(\overline{L})}\cap\pr_1^{-1}(1)).
\]
The underlying $\RR$-vector space $\vAff(\aS_{X|Y}(\overline{L}))=\pr_2(\aSpan{\RR}{\aS_{X|Y}(\overline{L})}\cap\pr_1^{-1}(0))$ has the natural integral structure defined by
\[
 \vAff_{\ZZ}(\aS_{X|Y}(\overline{L})):=\pr_2(\aSpan{\ZZ}{\aS_{X|Y}(\overline{L})}\cap\pr_1^{-1}(0))
\]
(see \cite{BoucksomBourbaki} for details).

If $\akappaq{X|Y}(\overline{L})=\dim Y$, then we set $|\aS_{X|Y}(\overline{L})|$ as the volume of the fundamental domain of $\vAff_{\ZZ}(\aS_{X|Y}(\overline{L}))\subset\RR^{\dim Y+1}$ measured by $\vol_{\RR^{\dim Y+1}}$, and, if $\akappaq{X|Y}(\overline{L})<\dim Y$, then set $|\aS_{X|Y}(\overline{L})|:=0$.
\end{definition}

\begin{lemma}\label{lem:propertiesakappa}
Suppose that $Y\not\subset\aSBs(\overline{L})$, or equivalently, $\akappaq{X|Y}(\overline{L})\geq 0$.
\begin{enumerate}
\item[\textup{(1)}] For every sufficiently large $m\in\aN_{X|Y}(\overline{L})$,
\[
 \akappaq{X|Y}(\overline{L})=\max_{m\in\aN_{X|Y}(\overline{L})}\left\{\dim\aPhi_{m\overline{L},K}(Y)\right\}=\dim\aPhi_{m\overline{L},K}(Y).
\]
\item[\textup{(2)}] Suppose that there exist an $m_0\geq 1$ and an $s_0\in\aHzst(m_0\overline{L})$ such that the restriction $s_0|_Y\in\aHzst(m_0\overline{L}|_Y)\subset\Hz(m_0\mathscr{M}_U)$ does not vanish at the center $\xi$ of $F_{\sbullet}$.
Then
\[
 \dim_{\RR}\aSpan{\RR}{\aS_{X|Y}(\overline{L})}=\akappaq{X|Y}(\overline{L})+2.
\]
\end{enumerate}
\end{lemma}

\begin{proof}
(1): We have
\begin{equation}
 \akappaq{X|Y}(\overline{L})=\dim\overline{\Image\left(\aPhi_{m,K}:Y\dashrightarrow\PP_K(\aSpan{K}{\aHzstq{X|Y}(m\overline{L})})\right)}^{\rm Zar}
\end{equation}
for every sufficiently large $m\in\aN_{X|Y}(\overline{L})$ by applying the arguments in \cite[Lemma~2.3]{Bou_Cac_Lop13} and \cite[\S 8.2.1, Theorem~A]{EisenbudCA} (see also \cite[Th\'eor\`eme~3.15]{BoucksomBourbaki}).
Since $\aSpan{K}{\aHzst(m\overline{L})}\to\aSpan{K}{\aHzstq{X|Y}(m\overline{L})}$ is surjective, we have a commutative diagram
\[
\xymatrix{ X \ar@{-->}[rr]^-{\aPhi_{m\overline{L},K}} && \PP_K(\aSpan{K}{\aHzst(m\overline{L})}) \\ Y \ar[u] \ar@{-->}[rr]^-{\aPhi_{m,K}} && \PP_K(\aSpan{K}{\aHzstq{X|Y}(m\overline{L})}) \ar[u]
}
\]
and $\aPhi_{m,K}$ is the restriction of $\aPhi_{m\overline{L},K}$ to $Y$ for each $m\in\aN_{X|Y}(\overline{L})$.

(2): We write $\kappa:=\akappaq{X|Y}(\overline{L})$.
Let $\mathscr{M}'$ be a line bundle on $\mathscr{Y}'$ such that $(\mathscr{Y}_{U_0}',\mathscr{M}_{U_0}')$ is a $U_0$-model of definition for $\overline{M}^{\{\infty\}}$ and $\overline{M}\leq\overline{\mathscr{M}'}^{\rm ad}$.
For every $a\geq 1$, we then have $\aHzf(a\overline{M})\subset\Hz(a\mathscr{M}')$.
We choose an $a\gg 1$ having the properties that $m_0$ divides $a$, $\aBs(a\overline{L})=\aSBs(\overline{L})$, and the Zariski closure of $\aPhi_{a,K}(Y)$ has Krull dimension $\kappa$ (see the assertion (1)).
Let $s\in\aHzst(a\overline{L})$ be the tensor power of $s_0$.
The Kodaira map
\[
 \aPhi_{a,K}:Y\dashrightarrow Q:=\PP_K(\aSpan{K}{\aHzstq{X|Y}(a\overline{L})})
\]
extends to a Kodaira map
\[
\aPhi_{a,O_K}:\mathscr{Y}'\dashrightarrow\mathscr{Q}\subset\mathscr{P},
\]
where $\mathscr{Q}:=\PP_{O_K}(\aSpan{O_K}{\aHzstq{X|Y}(a\overline{L})})$ and $\mathscr{P}:=\PP_{O_K}(\aSpan{O_K}{\aHzst(a\overline{L})})$.
By the hypothesis, $\aPhi_{a,O_K}$ is defined at $\xi$.
Let $\mathscr{H}$ be the hyperplane line bundle on $\mathscr{P}$.
For every $b\geq 1$, we then have $\Hz(b\mathscr{H})=\Sym_{O_K}^b\aSpan{O_K}{\aHzst(a\overline{L})}$.

Let $\mathscr{Z}$ be the Zariski closure of $\aPhi_{a,O_K}(\mathscr{Y}')$ and let $\mathscr{Z}_s:=\{y\in\mathscr{Z}\,:\,s(y)\neq 0\}$.
Let $E$ be the field of rational functions on $\mathscr{Z}$ and let $R:=\Hz(\mathcal{O}_{\mathscr{Z}_s})$.
Note that $\mathcal{O}_{\mathscr{Z}_s}\xrightarrow{\sim}\mathscr{H}|_{\mathscr{Z}_s}$, $\phi\mapsto \phi (s|_{\mathscr{Z}_s})$, is isomorphic.
Since the field of fractions of $R$ is $E$ and $E$ is a subextension of $\Rat(Y)/K$, one can find, by Lemma~\ref{lem:valmapissurj}(3), $\phi_0,\dots,\phi_{\kappa}\in R\setminus\{0\}$ such that $\bm{w}_{F_{\sbullet}}(\phi_0),\dots,\bm{w}_{F_{\sbullet}}(\phi_{\kappa})$ are $\ZZ$-linearly independent.
For some $b\gg 1$, $\phi_i (s|_{\mathscr{Z}})^{\otimes b}$ extends to a global section of $b\mathscr{H}|_{\mathscr{Z}}$ for every $i$, and the restriction $\Hz(b\mathscr{H})\to\Hz(b\mathscr{H}|_{\mathscr{Z}})$ is surjective.

For each $i$, we choose a lift of $\phi_i (s|_{\mathscr{Z}})^{\otimes b}\in\Hz(b\mathscr{H}|_{\mathscr{Z}})$ to $\Hz(b\mathscr{H})$, and let $e_i$ be the image of the lift via
\[
 \Hz(b\mathscr{H})=\Sym_{O_K}^b\aSpan{O_K}{\aHzst(a\overline{L})}\to\aSpan{O_K}{\aHzst(ab\overline{L})}.
\]
By tensoring $s$ furthermore, we have $e_i':=e_i\otimes s^{\otimes c}\in\aHzst(a(b+c)\overline{L})$ for every $i$.
The restriction $s|_{Y}\in\aHzst(a\overline{M})\subset\Hz(a\mathscr{M}')$ gives a local frame of $a\mathscr{M}'$ on the open neighborhood $\mathscr{Y}_{s}'$ of $\xi$.
By construction,
\[
 (a,\bm{w}_{F_{\sbullet}}(s|_Y)),\,(a(b+c),\bm{w}_{F_{\sbullet}}(e_0'|_Y)),\,\dots,\,(a(b+c),\bm{w}_{F_{\sbullet}}(e_{\kappa}'|_Y))
\]
are, respectively, equal to
\[
 (a,0,\dots,0),\,(a(b+c),\bm{w}_{F_{\sbullet}}(\phi_0)),\,\dots,\,(a(b+c),\bm{w}_{F_{\sbullet}}(\phi_{\kappa})).
\]
So they are $\ZZ$-linearly independent.
\end{proof}

\begin{lemma}\label{lem:propertiesYbig}
Let $Y$ be a closed subvariety of $X$ and let $\overline{L}$ be an adelically metrized line bundle on $X$.
\begin{enumerate}
\item[\textup{(1)}] If $\overline{L}$ is $Y$-big, then $\aN_{X|Y}(\overline{L})\supset\{m\in\NN\,:\,m\geq l_0\}$ for some $l_0\geq 1$.
\item[\textup{(2)}] The following are equivalent.
\begin{enumerate}
\item[\textup{(a)}] $\overline{L}$ is $Y$-big.
\item[\textup{(b)}] Given any adelically metrized line bundle $\overline{N}$, $m\overline{L}+\overline{N}$ is $Y$-effective for every $m\gg 1$.
\item[\textup{(c)}] $Y\not\subset\aBsp(\overline{L})$.
\end{enumerate}
\item[\textup{(3)}] If $\overline{L}$ is $Y$-big, then $\akappaq{X|Y}(\overline{L})=\dim Y$.
\item[\textup{(4)}] If $\overline{L}$ is $Y$-big, then $\aS_{X|Y}(\overline{L})$ generates $\ZZ\times\ZZ^{\dim Y+1}$.
\end{enumerate}
\end{lemma}

\begin{proof}
(1): There exist a free and w-ample adelically metrized line bundle $\overline{A}$ and an $a\geq 1$ such that $a\overline{L}\geq_Y\overline{A}$.
By Lemma~\ref{lem:w-ample}(1), there exists a $b\geq 1$ such that $\overline{L}+b\overline{A}$ is free.
Thus $a\overline{L}$ and $(ab+1)\overline{L}$ are both $Y$-effective.
The assertion follows from the following claim.

\begin{claim}\label{clm:coprimeintegers}
Let $a,b\in\NN$ be coprime positive integers.
Then there exists a $c_0\gg 1$ such that
\[
 \left\{ax+by\,:\,x,y\in\NN\right\}\supset\left\{c\in\NN\,:\,c\geq c_0\right\}.
\]
\end{claim}

\begin{proof}[Proof of Claim~\ref{clm:coprimeintegers}]
We may assume that $a>1$, $b>1$, and $ax_0-by_0=1$ for some $x_0,y_0\in\NN$.
For every $0\leq c<b$ and every $y\geq cy_0$, we have
\[
 c+by=c(ax_0-by_0)+by\in\left\{ax+by\,:\,x,y\in\NN\right\}.
\]
So we can set $c_0:=b(b-1)y_0$.
\end{proof}

(2): (b) $\Rightarrow$ (a) is clear.

(a) $\Rightarrow$ (b): There exist an $a\geq 1$ and a w-ample adelically metrized line bundle $\overline{A}$ such that $a\overline{L}\geq_Y\overline{A}$.
By Lemma~\ref{lem:w-ample}(1), $m\overline{A}+\overline{N}$ is free for every $m\gg 1$.
So $ma\overline{L}+\overline{N}$ is $Y$-effective for every $m\gg 1$.
By the assertion (1), $(ma+r)\overline{L}$ is $Y$-effective for $r=0,1,\dots,(a-1)$ and every $m\gg 1$.
Thus $(ma+r)\overline{L}+\overline{N}$ is $Y$-effective for $r=0,1,\dots,(a-1)$ and every $m\gg 1$.

(a) $\Rightarrow$ (c): There exist an $a\geq 1$ and a w-ample adelically metrized line bundle $\overline{A}$ such that $a\overline{L}\geq_Y\overline{A}$.
Then $Y\not\subset\aSBs(\overline{L}-(1/a)\overline{A})$.

(c) $\Rightarrow$ (a): There exists a w-ample adelically metrized $\QQ$-line bundle $\overline{A}$ such that $Y\not\subset\aSBs(\overline{L}-\overline{A})$.
Thus there exists an $a\geq 1$ such that $a\overline{L}\geq_Ya\overline{A}$.

(3): By Lemmas~\ref{lem:Kodairatype}, \ref{lem:propertiesakappa}, and the assertions (1),(2), $\akappaq{X|Y}(\overline{L})=\dim\aPhi_{m\overline{L},K}(Y)=\dim Y$ for every sufficiently large $m\in\aN_{X|Y}(\overline{L})$.

(4): Let $\mathscr{X}$ be the $O_K$-model of $X$ fixed before.
By Lemma~\ref{lem:valmapissurj}(1), one can find an ample line bundle $\mathscr{A}$ on $\mathscr{X}$,
\[
 s_0,s_1,\dots,s_{\dim Y+1}\in\Hz(\mathscr{A}),\quad\text{and an}\quad s\in\aHzf(\overline{L}^{\{\infty\}}+\mathscr{A}^{\rm ad})
\]
such that $\Hz(\mathscr{A})\to\Hz(\mathscr{A}|_{\mathscr{Y}})$ is surjective,
\[
 \bm{w}_{F_{\sbullet}}(s_0|_Y)=(0,\dots,0),\,\bm{w}_{F_{\sbullet}}(s_1|_Y)=(1,0,\dots,0),\,\dots,\,\bm{w}_{F_{\sbullet}}(s_{\dim Y+1}|_Y)=(0,\dots,0,1)
\]
in $\ZZ^{\dim Y+1}$, and $s$ does not vanish at the center of $F_{\sbullet}$.
One can choose a suitable metric on $\mathscr{A}$ such that $s_0,\dots,s_{\dim Y+1}\in\aHzst(\overline{\mathscr{A}}^{\rm ad})$ and $s\in\aHzst(\overline{L}+\overline{\mathscr{A}}^{\rm ad})$.
Set $\overline{A}:=\overline{\mathscr{A}}^{\rm ad}$.
By the assertion (2), there exist an $a\geq 1$ and a $t\in\aHzsm(a\overline{L}-\overline{A})$ such that $t|_Y\neq 0$.
Then
\[
 (a+1,\bm{w}_{F_{\sbullet}}((s\otimes t)|_Y))-(a,\bm{w}_{F_{\sbullet}}((s_0\otimes t)|_Y))=(1,0,\dots,0)\in\ZZ\times\ZZ^{\dim Y+1}
\]
and
\[
 (a,\bm{w}_{F_{\sbullet}}((s_i\otimes t)|_Y))-(a,\bm{w}_{F_{\sbullet}}((s_0\otimes t)|_Y))=(0,\dots,0,\overset{i}{1},0,\dots,0)\in\ZZ\times\ZZ^{\dim Y+1}
\]
for $i=1,\dots,\dim Y+1$.
\end{proof}

\begin{proposition}\label{prop:KKOk}
\begin{enumerate}
\item[\textup{(1)}] $\aDelta_{X|Y}(\overline{L})$ is a compact convex body in $\vAff(\aS_{X|Y}(\overline{L}))$.
\item[\textup{(2)}] Let $\vol_{\aS_{X|Y}(\overline{L})}$ be the Euclidean measure on $\Aff(\aS_{X|Y}(\overline{L}))$ normalized by the integral structure $\vAff_{\ZZ}(\aS_{X|Y}(\overline{L}))$.
Then
\begin{equation}\label{eqn:KKOk}
 \vol_{\aS_{X|Y}(\overline{L})}\left(\aDelta_{X|Y}(\overline{L})\right)=\lim_{\substack{m\in\aN_{X|Y}(\overline{L}), \\ m\to+\infty}}\frac{\sharp\bm{w}_{F_{\sbullet}}\left(\aHzstq{X|Y}(m\overline{L})\setminus\{0\}\right)}{m^{\akappaq{X|Y}(\overline{L})+1}}\quad\in\RR_{>0}.
\end{equation}
\end{enumerate}
\end{proposition}

\begin{proof}
Let $\overline{F_0}:=\overline{\mathcal{O}}_{Y}([\mathfrak{p}])$ (Remark~\ref{rem:remadelic}) and choose an adelically metrized line bundle $\overline{A}:=\overline{\mathscr{A}}^{\rm ad}$ that is associated to an ample $C^{\infty}$-Hermitian line bundle $\overline{\mathscr{A}}$ on $\mathscr{Y}'$.
For any non-zero section $s\in\aHzsm(m\overline{M})\setminus\{0\}$, $m\overline{M}-w_0(s)\overline{F_0}$ is pseudoeffective.
Thus, by Proposition~\ref{prop:aintnum}(1),
\[
 \frac{w_0(s)}{m}\leq\frac{\adeg\left(\overline{M}\cdot\overline{A}^{\cdot \dim Y}\right)}{\deg\left((\mathscr{A}|_{F_0})^{\cdot \dim Y}\right)}=:b'.
\]
On the other hand, by using a general result \cite[Lemme~3.5]{BoucksomBourbaki}, one can find a constant $b''>0$ such that
\[
 w_i(s)\leq (m_1+m_2)b''
\]
for $i=1,\dots,\dim Y$, $m_1\geq 1$, $m_2\geq 1$, and every $s\in\Hz(\nu^*(m_1\mathscr{M}_{U})|_{F_0}+m_2\mathcal{O}_{\mathscr{Y}_U'}(-F_0)|_{F_0})\setminus\{0\}$.
Set $b:=(1+b')b''$.
Then $w_i(s)\leq mb$ for $i=0,\dots,\dim Y$, $m\geq 1$, and every $s\in\aHzsm(m\overline{M})\setminus\{0\}$.
The last formula (\ref{eqn:KKOk}) follows from the assertion (1), Lemma~\ref{lem:propertiesakappa}(2), and \cite[Th\'eor\`eme~1.12]{BoucksomBourbaki}.
\end{proof}

We write
\[
 \aRq{X|Y}(\overline{L})_{\sbullet}:=\bigoplus_{m\geq 0}\aSpan{K_Y}{\aHzstq{X|Y}(m\overline{L})}.
\]
Then $\kappa\left(\aRq{X|Y}(\overline{L})_{\sbullet}\right)=\akappaq{X|Y}(\overline{L})$.
By arguing over an algebraic closure of $K_Y$ and applying \cite[Th\'eor\`eme~3.7]{BoucksomBourbaki} (or \cite[Corollary~3.11]{Kaveh_Khovanskii12}), the sequence
\begin{equation}
 \left(\frac{\dim_{K_Y}\aRq{X|Y}(\overline{L})_m}{m^{\akappaq{X|Y}(\overline{L})}/\akappaq{X|Y}(\overline{L})!}\right)_{m\in\aN_{X|Y}(\overline{L})}
\end{equation}
converges to a positive real number $e\left(\aRq{X|Y}(\overline{L})_{\sbullet}\right)$.

\begin{proposition}\label{prop:Yuan}
Set
\[
 D=D_{X|Y}(\overline{L}):=\log(4)[K_Y:\QQ]\delta(\overline{L}|_Y)\cdot\frac{e\left(\aRq{X|Y}(\overline{L})_{\sbullet}\right)}{\akappaq{X|Y}(\overline{L})!}
\]
(see (\ref{eqn:defnsigmainv}) for definition of $\delta(\overline{L}|_Y)$).
Suppose that there exists an $s_0\in\aHzstq{X|Y}(m_0\overline{L})\setminus\{0\}$ for an $m_0\geq 1$.
Let $\Psi:=\left\{y\in\mathscr{Y}_U'\,:\,s_0(y)=0\right\}$ and let $F_{\sbullet}$ be a $\Psi$-good flag on $\mathscr{Y}_U'$ over a prime number $p$.
Then
\[
 \limsup_{\substack{m\in\aN_{X|Y}(\overline{L}), \\ m\to+\infty}}\left|\frac{\log\sharp\CLq{X|Y}(m\overline{L})}{m^{\akappaq{X|Y}(\overline{L})+1}}-\frac{\sharp\bm{w}_{F_{\sbullet}}(\aHzstq{X|Y}(m\overline{L})\setminus\{0\})\log(p)}{m^{\akappaq{X|Y}(\overline{L})+1}}\right|\leq\frac{D}{\log(p)}.
\]
\end{proposition}

\begin{proof}
Set
\begin{align*}
 D_m:=& \left(\log(4)\delta(m\overline{L}|_Y)+\log(4p)\log\left(4p\dim_{K_Y}\aSpan{K_Y}{\aHzstq{X|Y}(m\overline{L})}\right)\right)\\
 &\qquad\qquad\qquad\qquad\qquad\qquad\qquad \times[K_Y:\QQ]\dim_{K_Y}\aSpan{K_Y}{\aHzstq{X|Y}(m\overline{L})}.
\end{align*}
Then, by Theorem~\ref{thm:Yuan} and Lemma~\ref{lem:rkcomparison}, 
\begin{equation}
 \left|\log\sharp\CLq{X|Y}(m\overline{L})-\sharp\bm{w}_{F_{\sbullet}}(\aHzstq{X|Y}(m\overline{L})\setminus\{0\})\log(p)\right|\leq\frac{D_m}{\log(p)}.
\end{equation}
Given any $\varepsilon>0$,
\begin{equation}
 \dim_{K_Y}\aSpan{K_Y}{\aHzstq{X|Y}(m\overline{L})}\leq\frac{e\left(\aRq{X|Y}(\overline{L})_{\sbullet}\right)+\varepsilon}{\akappaq{X|Y}(\overline{L})!}m^{\akappaq{X|Y}(\overline{L})}
\end{equation}
holds for every sufficiently large $m\in\aN_{X|Y}(\overline{L})$ by the arguments above.
Therefore, given any $\varepsilon>0$, $D_m\leq (D+\varepsilon)m^{\akappaq{X|Y}(\overline{L})+1}$ holds for every sufficiently large $m\in\aN_{X|Y}(\overline{L})$.
\end{proof}

\begin{theorem}\label{thm:convergence}
Let $X$ be a projective variety over a number field and let $Y$ be a closed subvariety of $X$.
\begin{enumerate}
\item[\textup{(1)}] For every adelically metrized line bundle $\overline{L}$ on $X$ with $\akappaq{X|Y}(\overline{L})\geq 0$, the sequence
\[
 \left(\frac{\log\sharp\CLq{X|Y}(m\overline{L})}{m^{\akappaq{X|Y}(\overline{L})+1}}\right)_{m\in\aN_{X|Y}(\overline{L})}
\]
converges to a positive real number.
\item[\textup{(2)}] Let $\overline{L}$ be an adelically metrized line bundle on $X$ such that either $\overline{L}$ is $Y$-big or $\akappaq{X|Y}(\overline{L})<\dim Y$.
Then the sequence
\[
 \left(\frac{\log\sharp\CLq{X|Y}(m\overline{L})}{m^{\dim Y+1}/(\dim Y+1)!}\right)_{m\geq 1}
\]
converges to $\avolq{X|Y}(\overline{L})$.
\end{enumerate}
\end{theorem}

\begin{definition}\label{defn:amult}
Suppose that $\akappaq{X|Y}(\overline{L})\geq 0$.
We define the \emph{arithmetic multiplicity of $\overline{L}$ along $Y$} as
\[
 \aeq{X|Y}(\overline{L}):=\lim_{\substack{m\in\aN_{X|Y}(\overline{L}), \\ m\to+\infty}}\frac{\log\sharp\CLq{X|Y}(m\overline{L})}{m^{\akappaq{X|Y}(\overline{L})+1}/(\akappaq{X|Y}(\overline{L})+1)!}\quad\in\RR_{>0}.
\]
By Theorem~\ref{thm:convergence}(1), there exists a positive constant $c>0$ such that
\[
 c^{-1}m^{\akappaq{X|Y}(\overline{L})+1}\leq\log\sharp\CLq{X|Y}(m\overline{L})\leq cm^{\akappaq{X|Y}(\overline{L})+1}
\]
holds for every $m\in\aN_{X|Y}(\overline{L})$.
\end{definition}

\begin{proof}[Proof of Theorem~\ref{thm:convergence}]
(1): Let $D:=D_{X|Y}(\overline{L})$ be the constant as in Proposition~\ref{prop:Yuan}.
We can find an $s_0\in\aHzstq{X|Y}(m_0\overline{L})\setminus\{0\}$ for an $m_0\geq 1$.
Set $\Psi:=\left\{y\in\mathscr{Y}_U'\,:\,s_0(y)=0\right\}$.
Given any $\varepsilon>0$, we can find a prime number $p$ such that there exists a $\Psi$-good flag $F_{\sbullet}$ on $\mathscr{Y}_U'$ over $p$ and
\begin{equation}
 \frac{D}{\log(p)}\leq\varepsilon
\end{equation}
(Lemma~\ref{lem:existgoodflags}).
By Propositions~\ref{prop:Yuan} and \ref{prop:KKOk}, we have
\[
 0\leq\limsup_{\substack{m\in\aN_{X|Y}(\overline{L}), \\ m\to+\infty}}\frac{\log\sharp\CLq{X|Y}(m\overline{L})}{m^{\akappaq{X|Y}(\overline{L})+1}}-\liminf_{\substack{m\in\aN_{X|Y}(\overline{L}), \\ m\to+\infty}}\frac{\log\sharp\CLq{X|Y}(m\overline{L})}{m^{\akappaq{X|Y}(\overline{L})+1}}\leq 2\varepsilon.
\]
Hence we conclude.

(2): If $\aHzstq{X|Y}(m\overline{L})=\{0\}$ for every $m$, then the assertion is obvious.
If $\overline{L}$ is $Y$-big, then the assertion is nothing but the assertion (1) (see Lemma~\ref{lem:propertiesYbig}(1)).
If $0\leq\akappaq{X|Y}(\overline{L})<\dim Y$, then by the assertion (1),
\[
 0\leq\frac{\log\sharp\CLq{X|Y}(m\overline{L})}{m^{\dim Y+1}}=\frac{\log\sharp\CLq{X|Y}(m\overline{L})}{m^{\akappaq{X|Y}(\overline{L})+1}}\times\left(\frac{1}{m}\right)^{\dim Y-\akappaq{X|Y}(\overline{L})}\to 0.
\]
\end{proof}

\begin{theorem}\label{thm:propertiesavol}
Let $X$ be a projective variety over a number field, let $Y$ be a closed subvariety of $X$, and let $\overline{L},\overline{M}$ be adelically metrized line bundles on $X$.
\begin{enumerate}
\item[\textup{(1)}] If $\akappaq{X|Y}(\overline{L})\geq 0$, then for every integer $a\geq 1$
\[
 \aeq{X|Y}(a\overline{L})=a^{\akappaq{X|Y}(\overline{L})+1}\cdot\aeq{X|Y}(\overline{L}).
\]
\item[\textup{(2)}] For every integer $a\geq 1$,
\[
 \avolq{X|Y}(a\overline{L})=a^{\dim Y+1}\cdot\avolq{X|Y}(\overline{L}).
\]
\item[\textup{(3)}] If $\akappaq{X|Y}(\overline{L})=\dim Y$ and $\akappaq{X|Y}(\overline{M})\geq 0$, then $\akappaq{X|Y}(\overline{L}+\overline{M})=\dim Y$ and
\begin{align*}
 &\left(|\aS_{X|Y}(\overline{L}+\overline{M})|\cdot\avolq{X|Y}(\overline{L}+\overline{M})\right)^{1/(\dim Y+1)} \\
 &\qquad\qquad \geq \left(|\aS_{X|Y}(\overline{L})|\cdot\avolq{X|Y}(\overline{L})\right)^{1/(\dim Y+1)}+\left(|\aS_{X|Y}(\overline{M})|\cdot\avolq{X|Y}(\overline{M})\right)^{1/(\dim Y+1)},
\end{align*}
where $|\aS_{X|Y}(\cdot)|$ is defined in Definition~\ref{defn:languageofS}.
\end{enumerate}
\end{theorem}

\begin{proof}
Note that $\akappaq{X|Y}(a\overline{L})=\akappaq{X|Y}(\overline{L})$ by Lemma~\ref{lem:propertiesakappa}(1).
Since
\[
 \lim_{\substack{m\in\aN_{X|Y}(a\overline{L}), \\ m\to+\infty}}\frac{\log\sharp\CLq{X|Y}(ma\overline{L})}{m^{\akappaq{X|Y}(a\overline{L})+1}}=a^{\akappaq{X|Y}(\overline{L})+1}\cdot\lim_{\substack{m\in\aN_{X|Y}(a\overline{L}), \\ m\to+\infty}}\frac{\log\sharp\CLq{X|Y}(ma\overline{L})}{(ma)^{\akappaq{X|Y}(\overline{L})+1}},
\]
we have the assertion (1).

(2): If $\akappaq{X|Y}(\overline{L})<\dim Y$, then the both sides are zero.
If $\akappaq{X|Y}(\overline{L})=\dim Y$, then the assertion is nothing but the assertion (1).

(3): The first assertion follows from Lemma~\ref{lem:effectiveincrease}(1).
If $\akappaq{X|Y}(\overline{M})<\dim Y$, then $\avolq{X|Y}(\overline{M})=0$ and the assertion is clear by Lemma~\ref{lem:effectiveincrease}(1).
Suppose that $\akappaq{X|Y}(\overline{M})=\dim Y$.
Let $D_{X|Y}(\overline{L}+\overline{M})$, $D_{X|Y}(\overline{L})$, and $D_{X|Y}(\overline{M})$ be the constants as in Proposition~\ref{prop:Yuan}, and set $D:=\max\left\{D_{X|Y}(\overline{L}+\overline{M}),D_{X|Y}(\overline{L}),D_{X|Y}(\overline{M})\right\}$.
We can find an $m_0\geq 1$, an $s_0\in\aHzstq{X|Y}(m_0\overline{L})\setminus\{0\}$, and a $t_0\in\aHzstq{X|Y}(m_0\overline{M})\setminus\{0\}$.
Set
\[
 \Psi:=\left\{y\in\mathscr{Y}_U'\,:\,s_0(y)=0\right\}\cup\left\{y\in\mathscr{Y}_U'\,:\,t_0(y)=0\right\}.
\]
Given any $\varepsilon>0$, we can find a prime number $p$ such that there exists a $\Psi$-good flag $F_{\sbullet}$ on $\mathscr{Y}_U'$ over $p$ and
\begin{equation}
 \frac{D}{\log(p)}\leq\varepsilon
\end{equation}
(Lemma~\ref{lem:existgoodflags}).
By using Propositions~\ref{prop:KKOk} and \ref{prop:Yuan} and by applying the Brunn--Minkowski inequality to the convex bodies
\[
 \aDelta_{X|Y}(\overline{L})+\aDelta_{X|Y}(\overline{M})\subset\aDelta_{X|Y}(\overline{L}+\overline{M})
\]
in $\RR^{\dim Y+1}$, we have
\begin{align*}
 &\left(|\aS_{X|Y}(\overline{L}+\overline{M})|\cdot\avolq{X|Y}(\overline{L}+\overline{M})\right)^{1/(\dim Y+1)} \\
\geq &\left(|\aS_{X|Y}(\overline{L})|\cdot\avolq{X|Y}(\overline{L})\right)^{1/(\dim Y+1)}+\left(|\aS_{X|Y}(\overline{M})|\cdot\avolq{X|Y}(\overline{M})\right)^{1/(\dim Y+1)}-3\varepsilon.
\end{align*}
Hence we conclude.
\end{proof}

\begin{corollary}\label{cor:propertiesarestvol}
Let $X$ be a normal projective variety over a number field and let $\overline{L}$ be an adelically metrized line bundle on $X$.
\begin{enumerate}
\item[\textup{(1)}] We have
\[
 \aBsp(\overline{L})=\bigcup_{\substack{Z\subset X, \\ \avolq{X|Z}(\overline{L})=0}}Z=\bigcup_{\substack{Z\subset X, \\ \akappaq{X|Z}(\overline{L})<\dim Z}}Z.
\]
\item[\textup{(2)}] For any prime divisor $Y$ on $X$, $\avolq{X|Y}(\overline{L})>0$ if and only if $\akappaq{X|Y}(\overline{L})=\dim X-1$.
\end{enumerate}
\end{corollary}

\begin{proof}
(1): By Proposition~\ref{prop:aaugbs}(1), we have $\aBsp(\overline{L})=\Bsp(L)\cup\aSBs(\overline{L})$.
If $Z\subset\aSBs(\overline{L})$, then clearly $\avolq{X|Z}(\overline{L})=0$.
If $Z$ is a component of $\Bsp(L)$ and is not contained in $\aSBs(\overline{L})$, then by \cite[Theorem~B]{Bou_Cac_Lop13} we have
\[
 \kappa_{X|Z}(\overline{L})<\dim Z.
\]
Thus $\akappaq{X|Z}(\overline{L})<\dim Z$ and $\avolq{X|Z}(\overline{L})=0$.
On the other hand, if $Z\not\subset\aBsp(\overline{L})$, then, by Lemma~\ref{lem:propertiesYbig}(2),(3), we have $\akappaq{X|Z}(\overline{L})=\dim Z$ and $\avolq{X|Z}(\overline{L})>0$.

The assertion (2) results from the assertion (1).
\end{proof}

\begin{corollary}\label{cor:continuity}
\begin{enumerate}
\item[\textup{(1)}] For any $Y$-big adelically metrized $\QQ$-line bundle $\overline{L}$ and for any adelically metrized $\QQ$-line bundles $\overline{A}_1,\dots,\overline{A}_r$,
\[
 \lim_{\varepsilon_1\to 0,\dots,\varepsilon_r\to 0}\avolq{X|Y}(\overline{L}+\varepsilon_1\overline{A}_1+\dots+\varepsilon_r\overline{A}_r)=\avolq{X|Y}(\overline{L}).
\]
\item[\textup{(2)}] If $\overline{L}_1,\overline{L}_2$ are $Y$-big adelically metrized $\QQ$-line bundles and $\overline{L}_2-\overline{L}_1$ is $Y$-pseudoeffective, then
\[
 \avolq{X|Y}(\overline{L}_1)\leq\avolq{X|Y}(\overline{L}_2).
\]
\end{enumerate}
\end{corollary}

\begin{proof}
Note that the cone of all the $Y$-big adelically metrized $\QQ$-line bundles is open by Lemma~\ref{prop:aaugbs}(4).
The assertion (1) results from Theorem~\ref{thm:propertiesavol}(3) and the standard arguments (see \cite[Theorem~5.2]{Ein_Laz_Mus_Nak_Pop06}, \cite[Proposition~1.3.1]{MoriwakiEst}).

(2): Let $\overline{A}$ be a $Y$-big adelically metrized line bundle.
For any $\varepsilon\in\QQ_{>0}$, some sufficiently large multiple of $\overline{L}_2-\overline{L}_1+\varepsilon\overline{A}$ becomes a $Y$-effective adelically metrized line bundle, so
\[
 \avolq{X|Y}(\overline{L}_1)\leq\avolq{X|Y}(\overline{L}_2+\varepsilon\overline{A})
\]
by Lemma~\ref{lem:effectiveincrease}(1).
By taking $\varepsilon\to 0+$, we have the assertion (2).
\end{proof}

\begin{theorem}\label{thm:amultcontinuity}
Let $X$ be a projective variety over a number field, let $Y$ be a closed subvariety of $X$, and let $\overline{L},\overline{M}$ be two adelically metrized line bundles on $X$.
\begin{enumerate}
\item[\textup{(1)}] If $\overline{L}\leq_Y\overline{M}$ and $\akappaq{X|Y}(\overline{L})=\akappaq{X|Y}(\overline{M})$, then $\aeq{X|Y}(\overline{L})\leq\aeq{X|Y}(\overline{M})$.
\item[\textup{(2)}] If $Y\not\subset\aSBs(\overline{L})$, then
\[
 \lim_{\lambda\to 0}\aeq{X|Y}(\overline{L}+\overline{\mathcal{O}}_X(\lambda[\infty]))=\aeq{X|Y}(\overline{L}).
\]
\end{enumerate}
\end{theorem}

\begin{proof}
(1): We may assume that $\akappaq{X|Y}(\overline{L})=\akappaq{X|Y}(\overline{M})\geq 0$.
Since $\CLq{X|Y}(m\overline{L})\subset\CLq{X|Y}(m\overline{M})$ for every $m\geq 1$ and $\akappaq{X|Y}(\overline{L})=\akappaq{X|Y}(\overline{M})$,
\[
 \lim_{\substack{m\in\aN_{X|Y}(\overline{L}), \\ m\to+\infty}}\frac{\log\sharp\CLq{X|Y}(m\overline{L})}{m^{\akappaq{X|Y}(\overline{L})+1}/(\akappaq{X|Y}(\overline{L})+1)!}\leq\lim_{\substack{m\in\aN_{X|Y}(\overline{L}), \\ m\to+\infty}}\frac{\log\sharp\CLq{X|Y}(m\overline{M})}{m^{\akappaq{X|Y}(\overline{M})+1}/(\akappaq{X|Y}(\overline{M})+1)!}.
\]

(2): We start the proof with the following claims.

\begin{claim}\label{clm:amultcontinuity1}
There exists a rational number $\lambda_0>0$ such that, for every $\lambda\in\RR$ with $\lambda>-\lambda_0$,
\[
 \akappaq{X|Y}(\overline{L}+\overline{\mathcal{O}}_X(\lambda[\infty]))=\kappa_{X|Y}(L).
\]
\end{claim}

\begin{proof}
Take an $a\gg 1$ and $s_1,\dots,s_N\in\aHzst(a\overline{L})$ such that $\{x\in X\,:\,s_1(x)=\dots=s_N(x)=0\}=\aSBs(\overline{L})$.
Let $\lambda_0$ be a rational number with $0<\lambda_0\leq\min_i\{-\log\|s_i\|_{\sup}\}$.
For every $\lambda$ with $\lambda>-\lambda_0$, all of $s_i$'s belong to $\aHzst(\overline{L}+\overline{\mathcal{O}}_X(\lambda[\infty]))$ and $\aSBs(\overline{L}+\overline{\mathcal{O}}_X(\lambda[\infty]))\subset\aSBs(\overline{L})$.
Therefore the claim follows from Lemma~\ref{lem:kappaqequal}.
\end{proof}

\begin{claim}\label{clm:amultcontinuity2}
If $Y\not\subset\aSBs(\overline{L})$, then there exists an $m_0\geq 1$ such that, for every $\lambda\in\QQ$ with $\lambda\leq\lambda_0$, $m_0\overline{L}\geq_Y\lambda\overline{\mathcal{O}}_X([\infty])$.
\end{claim}

\begin{proof}
We take an $a\geq 1$ and an $s\in\aHzst(a\overline{L})$ such that $s|_Y\neq 0$.
We can choose an $m_0\geq 1$ such that $\exp(\lambda_0)\|s\|_{\sup}^{m_0}<1$.
Then $m_0\overline{L}\geq_Y\overline{\mathcal{O}}_X(\lambda[\infty])=\lambda\overline{\mathcal{O}}_X([\infty])$ for every $\lambda\in\QQ$ with $\lambda\leq\lambda_0$.
\end{proof}

By Theorem~\ref{thm:propertiesavol}(1), Claims~\ref{clm:amultcontinuity1} and \ref{clm:amultcontinuity2}, and the assertion (1), we have
\begin{align*}
 &\left(1-\frac{m_0}{\lambda_0}|\lambda|\right)^{\akappaq{X|Y}(\overline{L})+1}\cdot\aeq{X|Y}(\overline{L})=\aeq{X|Y}\left(\left(1-\frac{m_0}{\lambda_0}|\lambda|\right)\overline{L}\right) \\
 &\qquad\qquad\qquad \leq\aeq{X|Y}(\overline{L}+\lambda\overline{\mathcal{O}}_X([\infty])) \\
 &\qquad\qquad\qquad \leq\aeq{X|Y}\left(\left(1+\frac{m_0}{\lambda_0}|\lambda|\right)\overline{L}\right)=\left(1+\frac{m_0}{\lambda_0}|\lambda|\right)^{\akappaq{X|Y}(\overline{L})+1}\cdot\aeq{X|Y}(\overline{L})
\end{align*}
for every $\lambda\in\QQ$ with $|\lambda|<\lambda_0/m_0$.
Hence we conclude.
\end{proof}

\section{Generalized Fujita approximation}\label{sec:genFujita}

In this section, we obtain a formula expressing an arithmetic restricted volume as a supremum of heights of projective varieties (Theorem~\ref{thm:genFujita}).
We can regard Theorem~\ref{thm:genFujita} as an arithmetic analogue of the generalized Fujita approximation proved in \cite[Theorem~2.13]{Ein_Laz_Mus_Nak_Pop06} (in our case, $\dim Y$ can be zero).

\begin{proposition}\label{prop:nefcase}
Let $X$ be a smooth projective variety over $K$ and let $\overline{L}$ be an adelically metrized line bundle on $X$.
If $\overline{L}$ is nef and $Y$-big, then
\[
 \avolq{X|Y}(\overline{L})=\adeg\left((\overline{L}|_Y)^{\cdot \dim Y+1}\right).
\]
\end{proposition}

\begin{proof}
If $\overline{L}$ is associated to an ample $C^{\infty}$-Hermitian line bundle on an $O_K$-model of $X$, then the assertion follows from \cite[Corollary~7.2(1)]{MoriwakiEst}.
We assume that $\overline{L}$ is associated to a nef continuous Hermitian line bundle on some $O_K$-model $\mathscr{X}$ of $X$.
Fix an adelically metrized line bundle $\overline{H}$ associated to a $Y$-effective ample $C^{\infty}$-Hermitian line bundle $\overline{\mathscr{H}}$ on $\mathscr{X}$.
For every rational number $\varepsilon>0$, there exists a non-negative continuous function $\lambda_{\varepsilon}:X_{\infty}^{\rm an}\to\RR_{\geq 0}$ such that $\|\lambda_{\varepsilon}\|_{\sup}<\varepsilon$ and $\overline{L}+\varepsilon\overline{H}+\overline{\mathcal{O}}_X(\lambda_{\varepsilon}[\infty])$ is associated to an ample $C^{\infty}$-Hermitian line bundle on $\mathscr{X}$ (\cite[Theorem~1]{Blocki_Kolo07}).
We then have
\[
 \avolq{X|Y}\left(\overline{L}+\varepsilon\overline{H}+\overline{\mathcal{O}}_X(\lambda_{\varepsilon}[\infty])\right)=\adeg\left((\overline{L}+\varepsilon\overline{H}+\overline{\mathcal{O}}_X(\lambda_{\varepsilon}[\infty]))|_Y^{\cdot \dim Y+1}\right)
\]
for every rational number $\varepsilon>0$.
By Corollary~\ref{cor:continuity},
\begin{align*}
 \avolq{X|Y}(\overline{L}) &\leq\avolq{X|Y}\left(\overline{L}+\varepsilon\overline{H}+\overline{\mathcal{O}}_X(\lambda_{\varepsilon}[\infty])\right) \\
 &\leq\avolq{X|Y}\left(\overline{L}+\varepsilon\overline{H}+\varepsilon\overline{\mathcal{O}}_X([\infty])\right)\xrightarrow{\varepsilon\downarrow 0}\avolq{X|Y}(\overline{L}).
\end{align*}
On the other hand, by Proposition~\ref{prop:aintnum}(1),
\begin{align*}
 &\adeg\left((\overline{L}|_Y)^{\cdot \dim Y+1}\right)\\
 &\qquad\qquad\leq\adeg\left((\overline{L}+\varepsilon\overline{H}+\overline{\mathcal{O}}_X(\lambda_{\varepsilon}[\infty]))|_Y^{\cdot \dim Y+1}\right) \\
 &\qquad\qquad\leq \adeg\left((\overline{L}+\varepsilon\overline{H}+\varepsilon\overline{\mathcal{O}}_X([\infty]))|_Y^{\cdot \dim Y+1}\right)\xrightarrow{\varepsilon\downarrow 0}\adeg\left((\overline{L}|_Y)^{\cdot \dim Y+1}\right).
\end{align*}
Hence we conclude in this case.

In general, thanks to Proposition~\ref{prop:adelicapproximation}(2), there exists a finite subset $S\subset M_K^{\rm f}$ such that, for any rational number $\varepsilon>0$, one can find $O_K$-models $(\mathscr{X}_{\varepsilon},\mathscr{L}_{\varepsilon,1})$ and $(\mathscr{X}_{\varepsilon},\mathscr{L}_{\varepsilon,2})$ of $(X,L)$ such that $\mathscr{L}_{\varepsilon,i}$ are relatively nef and
\[
 \overline{L}-\varepsilon\sum_{P\in S}\overline{\mathcal{O}}_X([P])\leq\left(\mathscr{L}_{\varepsilon,1},|\cdot|_{\infty}^{\overline{L}}\right)^{\rm ad}\leq\overline{L}\leq\left(\mathscr{L}_{\varepsilon,2},|\cdot|_{\infty}^{\overline{L}}\right)^{\rm ad}\leq\overline{L}+\varepsilon\sum_{P\in S}\overline{\mathcal{O}}_X([P]).
\]

\begin{claim}
We have
\begin{equation}\label{eqn:nefcase1}
 \overline{L}-\varepsilon\sum_{P\in S}\overline{\mathcal{O}}_X([P])\leq_Y\overline{\mathscr{L}}_{\varepsilon,1}^{\rm ad}\leq_Y\overline{L}\leq_Y\overline{\mathscr{L}}_{\varepsilon,2}^{\rm ad}\leq_Y\overline{L}+\varepsilon\sum_{P\in S}\overline{\mathcal{O}}_X([P]).
\end{equation}
\end{claim}

\begin{proof}
We can write the four differences in the form of
\[
 \left(\mathcal{O}_X,\sum_{P\in S}\lambda_P[P]\right),
\]
where $\lambda_P:X_P^{\rm an}\to\RR_{\geq 0}$ is a non-negative continuous function on $X_P^{\rm an}$.
Then $1\in\Hz(\mathcal{O}_X)$ is a small section satisfying $1|_Y\neq 0$.
\end{proof}

Set $\lambda:=-(\varepsilon/[K:\QQ])\sum_{P\in S}\log|\varpi_P|_P$.
Then $\overline{\mathscr{L}}_{\varepsilon,1}+\overline{\mathcal{O}}_X(\lambda[\infty])$ is nef.
By (\ref{eqn:nefcase1}) and Proposition~\ref{prop:aintnum}(1),(2), we have
\begin{align*}
 &\avolq{X|Y}(\overline{\mathscr{L}}_{\varepsilon,1}^{\rm ad})+[K:\QQ]\lambda\deg\left((L|_Y)^{\cdot\dim Y}\right) \\
 &\qquad\qquad\qquad=\adeg\left((\overline{\mathscr{L}}_{\varepsilon,1}^{\rm ad}|_Y+\overline{\mathcal{O}}_Y(\lambda[\infty]))^{\cdot \dim Y+1}\right) \\
 &\qquad\qquad\qquad\leq\adeg\left((\overline{L}|_Y+\overline{\mathcal{O}}_Y(\lambda[\infty]))^{\cdot \dim Y+1}\right) \\
 &\qquad\qquad\qquad=\adeg\left((\overline{L}|_Y)^{\cdot\dim Y+1}\right)+[K:\QQ]\lambda\deg\left((L|_Y)^{\cdot\dim Y}\right) \\
 &\qquad\qquad\qquad\leq\adeg\left((\overline{\mathscr{L}}_{\varepsilon,2}^{\rm ad}|_Y)^{\cdot\dim Y+1}\right)+[K:\QQ]\lambda\deg\left((L|_Y)^{\cdot\dim Y}\right) \\
 &\qquad\qquad\qquad=\avolq{X|Y}(\overline{\mathscr{L}}_{\varepsilon,2}^{\rm ad})+[K:\QQ]\lambda\deg\left((L|_Y)^{\cdot\dim Y}\right),
\end{align*}
so
\begin{align*}
 \avolq{X|Y}\left(\overline{L}-\varepsilon\sum_{P\in S}\overline{\mathcal{O}}_X([P])\right) &\leq\avolq{X|Y}(\overline{\mathscr{L}}_{\varepsilon,1}^{\rm ad})\\
 &\leq\adeg\left((\overline{L}|_Y)^{\cdot\dim Y+1}\right)\\
 &\leq\avolq{X|Y}(\overline{\mathscr{L}}_{\varepsilon,2}^{\rm ad})\leq\avolq{X|Y}\left(\overline{L}-\varepsilon\sum_{P\in S}\overline{\mathcal{O}}_X([P])\right)
\end{align*}
for every rational number $\varepsilon>0$.
Therefore, we conclude by Corollary~\ref{cor:continuity}.
\end{proof}

Let $X$ be a normal projective variety over $K$, let $\overline{L}$ be an adelically metrized $\QQ$-line bundle on $X$, and let $Z$ be a closed subvariety of $X$.
A \emph{$Z$-compatible approximation for $\overline{L}$} is a pair $(\mu:X'\to X,\overline{M})$ of a projective birational $K$-morphism $\mu:X'\to X$ and a nef adelically metrized $\QQ$-line bundle $\overline{M}$ on $X'$ having the following properties.
\begin{enumerate}
\item[\textup{(a)}] $X'$ is smooth and $\mu$ is isomorphic around the generic point of $Z$.
\item[\textup{(b)}] Let $\mu_*^{-1}(Z)$ be the strict transform of $Z$ via $\mu$.
Then $\overline{M}$ is $\mu_*^{-1}(Z)$-big and $\mu^*\overline{L}-\overline{M}$ is a $\mu_*^{-1}(Z)$-pseudoeffective adelically metrized $\QQ$-line bundle.
\end{enumerate}
We denote by $\aTheta_Z(\overline{L})$ the set of all the $Z$-compatible approximations for $\overline{L}$.

\begin{lemma}
$\aTheta_Z(\overline{L})\neq\emptyset$ if and only if $\overline{L}$ is $Z$-big.
\end{lemma}

\begin{proof}
If $(\mu:X'\to X,\overline{M})\in\aTheta_Z(\overline{L})$, then $\mu^*\overline{L}$ is $\mu_*^{-1}(Z)$-big.
By Lemmas~\ref{lem:augbs_main}(2) and \ref{lem:propertiesYbig}, we have $\mu_*^{-1}(Z)\not\subset\mu^{-1}\aBsp(\overline{L})\cup\Exc(\mu)$.
Thus $Z\not\subset\aBsp(\overline{L})$.
The ``if'' part is obvious.
\end{proof}

\begin{theorem}\label{thm:genFujita}
Let $X$ be a normal projective variety over $K$, let $Z$ be a closed subvariety of $X$, and let $\overline{L}$ be an adelically metrized $\QQ$-line bundle on $X$.
If $\overline{L}$ is $Z$-big, then, for every closed subvariety $Y$ containing $Z$, we have
\[
 \avolq{X|Y}(\overline{L})=\sup_{(\mu,\overline{M})\in\aTheta_Z(\overline{L})}\adeg\left((\overline{M}|_{\mu_*^{-1}(Y)})^{\cdot(\dim Y+1)}\right).
\]
\end{theorem}

\begin{proof}
The inequality $\geq$ is obvious (Lemma~\ref{lem:effectiveincrease}(1),(2) and Proposition~\ref{prop:nefcase}), so it suffices to show that
\[
 \avolq{X|Y}(\overline{L})\leq\sup_{(\mu,\overline{M})\in\aTheta_Z(\overline{L})}\avolq{X'|\mu_*^{-1}(Y)}(\overline{M}).
\]
We can assume that $\overline{L}$ is an adelically metrized line bundle on $X$.
Let $\mu_m:X_m\to X$, $\overline{M}_m$, and $\overline{F}_m$ be as in Proposition~\ref{prop:adelicmetric}.
Without loss of generality, we can assume that $X_m$ are all smooth.
Let $Z_m:=\mu_{m*}^{-1}(Z)$ (respectively, $Y_m:=\mu_{m*}^{-1}(Y)$) be the strict transform of $Z$ (respectively, $Y$) via $\mu_m$.
If $\aBs(m\overline{L})=\aSBs(\overline{L})$, then $Z_m$ is contained neither in $\aBsp(\overline{M}_m)\subset\mu_m^{-1}\aBsp(\overline{L})$ nor in $\aSBs(\overline{F}_m)\subset\mu_m^{-1}\aSBs(\overline{L})$.
So $\overline{M}_m$ is $Z_m$-big and $\overline{F}_m$ is $Z_m$-effective.
The theorem follows from Proposition~\ref{prop:genFujita} below.
\end{proof}

\begin{proposition}\label{prop:genFujita}
We have
\[
 \avolq{X|Y}(\overline{L})=\lim_{\substack{\aBs(m\overline{L})=\aSBs(\overline{L}), \\ m\to+\infty}}\frac{\avolq{X_m|Y_m}(\overline{M}_m)}{m^{\dim Y+1}}.
\]
\end{proposition}

\begin{proof}
First, we show the convergence of the sequence.

\begin{claim}\label{clm:genFujitaconv}
The sequence
\[
 \left(\frac{\avolq{X_m|Y_m}(\overline{M}_m)}{m^{\dim Y+1}}\right)_{\aBs(m\overline{L})=\aSBs(\overline{L})}
\]
converges.
\end{claim}

\begin{proof}
Take positive integers $m,n$ such that $\aBs(m\overline{L})=\aBs(n\overline{L})=\aSBs(\overline{L})$.
Then $\aBs((m+n)\overline{L})=\aSBs(\overline{L})$.
Let $\widetilde{\mu}:\widetilde{X}\to X$ be a birational projective $K$-morphism such that $\widetilde{\mu}$ dominates all of $\mu_m$, $\mu_n$, and $\mu_{m+n}$ and $\widetilde{\mu}$ is isomorphic over $X\setminus\aSBs(\overline{L})$.
In particular, we have a diagram
\[
\xymatrix{ & X_m \ar[dr]^-{\mu_m} & \\ X' \ar[ur]^-{\nu_1} \ar[r]^-{\nu_3} \ar[dr]_-{\nu_2} & X_{m+n} \ar[r]^-{\mu_{m+n}} & X. \\ & X_n \ar[ur]_-{\mu_n} &
}
\]
Let $\widetilde{Z}$ be the strict transform of $Z$ via $\widetilde{\mu}$.
Then what we are going to show is
\begin{equation}\label{eqn:genFujitaineq}
 \nu_1^*\overline{F}_m+\nu_2^*\overline{F}_n\geq_{\widetilde{Z}}\nu_3^*\overline{F}_{m+n}.
\end{equation}
Since $\nu_1^*1_{F_m}\otimes\nu_2^*1_{F_n}$ vanishes along $\nu_3^{-1}\Supp(1_{F_{m+n}})$, there exists a section $\iota\in\Hz(\nu_1^*F_m+\nu_2^*F_n-\nu_3^*F_{m+n})$ such that
\[
 \nu_1^*1_{F_m}\otimes\nu_2^*1_{F_n}=\nu_3^*1_{F_{m+n}}\otimes\iota\quad\text{and}\quad\Supp(\iota)\subset{\widetilde{\mu}}^{-1}\aSBs(\overline{L}).
\]
For each $v\in M_K^{\rm f}$ and $x\in{X'}_v^{\rm an}$,
\begin{align*}
 &|\iota|_v^{\nu_1^*\overline{F}_m+\nu_2^*\overline{F}_n-\nu_3^*\overline{F}_{m+n}}(x) \\
 &\qquad =\frac{\max_{s\in\aHzst(m\overline{L})}\left\{|s|_v^{m\overline{L}}({\widetilde{\mu}}_v^{\rm an}(x))\right\}\cdot\max_{t\in\aHzst(n\overline{L})}\left\{|t|_v^{n\overline{L}}({\widetilde{\mu}}_v^{\rm an}(x))\right\}}{\max_{u\in\aHzst((m+n)\overline{L})}\left\{|u|_v^{(m+n)\overline{L}}({\widetilde{\mu}}_v^{\rm an}(x))\right\}}\leq 1.
\end{align*}
In the same way, we can see $\|\iota\|_{\infty}^{\nu_1^*\overline{F}_m+\nu_2^*\overline{F}_n-\nu_3^*\overline{F}_{m+n}}\leq 1$.
Therefore, we obtain the formula (\ref{eqn:genFujitaineq}).

By Theorem~\ref{thm:propertiesavol}(3), Lemma~\ref{lem:effectiveincrease}, and the formula (\ref{eqn:genFujitaineq}), we have
\begin{align*}
 &\avolq{X_m|Y_m}(\overline{M}_m)^{1/(\dim Y+1)}+\avolq{X_n|Y_n}(\overline{M}_n)^{1/(\dim Y+1)} \\
 &\qquad\qquad\qquad\qquad\qquad\qquad \leq\avolq{X_{m+n}|Y_{m+n}}(\overline{M}_{m+n})^{1/(\dim Y+1)}.
\end{align*}
By the standard arguments, the sequence converges to its supremum.
\end{proof}

We fix an integer $a\geq 1$ such that $\aBs(a\overline{L})=\aSBs(\overline{L})$ and there exists an $s_0\in\aHzstq{X|Y}(a\overline{L})\setminus\{0\}$.
By Claim~\ref{clm:genFujitaconv}, it suffices to show the equation
\begin{equation}
 \avolq{X|Y}(\overline{L})=\lim_{m\to+\infty}\frac{\avolq{X_{ma}|Y_{ma}}(\overline{M}_{ma})}{(ma)^{\dim Y+1}}.
\end{equation}
We fix models and flags as follows.
Let $(\mathscr{X},\mathscr{L})$ be an $O_K$-model of $(X,L)$ such that $\mathscr{X}$ is normal, $\mathscr{L}$ is a line bundle on $\mathscr{X}$, and $\overline{L}\leq\overline{\mathscr{L}}^{\rm ad}$.
Let $U_0$ be a non-empty open subset of $\Spec(O_K)$ such that $(\mathscr{X}_{U_0},\mathscr{L}_{U_0})$ gives a $U_0$-model of definition for $\overline{L}^{\{\infty\}}$.
For each $m\geq 1$, let
\[
 \widehat{\mathfrak{b}}_{m,O_K}:=\Image\left(\aSpan{O_K}{\aHzst(m\overline{L})}\otimes_{O_K}(-m\mathscr{L})\to\mathcal{O}_{\mathscr{X}}\right)
\]
and let $\mu_{m,O_K}:\mathscr{X}_m\to\mathscr{X}$ be a normalized blow-up such that $\mathscr{X}_{m,K}$ is smooth and $\widehat{\mathfrak{b}}_{m,O_K}\mathcal{O}_{\mathscr{X}_m}$ is Cartier.
Let $\mathscr{F}_m:=\SHom_{\mathcal{O}_{\mathscr{X}_m}}(\widehat{\mathfrak{b}}_{m,O_K}\mathcal{O}_{\mathscr{X}_m},\mathcal{O}_{\mathscr{X}_m})$ and let $\mathscr{M}_m:=\mu_{m,O_K}^*(m\mathscr{L})-\mathscr{F}_m$.
Then $(\mathscr{X}_{m,U_0},\mathscr{F}_{m,U_0})$ and $(\mathscr{X}_{m,U_0},\mathscr{M}_{m,U_0})$ give $U_0$-models of definition for $\overline{F}_m^{\{\infty\}}$ and $\overline{M}_m^{\{\infty\}}$, respectively (Claim~\ref{clm:adelicmetric1}).

Let $\mathscr{Y}$ (respectively, $\mathscr{Y}_m$) be the Zariski closure of $Y$ (respectively, $Y_m$) in $\mathscr{X}$ (respectively, $\mathscr{X}_m$).
Let
\[
 \nu_{\mathscr{Y}}:\mathscr{Y}'\to\mathscr{Y}\quad\text{and}\quad\nu_{\mathscr{Y}_m}:\mathscr{Y}_m'\to\mathscr{Y}_m
\]
be the relative normalizations in $Y$ and in $Y_m$, respectively, and let $\mu_m':\mathscr{Y}_m'\to\mathscr{Y}'$ be the induced morphism.
Let $U$ be the inverse image of $U_0$ via $\Spec(O_{K_Y})\to\Spec(O_K)$.
Then $(\mathscr{Y}_U,\mathscr{L}|_{\mathscr{Y}_U})$ and $(\mathscr{Y}_{m,U},\mathscr{M}_m|_{\mathscr{Y}_{m,U}})$ give $U$-models of definition for $(\overline{L}|_Y)^{\{\infty\}}$ and $(\overline{M}_m|_{Y_m})^{\{\infty\}}$, respectively.

Let $D:=D_{X|Y}(a\overline{L})$ be the constant as in Proposition~\ref{prop:Yuan}, and let $\Psi:=\left\{y\in\mathscr{Y}_U'\,:\,s_0(y)=0\right\}$.
Given any $\varepsilon>0$, we can find a prime number $p$ such that there exists a $\Psi$-good flag $F_{\sbullet}$ on $\mathscr{Y}_U'$ over $p$ and
\begin{equation}\label{eqn:genFujita1}
 \frac{D}{\log(p)}\leq\frac{\varepsilon}{3}
\end{equation}
(Lemma~\ref{lem:existgoodflags}).
By Lemma~\ref{lem:imageflag}(2), ${\mu'}_{m*}^{-1}(F_{\sbullet})$ is a ${\mu'}_m^{-1}(\Psi)$-good flag on $\mathscr{Y}_m'$ for every $m\geq 1$ and ${\mu'}_{ma}^{-1}(\Psi)=\left\{y\in\mathscr{Y}_{ma}'\,:\,{\mu'}_{ma}^*(s_0^{\otimes m})(y)=0\right\}$ (see Proposition~\ref{prop:adelicmetric}).

\begin{claim}\label{clm:compare_sigma}
For every $m\geq 1$,
\[
 \delta(\overline{M}_m|_{Y_m})\leq m\delta(\overline{L}|_Y).
\]
\end{claim}

\begin{proof}[Proof of Claim~\ref{clm:compare_sigma}]
For any nef adelically metrized line bundle $\overline{A}$ on $Y$ with $\vol(A)>0$, we have
\[
 \delta(\overline{M}_m|_{Y_m})\leq\frac{\adeg\left(\overline{M}_m|_{Y_m}\cdot(\mu_m|_{Y_m})^*\overline{A}^{\cdot\dim Y}\right)}{\vol((\mu_m|_{Y_m})^*A)}\leq\frac{\adeg\left(m\overline{L}|_{Y}\cdot\overline{A}^{\cdot\dim Y}\right)}{\vol(A)}
\]
by Proposition~\ref{prop:aintnum}(3).
\end{proof}

\begin{claim}\label{clm:compare_e}
For every $m\geq 1$,
\[
 \frac{e\left(\aRq{X_{ma}|Y_{ma}}(\overline{M}_{ma})_{\sbullet}\right)}{\akappa_{X_{ma}|Y_{ma}}(\overline{M}_{ma})!}\leq m^{\dim Y}\cdot\frac{e\left(\aRq{X|Y}(a\overline{L})_{\sbullet}\right)}{\akappa_{X|Y}(a\overline{L})!}.
\]
\end{claim}

\begin{proof}[Proof of Claim~\ref{clm:compare_e}]
Note that we have $\akappa_{X|Y}(a\overline{L})=\akappa_{X_{ma}|Y_{ma}}(\overline{M}_{ma})=\dim Y$.
Since
\[
 \aRq{X_{ma}|Y_{ma}}(\overline{M}_{ma})_{\sbullet}\subset\aRq{X_{ma}|Y_{ma}}(\mu_{ma}^*(ma\overline{L}))_{\sbullet}=\aRq{X|Y}(ma\overline{L})_{\sbullet}
\]
as graded $K_Y$-algebras, we have the assertion.
\end{proof}

We set for $m\geq 1$
\begin{align*}
 &T(m):=\Bigl\{(km,\bm{w}_{{\mu'}_{ma*}^{-1}(F_{\sbullet})}(s\otimes (1_{F_{ma}}|_{Y_{ma}})^{\otimes k}))\,: \\
 &\qquad\qquad\qquad\qquad\qquad\qquad \forall k\geq 0,\,\forall s\in \CLq{X_{ma}|Y_{ma}}(k\overline{M}_{ma})\setminus\{0\}\Bigr\}
\end{align*}
(Lemma~\ref{lem:imageflag}(2)) and set
\[
 \Delta(m):=\overline{\left(\bigcup_{k\geq 1}\frac{1}{km}\bm{w}_{{\mu'}_{ma*}^{-1}(F_{\sbullet})}\left((\CLq{X_{ma}|Y_{ma}}(k\overline{M}_{ma})\setminus\{0\})\otimes (1_{F_{ma}}|_{Y_{ma}})^{\otimes k}\right)\right)}.
\]
Then $T(m)\subset\aS_{X|Y}(a\overline{L})$ and $\aS_{X|Y}(a\overline{L})_m\subset T(m)$ by Proposition~\ref{prop:adelicmetric} and Lemma~\ref{lem:imageflag}(2).
Thanks to \cite[Th{\'e}or{\`e}me~1.15]{BoucksomBourbaki}, there exists an $m_0\geq 1$ such that
\begin{equation}\label{eqn:genFujita2}
 \vol_{T(m)}\left(\Delta(m)\right)\log(p)\geq\vol_{\aS_{X|Y}(a\overline{L})}\left(\aDelta_{X|Y}(a\overline{L})\right)\log(p)-\frac{\varepsilon}{3}
\end{equation}
for every $m\geq m_0$.
By Proposition~\ref{prop:Yuan}, Claims~\ref{clm:compare_sigma} and \ref{clm:compare_e}, and the inequality (\ref{eqn:genFujita1}), we have
\begin{equation}\label{eqn:genFujita3}
 \vol_{\aS_{X|Y}(a\overline{L})}\left(\aDelta_{X|Y}(a\overline{L})\right)\log(p)\geq\avolq{X|Y}(a\overline{L})-\frac{\varepsilon}{3}
\end{equation}
and
\begin{equation}\label{eqn:genFujita4}
 \frac{\avolq{X_{ma}|Y_{ma}}(\overline{M}_{ma})}{m^{\dim Y+1}}\geq\vol_{T(m)}\left(\Delta(m)\right)\log(p)-\frac{\varepsilon}{3}.
\end{equation}
Therefore, by (\ref{eqn:genFujita2})--(\ref{eqn:genFujita4}),
\[
 \frac{\avolq{X_{ma}|Y_{ma}}(\overline{M}_{ma})}{(ma)^{\dim Y+1}}\geq\avolq{X|Y}(\overline{L})-\varepsilon
\]
for every $m\geq m_0$.
\end{proof}

\section*{Acknowledgement}

The author is very grateful to anonymous referees for carefully reading the manuscript, for pointing out many errors in it, and for giving many helpful comments.
The author is grateful to Professors Kawaguchi, Moriwaki, Yamaki, and Yoshikawa for giving him an opportunity to talk about the results at Kyoto.

\bibliography{ikoma}
\bibliographystyle{plain}

\end{document}